\documentclass[12pt,leqno]{article}
\usepackage{amsmath, amsthm, amsfonts, amssymb, color}
\usepackage{mathrsfs}
\usepackage[notcite,notref]{showkeys}

\theoremstyle{definition}

\newcommand{\scr}[1]{\mathscr #1}
\definecolor{wco}{rgb}{0.5,0.2,0.3}

\numberwithin{equation}{section}

\newcommand{\ua}{\uparrow}

\title{{\bf Forward-backward stochastic differential equations on tensor fields and  application to Navier-Stokes equations
on Riemannian manifolds}}

\author{
{\bf    Xin Chen$^{a)}$, Ana Bela Cruzeiro$^{b)}$, Wenjie Ye$^{c)}$, Qi Zhang$^{d,e)}$} 
\thanks{E-mail address: chenxin217@sjtu.edu.cn (X. Chen), ana.cruzeiro@tecnico.ulisboa.pt (A.B. Cruzeiro), yewenjie@amss.ac.cn(W. Ye), qzh@fudan.edu.cn (Q. Zhang)}
\\
\footnotesize {$^{a)}$ School of Mathematical Sciences, Shanghai Jiaotong University, Shanghai 200240, China}\\
\footnotesize {$^{b)}$ GFMUL and Dep. Mathematics Univ. Lisbon, Av. Rovisco Pais, 1049-001 Lisboa, Portugal}\\
\footnotesize {$^{c)}$ Academy of Mathematics and Systems Science, Chinese Academy of Sciences, Beijing 100190, China}\\
\footnotesize {$^{d)}$ School of Mathematical Sciences, Fudan University, Shanghai 200433, China}\\
\footnotesize {$^{e)}$ Laboratory of Mathematics for Nonlinear Science, Fudan University, Shanghai 200433, China}\\
}

\begin{document}

\allowdisplaybreaks

\def\R{\mathbb R} \def\Z{\mathbb Z} \def\ff{\frac} \def\ss{\sqrt}
\def\dd{\delta} \def\DD{\Delta} \def\vv{\varepsilon} \def\rr{\rho}
\def\<{\langle} \def\>{\rangle} \def\GG{\Gamma} \def\gg{\gamma}
\def\ll{\lambda} \def\LL{\Lambda} \def\nn{\nabla} \def\pp{\partial}
\def\d{\text{\rm{d}}} \def\loc{\text{\rm{loc}}} \def\bb{\beta} \def\aa{\alpha} \def\D{\mathbb D}
\def\E{\mathbb E}
\def\p{\mathbf P}
\def\h{\mathbf H}
\def\a{\alpha}
\def\O{\Omega}
\newcommand{\Ex}{{\bf E}}
\def\si{\sigma} \def\ess{\text{\rm{ess}}}
\def\beg{\begin} \def\beq{\beg}  \def\F{\scr F}
\def\ric{\text{\rm{Ric}}} \def\Hess{\text{\rm{Hess}}}\def\B{\scr B}
\def\e{\varepsilon} \def\ua{\underline a} \def\OO{\Omega} \def\b{\mathbf b}
\def\oo{\omega}     \def\tt{\tilde} \def\Ric{\text{\rm{Ric}}}
\def\cut{\text{\rm{cut}}}
\def\P{\mathbb P}
\def\ifn{I_n(f^{\bigotimes n})}
\def\fff{f(x_1)\dots f(x_n)} \def\ifm{I_m(g^{\bigotimes m})} \def\ee{\varepsilon}
\def\C{\mathbb C}
\def\CC{\widetilde{\mathbf{C}}}
\def\B{\scr B}
\def\S{\scr S}
\def\M{\scr M}
\def\F{\scr F}
\def\ll{\lambda}
\def\X{\scr X}
\def\T{\mathbb T}
\def\A{\scr A}
\def\LL{\scr L}
\def\gap{\mathbf{gap}}
\def\div{\text{\rm div}}
\def\dist{\text{\rm dist}}
\def\cut{\text{\rm cut}}
\def\supp{\text{\rm supp}}
\def\Var{\text{\rm Var}}
\def\p{\mathbf{P}}
\def\Cov{\text{\rm Cov}}
\def\Cut{\text{\rm Cut}}
\def\le{\leqslant}
\def\ge{\geqslant}
\def\I{\scr I}
\def\coth{\text{\rm coth}}
\def\Dom{\text{\rm Dom}}
\def\Cap{\text{\rm Cap}}
\def\Ent{\text{\rm Ent}}
\def\sect{\text{\rm sect}}\def\H{\mathbb H}
\def\g{\tilde {g}}
\def\o{\omega}
\def\u{\tilde{u}}
\def\c{\tilde{\theta}}\def\w{\tilde{\omega}}
\def\om{\tilde{\Omega}}\def\v{\varepsilon}
\def\U{\tilde{U}}
\def\b{\tilde \beta}
\def\S{\scr S}
\def\F{\scr F}
\def\L{\Lambda}

\newtheorem{theorem}{Theorem}[section]
\newtheorem{lemma}[theorem]{Lemma}
\newtheorem{proposition}[theorem]{Proposition}
\newtheorem{corollary}[theorem]{Corollary}

\theoremstyle{definition}
\newtheorem{definition}[theorem]{Definition}
\newtheorem{example}[theorem]{Example}
\newtheorem{remark}[theorem]{Remark}
\newtheorem{assumption}[theorem]{Assumption}
\newtheorem*{ack}{Acknowledgement}
\renewcommand{\theequation}{\thesection.\arabic{equation}}
\numberwithin{equation}{section}

\maketitle

\rm

\vskip 10mm

\begin{abstract}
In this paper we introduce a class of forward-backward stochastic differential equations on tensor fields
of Riemannian manifolds, which are related to semi-linear parabolic partial differential equations on tensor fields. Moreover, we will
use these forward-backward stochastic differential equations to give a stochastic characterization of incompressible Navier-Stokes equations
on Riemannian manifolds, where some extra conditions used  in \cite{FL} are not required.
\end{abstract}

\vskip 3mm

\section{Introduction}

Bismut \cite{Bi} first studied linear backward stochastic differential equations (written as BSDE for simplicity through this paper) which were originated from some
 stochastic control problems. A breakthrough
was made by Pardoux and Peng \cite{PP1}, who proved existence and uniqueness of  solutions for general non-linear BSDEs
 under uniform Lipschitz continuity conditions on their generators. In \cite{PP2} Pardoux and Peng described an important
connection between forward-backward stochastic differential equations
(written as FBSDE in this paper) and  parabolic semi-linear systems.
Peng \cite{Pe} also
showed the correspondence between infinite horizon FBSDEs and elliptic semi-linear systems.
We refer the
reader to \cite{CCQ,CQ,FQT,FWZ,PT,WY,ZZ} and references therein
for the probabilistic representation of several non-linear PDEs through FBSDEs.

Subsequently, many attempts to extend  BSDE theory from the Euclidean space to manifolds were made. Darling \cite{Da1} introduced
a double convex condition under which the global (in time) existence of unique solutions for BSDEs on manifolds
having only one local chart was established. Xing and Zitkovi\'c  \cite{XZ} proved the existence of unique Markovian
type solutions for manifold-valued FBSDEs based on a single convex condition, where the manifolds were still assumed to possess only
one local chart.
In \cite{B1,B2}, Blache studied
BSDEs on general Riemannian manifolds whose solutions were restricted to only one local chart. Chen and Ye \cite{CY1} gave a
new definition of manifold-valued BSDEs which were not necessarily situated in only one local chart, and
the existence of these manifold-valued solutions was also proved without any convex condition in \cite{CY1}.
Moreover, through such kind of manifold-valued FBSDEs, a probabilistic representation for heat flows of harmonic map associated
with time-changing Riemannian metrics was established by Chen and Ye \cite{CY}. See also Estrade and Pontier \cite{EP} and Chen and Cruzeiro \cite{CC}
on the study on Lie-group valued BSDEs.

As explained in \cite{B1,CY}, manifold-valued FBSDEs mentioned above are related to
a class of semi-linear systems with manifold-valued solutions, such as harmonic maps or heat flows of
harmonic map. So such kind of  manifold-valued FBSDEs can not be used
to study non-linear systems with tensor field-valued solutions on Riemannian manifolds, such as
Ricci flows or Navier-Stokes equations (written as NSE from now on) on manifolds. Following this observation, we
introduce  in this paper a kind of FBSDEs (see e.g. the equation \eqref{e2-7} below) on tensor fields of Riemannian manifolds. Then an equivalence
between \eqref{e2-7} and \eqref{e2-8}, which are
tensor fields-valued semi-linear systems on Riemannian manifolds, will be shown, see Theorem \ref{t2-1} and
Theorem \ref{t2-2} below for the details. Furthermore, using these tensor field-valued FBSDEs, we will also give
a probabilistic representation for NSEs
\eqref{e3-2} on Riemannian manifolds in Theorem \ref{t4-1}, from which
 local existence of unique solutions in Sobolev spaces to \eqref{e3-2} will be proved in Theorem \ref{t4-2}.

We also mention some  results concerning the study of the NSEs on Riemannian manifolds. Arnaudon and Cruzeiro \cite{AC}
characterized solutions of NSEs on Riemannian manifolds as critical points of action functionals induced by
corresponding stochastic Lagrangian paths (on Riemannian manifolds). In \cite{ACF}, Arnaudon, Fang and Cruzeiro studied minimal
points of these action functionals as well as their  equivalences to NSEs on Riemannian manifolds.
Constantin and Iyer \cite{CI} introduced a kind of stochastic functional systems using flows generated
by stochastic Lagrangian paths and used them to give a stochastic representation for NSEs on Euclidean space.
Fang and Luo \cite{FL} made
a crucial observation on the behaviour of Lie derivatives for vector fields associated with stochastic Lagrangian paths and applied it to extend the arguments in \cite{CI} to the case of NSEs on  Riemannian manifolds. See also Fang \cite{F} for the probabilistic
representation for NSEs on Riemannian manifolds associated with some special second order differential operators.
Based on expressions given in \cite{CI,F}, local existence of unique solutions in $C^k$ spaces to NSEs on Riemannian manifolds has been obtained
by Ledesma \cite{Le}. We also cite Cruzeiro and Shamarova \cite{CS} for the study of the relation between solutions of NSEs on torus and those of some related
FBSDEs.

We now highlight some aspects of our results.

\begin{itemize}
\item [(i)] As explained in \cite{CY1}, the main difficulty to deal with  manifold-valued BSDEs is the lack of
linear structure, which is necessary to define  It\^o integrals in equations. For the FBSDEs on tensor fields
studied in this paper, we will use stochastic parallel translations induced by forward equations, which
are SDEs on orthonormal frame bundles over base manifolds, in order to define solutions of backward equations (at different times)
in the same linear space, so that we can define It\^o integrals by the linear structure in such space.
\medskip
\item [(ii)] Stochastic Lagrangian paths have been used
in  \cite{AC,ACF,FL,F}  to study NSEs on Riemannian manifolds. In all these formulations,
there are crucial restrictions on
the coefficients of corresponding stochastic Lagrangian paths, namely that the vector fields $\{A_i\}_{1\le i \le k}$
in the first equation of \eqref{e4-3} must be divergence free, i.e. $\div A_i=0$ for every $1\le i \le k$,
see also a weaker condition \cite[Condition (d), Page 188]{FL}. But for general Riemannian manifolds, it seems
difficult to find divergence free vector fields $\{A_i\}_{1\le i \le k}$ which also satisfy \eqref{e4-1} and
\eqref{e4-2}. Such kind of vector fields can only be constructed for some special Riemannian manifolds; we refer the reader
to \cite[Section 4,5]{FL} for more details. Still, due to these assumptions on $\{A_i\}_{1\le i \le k}$, the stochastic
representation obtained in \cite{FL} is only valid for NSEs on Riemannian manifolds associated with
de Rham-Hodge Laplacian operator. We will remove these restrictions on $\{A_i\}_{1\le i \le k}$ in this paper. Particularly
in Theorem \ref{t4-1} and \ref{t4-2}, the divergence free condition $\div A_i=0$, $1\le i\le k$ is not required. Moreover,
as explained in Remark \ref{r4-1}, by changing the formulation accordingly, our results will be valid for NSEs on Riemannian manifolds
associated with different second order differential operators, not only de Rham-Hodge Laplacian operator.

\item [(iii)] In order to show  existence of solutions for NSEs via  a probabilistic representation
induced by  stochastic Lagrangian paths, suitable gradient estimates for flows generated by
associated SDEs are needed. In the Euclidean space, see e.g. \cite{CCQ,CQ,CI,FQT,Z},
diffusion coefficients for gradients of these flows vanish, since the  corresponding stochastic Lagrangian paths
have additive noises. This property is essential for the proof, so the arguments do not work for stochastic Lagrangian paths
 having multiplicative noises,
including the case of NSEs on Riemannian manifolds. In this paper we will use a technique concerning the
decomposition of noises introduced by Elworthy, LeJan and Li \cite{ELL,EL} to study gradient
flows for  stochastic Lagrangian paths on Riemannian manifolds (see Lemma \ref{l4-3} and Lemma \ref{l4-2}). This approach also
seems to provide a reasonable way to investigate a more general class of differential systems induced by stochastic
Lagrangian paths with multiplicative noises.

\item [(iv)] In this paper we only consider  decoupled FBSDEs (meaning
that forward equations do not depend on  backward equations) on tensor fields. For
coupled FBSDEs on tensor fields some extra efforts are required to prove  existence
of solutions. We will investigate  coupled FBSDEs on tensor fields
in a future project, which will also shed some light on the study of Ricci flows by stochastic methods.
\end{itemize}

\section{The theory of FBSDEs on tensor fields}
\subsection{Preliminary and Notations}
Suppose $(M,g)$ is a $d$-dimensional compact Riemannian manifold without boundary endowed with a Riemannian metric $g$.
Let $\nabla$ and $\mu(dx)$ denote the Levi-Civita connection and Riemannian volume measure associated with $g$, respectively.
For every $x\in M$, let $T_x M$ be the tangent space at $M$ and  denote by $T_x^*M$ the dual space
of $T_x M$. Then, for each $\xi \in T_x^*M$, there exists a unique element $\xi^{\sharp}\in T_x M$ such that
\begin{equation}\label{e2-1}
\langle \xi^{\sharp}, \eta\rangle(x)=(\xi,\eta),\ \ \forall\ \eta\in T_x M,
\end{equation}
where $\langle , \rangle(x)$ denotes the inner product on $T_xM$ induced by $g$ and
$(,)$ represents the standard pair between $T_x M$ and its dual $T_x^* M$. Based on
the expression above, we can define an inner product $\langle , \rangle_{(T^*_x M,g)}$ on $T_x ^* M$ such that
\begin{equation}\label{e2-2}
\langle \xi,\xi\rangle_{(T^*_x M,g)}:=\langle \xi^{\sharp}, \xi^{\sharp}\rangle 
,\ \ \forall\ \xi\in T_x^* M.
\end{equation}
For every non-negative integers $m,n$, we define the bundle of $(m,n)$ tensors
$T^{m,n}M$ as follows
\begin{equation*}
T^{m,n}M:=\bigcup_{x\in M}(T_x M)^{\otimes m}\otimes (T_x^* M)^{\otimes n},
\end{equation*}
where $\otimes m$ denotes the tensor product for $m$ times. For every $k\in \mathbb{N_+}$,
 $\alpha\in [0,1)$, let $C_b^{k+\alpha}\left(\Gamma(T^{m,n}M)\right)$ denote the collection of
all bounded $C^k$ sections of $T^{m,n}M$ whose $k$-th order covariant derivatives are uniformly $\alpha$-H\"older continuous.
(Here when $\alpha=0$, $C_b^{k}\left(\Gamma(T^{m,n}M)\right)$ just represents the collection of
all bounded $C^k$ sections of $T^{m,n}M$). Moreover, we define
\begin{equation*}
\begin{split}
&C_b^{k_1,k_2+\alpha}\left([0,T]; \Gamma(T^{m,n}M)\right):=
C_b^{k_1}\Big([0,T];C_b^{k_2+\alpha}\left(\Gamma(T^{m,n}M)\right)\Big).
\end{split}
\end{equation*}
In particular, when $m=1$, $n=0$, $C_b^k(TM):=C_b^k\left(\Gamma(T^{1,0}M)\right)$ is the collection of
bounded $C^k$ vector fields on $M$.

Let $\F(M)$ be orthonormal frame bundle over $(M,g)$ with  canonical projection
$\pi:\F(M)\rightarrow M$. Through this paper we will fix a standard basis of
$\R^d$ as
$e_i=(0,\cdots, \underbrace{1}_{i\ {\rm th}}, \cdots, 0)$, $1\le i \le d$. So for every $u\in \F(M)$ with
$\pi(u)=x\in M$, $\{E_i(u)\}_{i=1}^d:=\{ue_i\}_{i=1}^d$ is an orthogonal basis of $T_x M$, and
$\{E_i^{\sharp}(u)\}_{i=1}^d$ (which is defined by \eqref{e2-1}) is a dual basis of $T_x^* M$.
Hence, for every $\theta\in T_x^{m,n}M$ and $u\in \F(M)$ with $\pi(u)=x$, we have
a unique expression of $\theta$, namely
\begin{equation*}
\theta=\sum_{1\le i_1,\cdots, i_m\le d}\sum_{1\le j_1,\cdots,j_n\le d}\theta_{j_1\dots j_n}^{i_1\dots i_m}(u)E_{i_1}(u) \otimes \dots
\otimes E_{i_m}(u)\otimes E^{\sharp}_{j_1}(u)\otimes \dots \otimes E_{j_n}^\sharp(u),
\end{equation*}
where $\theta_{j_1\dots j_k}^{i_1\dots i_m}(u)\in \R$ (which depends on $u$). Based on the expression of $\theta$ above we
can define the scalarization $\S(\theta)(u)\in T_0^{m,n}\R^d$ for $\theta\in T_x^{m,n}M$ at $u$ as follows
\begin{equation}\label{e2-4}
\S(\theta)(u):=\sum_{1\le i_1,\cdots, i_m\le d}\sum_{1\le j_1,\cdots,j_n\le d}\theta_{j_1\dots j_n}^{i_1\dots i_m}(u)e_{i_1} \otimes \dots
\otimes e_{i_m}\otimes e^{\sharp}_{j_1}\otimes \dots \otimes e_{j_n}^\sharp,
\end{equation}
where $\{e_i^{\sharp}\}_{i=1}^d$ is the dual basis of $\{e_i\}_{i=1}^d$ on $\R^d$.

For all
$\eta \in T_0^{m,n}\R^d\simeq \R^{(m+n)d}$ with the following expression
\begin{equation*}
\eta=\sum_{1\le i_1,\cdots, i_m\le d}\sum_{1\le j_1,\cdots,j_n\le d}\eta_{j_1\dots j_n}^{i_1\dots i_m}e_{i_1} \otimes \dots
\otimes e_{i_m}\otimes e^{\sharp}_{j_1}\otimes \dots \otimes e_{j_n}^\sharp,
\end{equation*}
we define $u \eta\in T_{\pi(u)}^{m,n}M$ by
\begin{equation}\label{e2-6}
u\eta :=\sum_{1\le i_1,\cdots, i_m\le d}\sum_{1\le j_1,\cdots,j_n\le d}\eta_{j_1\dots j_n}^{i_1\dots i_m}E_{i_1}(u) \otimes \dots
\otimes E_{i_m}(u)\otimes E^{\sharp}_{j_1}(u)\otimes \dots \otimes E_{j_n}^\sharp(u).
\end{equation}
It is easy to see that definitions \eqref{e2-4} and \eqref{e2-6} are independent of the choice of
orthogonal basis on $\R^d$ and that there is a one-to-one correspondence between
$T_x^{m,n}M$ and $T_0^{m,n}\R^d$ such that $\S(u\eta)(u)=\eta$
and $u\big(\S(\theta)(u)\big)=\theta$ for all $\eta\in T_0^{m,n}\R^d$, $\theta \in T^{m,n}_x M$ and
$u\in \F(M)$ with $\pi(u)=x$.

\subsection{FBSDEs on tensor fields}

Given $C^{3+\alpha}$ (time-dependent) vector fields $A_i(t,\cdot)$, $0\le i\le k$, $t\in [0,T]$ on $M$, we consider the following $\F(M)$-valued stochastic differential equation
(we write SDE for simplicity)
\begin{equation}\label{e2-5}
\begin{cases}
& dU_s^{t,u}=\sum_{i=1}^k \h_{A_i(s)}(U_s^{t,u})\circ dB_s^i+\h_{A_0(s)}(U_s^{t,u}) ds,\ \ 0\le t\le s \le T,\\
& U_t^{t,u}=u,\ u\in \F(M),
\end{cases}
\end{equation}
where $\h_{A_i(s)}\in T\F(M)$ denotes the horizontal lift of $A_i(s,\cdot)$ and $B_s=\left(B_s^1,\cdots, B_s^k\right)$ is a standard
$\R^k$-valued Brownian motion on  probability space
$(\Omega, \mathscr{G}, \P)$ with associated filtration $\{\mathscr{G}_t\}_{t\ge 0}$. Set $X_s^{t,x}:=\pi\left(U_s^{t,u}\right)$ for all $0\le t\le s\le T$ and $u\in \F(M)$
satisfying $\pi(u)=x$ (it is well known that $X_s^{t,x}$ is independent of the choice of $u\in \F(M)$ with $\pi(u)=x$, see e.g.
\cite{H} for details).
For every $N>0$ we define
\begin{equation*}
\begin{split}
\mathscr{C}([t,s]; \R^N):=&\Big\{\xi:[t,s]\times \Omega\rightarrow \R^N;\ \xi\ \text{is}\ \mathscr{G}_r-\text{adapted\ for\ all}
\ r\in[t,s],\\
&\ \xi(\cdot,\omega)\ {\rm is\ continuous\ in}\ [t,s]\ {\rm for}\ a.s.\ \omega,\ \text{and}\ \E\left[\sup_{t\le r\le s}|\xi_r|^2\right]<\infty \Big\},
\end{split}
\end{equation*}
\begin{equation*}
\begin{split}
\mathscr{M}([t,s]; \R^N):=&\Big\{\xi:[t,s]\times \Omega\rightarrow \R^N;\ \xi\ \text{is}\ \mathscr{G}_r-\text{adapted\ for\ all}
\ r\in[t,s]\\
&\ \text{and}\ \E\left[\int_t^s|\xi_r|^2 dr\right]<\infty \Big\},
\end{split}
\end{equation*}
and we also set
\begin{equation*}
\begin{split}
\mathscr{C}([t,s]; \F(M)):=&\Big\{\xi:[t,s]\times \Omega\rightarrow \F(M);\ \xi\ \text{is}\ \mathscr{G}_r-\text{adapted\ for\ all}
\ r\in[t,s]\\
&\ \xi(\cdot,\omega)\ {\rm is\ continuous\ in}\ [t,s]\ {\rm for}\ a.s.\ \omega,\ \text{and}\ \E\left[\sup_{t\le r\le s}|\xi_r|_g^2 \right]<\infty\Big\},
\end{split}
\end{equation*}
where $|\xi_r|_g$ represents the standard norm in $\F(M)$ induced by Riemannian metric $g$.

For every $\alpha\in (0,1)$, let
\begin{equation*}
\begin{split}
&\quad C_b^{0,2+\alpha}\Big([0,T]\times M\times T^{m,n}M \times (T^{m,n}M)^{\otimes k}; T^{m,n}M\Big)\\
&:=\Big\{F: [0,T]\times M\times T^{m,n}M \times (T^{m,n}M)^{\otimes k} \rightarrow T^{m,n}M;
\text{the\ first\ two\ order\ covariant\ derivatives}\\
&\text{\ of}\ F\ \text{with\ respect\ to}\
x\in M, v\in T_x^{m,n}M, w\in (T_x^{m,n}M)^{\otimes k},\ \text{are\ bounded}
\text{\ and\ the\ second\ order}\\
&\text{\ covariant\ derivatives\ are}\ \alpha-\text{H\"older\ continuous}\Big\}
\end{split}
\end{equation*}

We will consider the following assumption on (the nonlinear generator) $F$
\begin{assumption}\label{a2-1}
Suppose that $F\in C_b^{0,2+\alpha}\Big([0,T]\times M\times T^{m,n}M \times (T^{m,n}M)^{\otimes k}; T^{m,n}M\Big)$
for some $\alpha\in (0,1)$ and
$F(t,x,v,w)\in T_x^{m,n}M$ for all $t\in [0,T]$, $x\in M$, $v\in T_x^{m,n}M$, $w\in (T_x^{m,n}M)^{\otimes k}$.
\end{assumption}
With the SDE \eqref{e2-5} as forward equation and $F$ satisfying Assumption \ref{a2-1}, we will consider the following  $T^{m,n}_0\R^d\simeq
\R^{(m+n)d}$-valued  FBSDE,
\begin{equation}\label{e2-7}
\begin{cases}
& dU_s^{t,u}=\sum_{i=1}^k \h_{A_i(s)}(U_s^{t,u})\circ dB_s^i+\h_{A_0(s)}(U_s^{t,u}) ds,\ \ 0\le t\le s \le T,\\
& d\tilde Y_s^{t,u}=\sum_{i=1}^k \tilde Z_s^{t,u,i}dB_s^i-\tilde F\left(s, U_s^{t,u}, \tilde Y_s^{t,u}, \tilde Z_s^{t,u}\right)ds,\\
& U_t^{t,u}=u,\ \tilde Y_T^{t,u}=\S(h)(U_T^{t,u}).
\end{cases}
\end{equation}
Here $\S(h)$ is the scalarization (defined by \eqref{e2-4}) of $h\in C_b^{2+\alpha}\left(\Gamma(T^{m,n}M)\right)$,
and $\tilde F:[0,T]\times\F(M)\times T_0^{m,n}\R^d\times (T_0^{m,n}\R^d)^{\otimes k} \to T_0^{m,n}\R^d$ is defined as follows
\begin{equation}\label{e2-7a}
\tilde F\left(t, u, v,w\right):=\S\left(F\left(t,\pi (u), u v, u w\right)\right)(u),\ u\in \F(M),\ v\in T_0^{m,n}\R^d,\
 w\in (T_0^{m,n}\R^d)^{\otimes k},
\end{equation}
where $uv$ and $uw$ are defined by \eqref{e2-6} and
$\S\left(F\left(t,\pi (u), u v, u w\right)\right)(u)$ is the scalarization of $F\left(t,\pi (u), u v, u w\right)$
at $u$, given by \eqref{e2-4}. Since $F\in C_b^{0,2+\alpha}\Big([0,T]\times M\times T^{m,n}M \times (T^{m,n}M)^{\otimes k};$ $ T^{m,n}M\Big)$,
according to the definition of $\tilde F$ there exists a constant $C_1>0$ such that
\begin{equation*}
\begin{split}
&\left|\tilde F\left(t,u,v_1,w_1\right)-\tilde F\left(t,u,v_2,w_2\right)\right|\le C_1\left(|v_1-v_2|+|w_1-w_2|\right),\\
&\ \forall\ t\in [0,T],\ u\in \F(M),\ v_1,v_2\in T_0^{m,n}\R^d,\ w_1,w_2\in (T_0^{m,n}\R^d)^{\otimes k}.
\end{split}
\end{equation*}
Note that we can view the second
equation in \eqref{e2-7} as an $\R^{(m+n)d}$-valued backward stochastic differential equation. By the uniform Lipschitz continuity
 for $\tilde F$ as well as \cite[Theorem 3.1]{PP1}, there exists a unique solution (for every fixed $t\in [0,T]$)
$(\tilde Y_\cdot^{t,u},\tilde Z_\cdot^{t,u})\in \mathscr{C}([t,T];\R^{(m+n)d})\times \mathscr{M}([t,T];\R^{(m+n)dk})$ to the second  and third equation
in \eqref{e2-7}.

\begin{remark}
Although \eqref{e2-7} is a FBSDE with an $\F(M)$-valued forward equation and a $T_0^{m,n}\R^d$-valued backward equation, we could view it
as a $T^{m,n}M$-valued FBSDE in some sense. Using the isometry $U_s^{t,u}:\R^d \to T_{X_s^{t,x}}M$, we can define
$Y_s^{t,x}:=U_s^{t,u}\tilde Y_s^{t,u}\in T_{X_s^{t,x}}^{m,n}M$ for all $u\in \F(M)$ with $\pi(u)=x$ (as explained
in the proof of Theorem \ref{t2-2} below, $Y_s^{t,x}$ is independent of the choice of $u$, so it is well defined).
Moreover, as proved in Theorem \ref{t2-2}, $\theta(t,x):=Y_{t}^{t,x}$, $t\in [0,T]$ is the solution of
a $T^{m,n}M$-valued parabolic partial differential equation.
\end{remark}

We will now present a result concerning the equivalence between  FBSDE
\eqref{e2-7} and the classical solution $\theta\in C_b^{1,2}\left([0,T];\Gamma(T^{m,n}M) \right)$ of the following semi-linear parabolic equation on $T^{m,n}M$,
\begin{equation}\label{e2-8}
\begin{cases}
&\frac{\partial \theta(t,x)}{\partial t}=-\mathscr{L}_t\theta(t,x)
-F\left(t,x,\theta(t,x), \left(\nabla_{A_1(t)}\theta(t,x){\color{red},\cdots ,} \nabla_{A_k(t)}\theta(t,x)\right)\right),\\
&\theta(T,x)=h(x)
\end{cases}
\end{equation}
where
\begin{equation}\label{e2-8a}
\mathscr{L}_t\theta(t,x):=\frac{1}{2}\sum_{i=1}^k \nabla^2\theta(t,x)\big(A_i(t,x),A_i(t,x)\big)+\frac{1}{2}\sum_{i=1}^k
\nabla_{\nabla_{A_i(t)}A_i(t)}\theta(t,x)+\nabla_{A_0(t)}\theta(t,x)
\end{equation}
with $\nabla^2$ denoting the second order covariant derivative on $T^{m,n}M$ induced by the Levi-Civita connection.

\begin{theorem}\label{t2-1}
Suppose that $\theta\in C_b^{1,2}\left([0,T];\Gamma(T^{m,n}M)\right)$ is
a classical solution to \eqref{e2-8}. Then for every fixed $t\in [0,T]$, $s\mapsto \Big(U_s^{t,u}, \S\left(\theta(s)\right)(U_s^{t,u}), \big(
\S\left(\nabla_{A_1(s)}\theta(s)\right)(U_s^{t,u}),$ $\cdots , \S\left(\nabla_{A_k(s)}\theta(s)\right)(U_s^{t,u})\big)\Big)$ is the unique solution of
\eqref{e2-7} in $\mathscr{C}([t,T];\mathscr{F}(M))\times\mathscr{C}([t,T];$ $\R^{(m+n)d})\times \mathscr{M}([t,T];\R^{(m+n)dk})$.
\end{theorem}
\begin{proof}
Suppose that $\theta\in C_b^{1,2}\left([0,T];\Gamma(T^{m,n}M)\right)$ is
a classical solution of \eqref{e2-8} and let $\S(\theta)$ be the scalarization of
$\theta$, which is defined by \eqref{e2-4}. According to \cite[Proposition 2.2.1]{H} we know that
$\S(\theta) \in C_b^2\left([0,T]\times\F(M); T_0^{m,n}\R^d\right)$ and the following property holds
\begin{equation}\label{t2-1-1}
\h_{A_i(t)}\S\left(\theta(t)\right)(u)=\S\left(\nabla_{A_i(t)}\theta(t)\right)(u),\ \frac{\partial }{\partial t}\S(\theta(t))(u)=
\S(\frac{\partial \theta}{\partial t})(t,u),\ 0\le i\le k,
\end{equation}
where $\S(\nabla_{A_i(t)}\theta(t))$ and  $\S(\frac{\partial \theta}{\partial t})$ denote the scalarization of $\nabla_{A_i(t)}\theta(t)$
and $\frac{\partial \theta}{\partial t}$ respectively. Based on this property we
conclude that
\begin{equation}\label{t2-1-2}
\begin{split}
\h_{A_i(t)}\left(\h_{A_i(t)}\S(\theta(t))\right)(u)&=\h_{A_i(t)}\left(\S\left(\nabla_{A_i(t)}\theta(t)\right)\right)(u)
=\S\left(\nabla_{A_i(t)}(\nabla_{A_i(t)}\theta(t))\right)(u)\\
&=\S\left(\nabla^2 \theta(t)(A_i(t),A_i(t))\right)(u)+
\S\left(\nabla_{\nabla_{A_i(t)}A_i(t)} \theta(t)\right)(u).
\end{split}
\end{equation}
Since $U_s^{t,u}$ is the solution of SDE \eqref{e2-5}, applying It\^o's formula we obtain
\begin{align*}
&\quad \S(h)(U_T^{t,u})=\S\left(\theta(T)\right)(U_T^{t,u})\\
&=\S\left(\theta(s)\right)(U_s^{t,u})+\sum_{i=1}^k \int_s^T \h_{A_i(r)}\S\left(\theta(r)\right)(U_r^{t,u})dB_r^i\\
&+\sum_{i=1}^k\frac{1}{2} \int_s^T\h_{A_i(r)}\h_{A_i(r)}\S\left(\theta(r)\right)(U_r^{t,u})dr
+\int_s^T \h_{A_0(r)}\S\left(\theta(r)\right)(U_r^{t,u})dr+\int_s^T \frac{\partial \S(\theta(r))}{\partial r}(U_r^{t,u})dr\\
&=\S\left(\theta(s)\right)(U_s^{t,u})+\sum_{i=1}^k \int_s^T \S\left(\nabla_{A_i(r)}\theta(r)\right)(U_r^{t,u})dB_r^i+
\int_s^T\S\left(\mathscr{L}_r\theta+\frac{\partial \theta(r)}{\partial r}\right)(U_r^{t,u})dr\\
&=\S\left(\theta(s)\right)(U_s^{t,u})+\sum_{i=1}^k \int_s^T \S\left(\nabla_{A_i(r)}\theta(r)\right)(U_r^{t,u})dB_r^i\\
&\quad \quad -\int_s^T \tilde F\left(r,U_r^{t,u},\S\left(\theta(r)\right)(U_r^{t,u}),
\left(\S\left(\nabla_{A_1(r)}\theta(r)\right),\cdots,\S\left(\nabla_{A_k(r)}\theta(r)\right)\right)(U_r^{t,u})\right)dr.
\end{align*}
The third equality above follows from \eqref{t2-1-1} and \eqref{t2-1-2}, and in the last equality we have applied
equation \eqref{e2-8} and the definition of $\tilde F$ in \eqref{e2-7a}.

Let
$\tilde Y_s^{t,u}:=\S\left(\theta(s)\right)(U_s^{t,u})$,
$\tilde Z_s^{t,u}:=\left(\S\left(\nabla_{A_1(s)}\theta(s)\right),\cdots, \S\left(\nabla_{A_k(s)}\theta(s)\right)\right)(U_s^{t,u})$.
From the equality above we see
that $\left(U_\cdot^{t,u}, \tilde Y_\cdot^{t,u}, \tilde Z_\cdot^{t,u}\right)$
is a solution of \eqref{e2-6}. On the other hand,
since $\theta\in C_b^2\left([0,T];\Gamma(T^{m,n}M)\right)$, it is easy to verify that $\left(U_\cdot^{t,u}, \tilde Y_\cdot^{t,u}, \tilde Z_\cdot^{t,u}\right)$
$\in \mathscr{C}([t,T];\mathscr{F}(M))\times\mathscr{C}([t,T];\R^{(m+n)d})\times \mathscr{M}([t,T];\R^{(m+n)dk})$. Uniqueness
of $\left(U_\cdot^{t,u}, \tilde Y_\cdot^{t,u}, \tilde Z_\cdot^{t,u}\right)$ is due to the uniform Lipschitz continuity of $\tilde F$ and
\cite[Theorem 3.1]{PP1}.
\end{proof}

\begin{theorem}\label{t2-2}
Suppose $(U_\cdot^{t,u}, \tilde Y_\cdot^{t,u},\tilde Z_\cdot^{t,u})\in
\mathscr{C}([t,T];\mathscr{F}(M))\times\mathscr{C}([t,T];\R^{(m+n)d})\times$ $ \mathscr{M}([t,T];\R^{(m+n)dk})$
is the unique solution of \eqref{e2-7}. Set
 $\theta(t,x):=u \tilde Y_t^{t,u}$ for any $u\in \F(M)$ satisfying
that $\pi(u)=x$ (in particular, the value of $\theta$ only depends on $x$ and is independent of the choice of $u$). Then
$\theta\in C_b^{1,2}\left([0,T];\Gamma(T^{m,n}M)\right)$ is
a classical solution of \eqref{e2-8}.
\end{theorem}
\begin{proof}
Since $\F(M)$ is a compact Riemannian manifold, by Nash embedding theorem we can find a smooth isometric embedding
$\Phi:\F(M)\to \R^L$. Analogously to (2.5) and (2.7) in \cite{CY1} (or (2.3) in \cite{CY}), we can extend
$\h_{A_i(s)}\in C^{3+\alpha}_b\left(\Gamma(T\F(M))\right)$, $\S(h)\in C_b^{2+\alpha}(\F(M);T_0^{m,n}\R^d)$, $\tilde F\in C^{0,2+\alpha}\Big([0,T]
\times\F(M)\times T_0^{m,n}\R^d \times (T_0^{m,n}\R^d)^{\otimes k};$ $ T_0^{m,n}\R^d\Big)$
(with $F$ satisfying Assumption \ref{a2-1}) to some $\bar \h_{A_i(s)}\in C_b^{3+\alpha}(\R^L;\R^L)$, $\bar h\in
C_b^{2+\alpha}(\R^L; T_0^{m,n}\R^L)$ and $\bar F\in
C^{0,2+\alpha}\Big([0,T]\times\R^L\times T_0^{m,n}\R^L \times (T_0^{m,n}\R^L)^{\otimes k};$ $T_0^{m,n}\R^L\Big)$ such that
\begin{equation}\label{t2-1-3a}
\begin{split}
&|\bar \nabla_u \bar F(t,\bar u,v,w)|\le c_1(1+|v|+|w|),\
\forall\ t\in [0,T],\ \bar u\in \R^L,\ v\in T_0^{m,n}\R^L,\ w\in (T_0^{m,n}\R^L)^{\otimes k},
\end{split}
\end{equation}
\begin{equation}\label{t2-1-3}
|\nabla_v \bar F(t,\bar u,v,w)|+
|\nabla_w \bar F(t,\bar u,v,w)|\le c_2,\
\forall\ t\in [0,T],\ \bar u\in \R^L,\ v\in T_0^{m,n}\R^L,\ w\in (T_0^{m,n}\R^L)^{\otimes k}.
\end{equation}
Here $\bar \nabla_u$ denotes the standard gradient in $\R^L$ (with respect to variable $\bar u$), $\nabla_v$, $\nabla_w$ represent
the standard gradient in $T_0^{m,n}\R^L$ (with respect to variable $v$) and $(T_0^{m,n}\R^L)^{\otimes k}$
(with respect to variable $w$) respectively.

Now we consider the following FBSDE
\begin{equation}\label{t2-1-4}
\begin{cases}
& d\bar U_s^{t,\bar u}=\sum_{i=1}^k \bar \h_{A_i(s)}(\bar U_s^{t,\bar u})\circ dB_s^i+\bar \h_{A_0(s)}(\bar U_s^{t,\bar u}) ds,\ \ 0\le t\le s \le T,\\
& d\bar Y_s^{t,\bar u}=\sum_{i=1}^k \bar Z_s^{t,u,i}dB_s^i-\bar F\left(s,\bar U_s^{t,\bar u}, \bar Y_s^{t,\bar u}, \bar Z_s^{t,\bar u}\right)ds,\\
& \bar U_t^{t,\bar u}=\bar u,\ Y_T^{t,\bar u}=\bar h(\bar U_T^{t,\bar u}),\ \forall\ \bar u \in \R^L.
\end{cases}
\end{equation}
Set $\bar \theta(t,\bar u):=\bar Y_t^{t,\bar u}$. Noticing that \eqref{t2-1-4} is a FBSDE on the Euclidean space
$\R^L\times \R^{(m+n)L}\times \R^{(m+n)kL}$ with coefficients satisfying \eqref{t2-1-3a} and \eqref{t2-1-3}, we can apply
\cite[Theorem 3.1, 3.2]{PP2} to obtain that $\bar \theta \in C_b^{1,2}([0,T]\times\R^L;T_0^{m,n}\R^L)$ is the unique classical solution
of the following parabolic equation on $T_0^{m,n}\R^L$,
\begin{equation}\label{t2-1-5}
\begin{cases}
&\frac{\partial \bar \theta}{\partial t}(t,\bar u)=-\overline{\mathscr{L}}_t\bar \theta(t,\bar u)-\bar
F\left(t,\bar u, \bar \theta(t,\bar u),\left(\bar \nabla_{\bar \h_{A_1(t)}}\bar \theta(t),\cdots,
\bar \nabla_{\bar \h_{A_k(t)}}\bar \theta(t)\right)(\bar u)\right),\\
& \bar \theta(T,\bar u)=\bar h(\bar u),\ \forall\ t\in [0,T],\ \bar u\in \R^L,
\end{cases}
\end{equation}
where
\begin{align*}
\overline{\mathscr{L}}_t\bar \theta(t,\bar u):=&
\frac{1}{2}\sum_{i=1}^k \bar \h_{A_i(t)}\bar \h_{A_i(t)}\bar \theta(t,\cdot)(\bar u)+
\bar \h_{A_0(t)}\bar \theta(t,\cdot)(\bar u).
\end{align*}
In particular, we remark that, although the regularity conditions here (i.e. $\bar h\in C^{2+\alpha}$, $\bar F\in C^{2+\alpha}$)
are different from those
in \cite[Theorem 3.2]{PP2}
(in \cite{PP2} it is required that
$\bar h \in C^3$ and $\bar F \in C^3$), by carefully tracking the proof of \cite[Theorem 2.9]{PP2}, the crucial estimates
(20) and (21) still hold under the present conditions. 

Since $\bar \h_{A_i(s)}$ is an extension of $\h_{A_i(s)}$ (i.e. $\h_{A_i(s)}(u)=\bar \h_{A_i(s)}(u)$ for
every $u\in \F(M)\subset \R^L$), according to \cite[Proposition 1.2.8, Theorem 1.2.9]{H} we have
$\bar U_s^{t,u}=U_s^{t,u}$ a.s. for every $u\in \F(M)\subset \R^L$ and $0\le t \le s\le T$. Therefore
$\bar h(\bar U_T^{t,u})=\S(h)(U_T^{t,u})$ a.s. (since $\bar h$ is an extension of $\S(h)$, we have $\bar h(u)=\S(h)(u)$ for every
$u\in \F(M)$), which implies that $\tilde Y_s^{t,u}=\bar Y_s^{t,u}$,
 a.s. for every $u\in \F(M)\subset \R^L$ and $0\le t \le s\le T$ since the solution of the BSDE (the second
 equation in \eqref{t2-1-4}) is unique.

Set $\tilde \theta(t,u):=\tilde Y_t^{t,u}$ for all $0\le t \le T$ and $u\in \F(M)$. Then
$\tilde \theta\in C_b^{1,2}([0,T]\times \F(M);T_0^{m,n}\R^d)$ and
$\tilde \theta(t,u)=\bar \theta(t,u)$ for every $u\in \F(M)$, $t\in [0,T]$.
On the other hand, if we restrict the variable $\bar u \in \R^L$ to $u\in \F(M)\subset \R^L$, we have
\begin{align*}
\h_{A_i(t)}\tilde \theta(t,\cdot)(u)=\bar \nabla_{\bar \h_{A_i(t)}}\bar \theta(t,\cdot)(u),\ u\in \F(M),\ 1\le i \le k,
\end{align*}
\begin{align*}
\overline{\mathscr{L}}_t\bar \theta(t,u)=\frac{1}{2}\sum_{i=1}^k \h_{A_i(t)} \h_{A_i(t)}\tilde  \theta(t, \cdot)(u)+
\h_{A_0(t)}\tilde \theta(t,\cdot)(u)=:\widetilde{\mathscr{L}}_t\tilde \theta(t,u) ,\ u\in \F(M).
\end{align*}
Combining all the properties above with \eqref{t2-1-5} implies that $\tilde \theta:[0,T]\times \F(M)\to T_0^{m,n}\R^d$ satisfies the following equation
 \begin{equation}\label{t2-1-6}
\begin{cases}
&\frac{\partial \tilde \theta}{\partial t}(t, u)=-\widetilde{\mathscr{L}}_t\tilde \theta(t, u)-\tilde
F\left(t, u, \tilde \theta(t, u),\left(\h_{A_1(t)}\tilde \theta(t),\cdots,
 \h_{A_k(t)}\tilde \theta(t)\right)(u)\right),\\
& \tilde \theta(T,u)=\S(h)(u),\ \forall\ t\in [0,T],\ u\in \F(M).
\end{cases}
\end{equation}

For any fixed $O\in SO(d)$, note that
$U_s^{t,uO}=U_s^{t,u}O$. Based on  this fact and  the following property
\begin{align*}
&\S(h)(uO)=O^{-1}\S(h)(u),\ \tilde F\left(t,uO,O^{-1}v,O^{-1}w\right)
=O^{-1}\tilde F\left(t,u,v,w\right),\\
&\forall\ u\in \F(M),\ v\in T_0^{m,n}\R^d,\ w\in
(T_0^{m,n}\R^d)^{\otimes k},
\end{align*}
it is not difficult to verify that $\tilde Y_s^{t,uO}=O^{-1}\tilde Y_s^{t,u}$ and $\tilde Z_s^{t,uO}=O^{-1}\tilde Z_s^{t,u}$.
This implies that the value of $u\tilde Y_t^{t,u}$ is the same
for all $u\in \F(M)$ satisfying $\pi(u)=x$, hence $\theta(t,x):=u\tilde Y_t^{t,u}$ (for every
$u\in \F(M)$ with $\pi(u)=x$) is well defined. Moreover, by definition of $\theta$ and the property that $\tilde \theta\in
C_b^{1,2}([0,T]\times \F(M);T_0^{m,n}\R^d)$, we have $\theta\in C_b^{1,2}([0,T]; \Gamma(T^{m,n}M))$.

Note that (by definition) $\S(\theta)(t,u)=\tilde Y_t^{t,u}=\tilde \theta(t,u)$
and combining equation \eqref{t2-1-6} with definition of $\tilde F$ and properties \eqref{t2-1-1}, \eqref{t2-1-2} yields
\begin{equation*}
\begin{cases}
&\S\left(\frac{\partial \theta}{\partial t}\right)(t, u)=-\S(\mathscr{L}_t\theta)(t,u)-\S\left(F
\left(t, x, \theta,\left(\nabla_{A_1(t)}\theta(t),\cdots,
 \nabla_{A_k(t)}\theta(t)\right)\right)\right)(u),\\
& \S(\theta)(T,u)=\S(h)(u),\ \forall\ t\in [0,T],\ u\in \F(M).
\end{cases}
\end{equation*}
According to the equation above,  $\theta$ is a classical solution of \eqref{e2-8} and  we have finished the proof.
\end{proof}

\section{Navier-Stokes equations on Riemannian manifolds}

From now we apply the theory of FBSDEs on tensor fields to derive a stochastic representation for the
incompressible Navier-Stokes equations on Riemannian manifolds. As before we always suppose  that $(M,g)$ is a $d$-dimensional compact Riemannian manifold with no boundary endowed with the associated Levi-Civita connection $\nabla$
and Riemannian volume measure $\mu$.

For every $v\in C^2(\Gamma(TM))$, we define the divergence $\div v$ of $v$ as follows
\begin{equation}\label{e3-1}
\div v(x):=\sum_{i=1}^d\left\langle \nabla_{E_i}v, E_i \right\rangle(x),
\end{equation}
where $\{E_i\}_{i=1}^d$ is an (arbitrarily fixed) orthonormal basis of $(T_x M,g)$ (note that
$\div v(x)$ is independent of the choice of $\{E_i\}_{i=1}^d$).

Through this paper, we use $\Delta_g$ and $\Delta$ to denote the Laplace-Beltrami operator (which acts on $C^2(M)$) and
the Bochner Laplacian operator (which acts on $C^2(\Gamma(TM))$) respectively, i.e. we have
\begin{align*}
&\Delta_g f(x):=\mathbf{Tr}(\nabla^2 f)(x)=\sum_{i=1}^d \nabla^2 f(x)\left(E_i,E_i\right)=\div(\nabla f)(x),\ \forall\ f\in C^2(M),\\
&\Delta v(x):=\mathbf{Tr}(\nabla^2 v)(x)=\sum_{i=1}^d \nabla^2 v(x)\left(E_i,E_i\right),\ \forall\ v\in C^2(\Gamma(TM)),
\end{align*}
where $\nabla^2$ means the second order covariant derivative (induced by $\nabla$) and
$\{E_i\}_{i=1}^d$ is an orthonormal basis of $(T_x M,g)$.


\subsection{Sobolev spaces}

For every $1\le p<\infty$, let
\begin{align*}
&L^p(M):=\left\{f:M \to \R; \|f\|_p:=\left(\int_M |f(x)|^p \mu(dx)\right)^{1/p}<\infty\right\},\\
&L^p\left(\Gamma(TM)\right):=\left\{v\in \Gamma(TM); \|v\|_p:=\left(\int_M |v(x)|_{T_x M}^p \mu(dx)\right)^{1/p}<\infty\right\},
\end{align*}
and
\begin{align*}
&L^\infty(M):=\left\{f: M \to \R; \|f\|_\infty:={\rm esssup}_{x\in M}|f(x)|<\infty\right\},\\
&L^\infty\left(\Gamma(TM)\right):=\left\{v\in \Gamma(TM); \|v\|_\infty:={\rm esssup}_{x\in M}|v(x)|_{T_x M}<\infty\right\}.
\end{align*}

Given a vector field $v\in \Gamma(TM)$, we say that the weak (covariant) derivative $\nabla v\in
\Gamma(T^{1,1}M)$ exists if,
for every $w_1\in C^\infty\left(\Gamma(TM)\right)$, we can find  $\nabla_{w_1}v\in \Gamma(TM)$ such that
\begin{align*}
&\quad \int_{M}\left\langle \nabla_{w_1}v, w_2\right\rangle(x)\mu(dx)\\
&=-\int_M
\left\langle v, w_2\right\rangle(x)\div w_1(x)\mu(dx)-\int_M \left\langle \nabla_{w_1}w_2, v\right\rangle(x)\mu(dx),
\ \forall\ w_1,w_2\in C^\infty\left(\Gamma(TM)\right).
\end{align*}
Suppose that $v\in \Gamma(TM)$ has an expression $v(x)=\sum_{1\le i\le d}v_i\left(\psi(x)\right)\partial_i$
under the local charts $(V,\psi)$; then it is easy to verify that $\nabla v$ exists if and only if
$v_i$ are weakly differentiable in $\psi(V)\subset \R^d$.
Analogously, we can also define the $k$-th order weak covariant derivative $\nabla^k f$, $\nabla^k v$ for every
$f:M\to \R$ and $k\ge 2$. Then for all $1\le p<\infty$ and $k\ge 1$, we define
\begin{align*}
W^{k,p}\left(M\right):=&\Bigg\{f:M \to \R;
{\rm weak\ derivative}\ \nabla^i f\ {\rm exists\ for\ all}\ 1\le i\le k,\\
&\|f\|_{k,p}:=\left(\int_M |f(x)|^p \mu(dx)+\sum_{i=1}^k\int_M |\nabla^i f(x)|_{T_x^{0,i}M}^p\mu(dx)\right)^{1/p}<\infty\Bigg\},
\end{align*}
\begin{align*}
W^{k,p}\left(\Gamma(TM)\right):=&\Bigg\{v\in \Gamma(TM);
{\rm weak\ derivative}\ \nabla^i v\ {\rm exists\ for\ all}\ 1\le i\le k,\\
&\|v\|_{k,p}:=\left(\int_M |v(x)|_{T_x M}^p \mu(dx)+\sum_{i=1}^k\int_M |\nabla^i v(x)|_{T_x^{1,i}M}^p\mu(dx)\right)^{1/p}<\infty\Bigg\}.
\end{align*}
In this paper, we usually write $|v(x)|$, $|\nabla^i v(x)|$ for $|v(x)|_{T_x M}$ and $|\nabla^i v(x)|_{T_x^{1,i}M}$
respectively, for simplicity.

Meanwhile, when $p=\infty$ we define
\begin{align*}
W^{k,\infty}\left(M\right):=&\Bigg\{f:M \to \R;
{\rm weak\ derivative}\ \nabla^i f\ {\rm exists\ for\ all}\ 1\le i\le k,\\
&\|f\|_{k,\infty}:={\rm esssup}_{x\in M}|f(x)|+\sum_{i=1}^k {\rm esssup}_{x\in M}|\nabla^i f(x)|<\infty\Bigg\},
\end{align*}
\begin{align*}
W^{k,p}\left(\Gamma(TM)\right):=&\Bigg\{v\in \Gamma(TM);
{\rm weak\ derivative}\ \nabla^i v\ {\rm exists\ for\ all}\ 1\le i\le k,\\
&\|v\|_{k,\infty}:={\rm esssup}_{x\in M}|v(x)|+\sum_{i=1}^k {\rm esssup}_{x\in M}|\nabla^i v(x)|<\infty\Bigg\}.
\end{align*}

For every $1\le p\le \infty$, we set
\begin{align*}
L_0^p(M):=\{f\in L^p(M);\int_M f(x)\mu(dx)=0\},
\end{align*}
$$L^p_{{\rm div}}(\Gamma(TM)):=\left\{v\in L^p(\Gamma(TM));\div v=0\right\}.$$


\subsection{Elliptic regularity estimates for $\Delta_g$}

Since $M$ is compact, it is well known that $\Delta_g: W^{2,p}(M)\cap L^p_0(M)\rightarrow L^p_0(M)$
is an isomorphism.  In particular, for every $f\in L_0^p(M)\cap C(M)$, we have a
probabilistic representation as
\begin{equation*}
\Delta_g^{-1}f(x)=\int_0^\infty\E\left[f(X_t^x)\right]dt,
\end{equation*}
where  $\{X_t^x\}_{t\ge 0}$ is the $M$-valued Brownian motion with initial point $x\in M$.

We believe that the following results on elliptic regularity of $\Delta_g$ are standard when $M$ is compact, although we did not
find them in any reference. For convenience of the reader, we provide here a detailed proof.

\begin{lemma}\label{l3-1}
For every $2\le p<\infty$, there exists a positive constant $C_1$ such that following statements hold
\begin{equation}\label{l3-1-1}
\|\Delta_g^{-1}f\|_p\le C_1\|f\|_p,\ \ \forall\ f\in L_0^p(M),
\end{equation}
\begin{equation}\label{l3-1-2}
\|\Delta_g^{-1}f\|_{2,p}\le C_1\|f\|_{p},\ \ \ \forall\ f\in L_0^p(M),
\end{equation}
\begin{equation}\label{l3-1-3}
\|\Delta_g^{-1}f\|_{3,p}\le C_1\|f\|_{1,p},\ \ \ \forall\ f\in L_0^p(M).
\end{equation}
\end{lemma}
\begin{proof}
By standard approximation arguments, without loss of generality we can assume that $f\in C^1(M)$, so we have
$\Delta_g^{-1}f(x)=\int_0^\infty\E\left[f(X_t^x)\right]dt$, where $\{X_t^x\}_{t\ge 0}$ denotes an $M$-valued Brownian
motion with initial value $x\in M$. We define the Markov semigroup $P_t$ associated with $\{X_t^x\}_{t\ge 0}$
by $P_t f(x):=\E\left[f(X_t^x)\right]$.

Since $M$ is compact, it is well known that the following Poincar\'e inequality
with respect to volume measure holds (see e.g. \cite[Chapter 19]{V} or \cite[Chapter 2]{W}),
\begin{equation*}
\int_M f^2(x)\mu(dx) \le c_1\int_M |\nabla f|^2(x)\mu(dx),\ \ \forall\ f\in C^1(M)\cap L^2_0(M),
\end{equation*}
which implies the $L^2$ exponential ergodicity of $\{P_t\}_{t\ge 0}$ as follows
\begin{equation}\label{l3-1-4}
\|P_t f\|_{2}\le e^{-\frac{t}{c_1}}\|f\|_{2},\ \ \forall\ f\in C^1(M)\cap L^2_0(M).
\end{equation}

On the other hand, since $M$ is compact and $\{P_t\}_{t\ge 0}$ is irreducible,
$\{X_t\}_{t\ge 0}$  is uniformly exponential ergodic and satisfies the following property (see e.g. \cite{DMT})
\begin{equation}\label{l3-1-5}
\|P_t f\|_\infty\le c_2e^{-c_3 t}\|f\|_\infty,\ \ \ \forall\ f\in C^1(M)\cap L^\infty_0(M).
\end{equation}
Based on \eqref{l3-1-4}, \eqref{l3-1-5} and applying the interpolation inequality, it follows that,
for every $2\le p<\infty$, there exist positive constants $c_4$, $c_5$ such that
\begin{equation*}
\|P_t f\|_p \le c_4e^{-c_5 t}\|f\|_p,\ \ \ \forall\ f\in C^1(M)\cap L^p_0(M).
\end{equation*}
Hence we have, for every $2\le p<\infty$,
\begin{align*}
\|\Delta_g^{-1}f\|_p&=
\left\|\int_0^\infty P_t fdt\right\|_p\le
\int_0^\infty\|P_t f\|_p dt\le c_4\|f\|_p \int_0^\infty e^{-c_5 t}dt\le c_6\|f\|_p,
\end{align*}
namely inequality \eqref{l3-1-1}.

Let $w:=\Delta_g^{-1}f$. Since $M$ is compact, we can find a finite open cover $\{V_k\}_{k=1}^m$ of $M$ such that
$V_k\subset\subset U_k$ and $(U_k,\psi_k)$ is  a local chart on $M$, i.e. $\psi_k:U_k \to \psi_k(U_k)\subset \R^d$ is  a smooth
diffeomorphism. Set $w_k(z):=w\left(\psi_k^{-1}(z)\right)$, $f_k(z):=f\left(\psi_k^{-1}(z)\right)$, $\forall z\in \psi_k(V_k)\subset \R^d$,
note that $\Delta_g w=f$, so we have the following expression in local charts $(U_k,\psi_k)$,
\begin{equation}\label{l3-1-6}
\sum_{i,j=1}^d g^{ij}(z)\partial_i \partial_j w_k(z)+\sum_{i,j=1}^d
\frac{\partial_i \left(\sqrt{|g|(z)}g^{ij}(z)\right)}{\sqrt{|g|(z)}}\partial_j w_k(z)=f_k(z), \ \forall\ z\in \psi_k(V_k),
\end{equation}
where $\{g^{ij}\}_{i,j=1}^d$ is the inverse matrix of $\{g_{ij}=g(\partial_i,\partial_j)\}_{i,j=1}^d$ under local
charts $(U_k, \psi_k)$ and $|g|$ denotes the determinant of matrix $\{g_{ij}\}_{i,j=1}^d$.

Based on \eqref{l3-1-6} and applying \cite[Theorem 9.11]{GT} we obtain the following $W^{2,p}$
estimates of $w_k$.
For all $2\le p<\infty$, there exists a $c_7>0$ such that  for every $1\le k \le m$,
\begin{align*}
\sum_{i,j=1}^d\int_{\psi_k(V_k)} \left(
\left|\partial_i w_k(z)\right|^p+\left|\partial_i \partial_j w_k(z)\right|^p\right)dz
&\le c_7\left(\int_{\psi_k(U_k)}|w_k(z)|^p+|f_k(z)|^p\right)dz.
\end{align*}
On the other hand, by \eqref{l3-1-1} we have
\begin{align*}
\int_{\psi_k(U_k)}\left|w_k(z)\right|^pdz&
\le c_8\int_{U_k} |w(x)|^p \mu(dx)\le c_8\int_M |w(x)|^p \mu(dx)\le c_9\|f\|_p^p.
\end{align*}
Combining all above estimates we obtain
\begin{align*}
\|w\|_{2,p}^p&\le \sum_{k=1}^m \int_{V_k}
\left(|w(x)|^p+|\nabla w(x)|^p+|\nabla^2 w(x)|^p\right)\mu(dx)\\
&\le c_{10}\sum_{k=1}^m \sum_{i,j=1}^d\int_{\psi_k(V_k)}
\left(|w_k(x)|^p+|\partial_i w_k(z)|^p+|\partial_i\partial_j w_k(z)|^p\right)dz\\
&\le c_{11}\sum_{k=1}^m \left(\int_{\psi_k(U_k)}|f_k(z)|^pdz+\|f\|_p^p\right)
\le c_{12}\|f\|_p^p.
\end{align*}
So we have proved \eqref{l3-1-2}.

Let $w_k^l(z):=\partial_l w_k(z)$, $\forall z\in \psi_k(U_k)$, $1\le l \le d$. Taking the derivative in \eqref{l3-1-6} yields
\begin{align*}
\sum_{i,j=1}^d g^{ij}(z)\partial_i \partial_j \left(\partial_lw_k\right)(z)+
\sum_{i,j=1}^d
\frac{\partial_i \left(\sqrt{|g|(z)}g^{ij}(z)\right)}{\sqrt{|g|(z)}}\partial_j
\left(\partial_l w_k\right)(z)=H_{k,l}(z),
\ \forall\ z\in \psi_k(V_k),
\end{align*}
where
\begin{align*}
H_{k,l}(z):=\partial_l f_k(z)-\sum_{i,j=1}^d \partial_l g^{ij}(z)\partial_i \partial_j w_k(z)-
\sum_{i,j=1}^d
\partial_l\left(\frac{\partial_i \left(\sqrt{|g|}g^{ij}\right)}{\sqrt{|g|}}\right)(z)\partial_j w_k(z).
\end{align*}
Therefore, still by \cite[Theorem 9.11]{GT}, \eqref{l3-1-2} and the expression of
$H_{k,l}$, we have
\begin{align*}
\int_{\psi_k(V_k)}\left|\partial_i\partial_j\partial_l w_k(z)\right|^pdz
&\le c_{13}\int_{\psi_k(U_k)}\left(|\partial_l w_k(z)|^p+|H_{k,l}(z)|^p\right)dz\\
&\le c_{14}\left(\|f\|_{1,p}^p+\|w\|^p_{2,p}\right)\le c_{15}\|f\|^p_{1,p}.
\end{align*}
Given these estimates and using the same method as above (decomposition of the integral into different
local charts), we verify \eqref{l3-1-3}.
\end{proof}

\begin{corollary}\label{c3-1}
For every $p\ge 2$, there exists a constant $C_2>0$ such that the following estimate holds
\begin{equation}\label{c3-1-1}
\|\nabla^2 f\|_{p}\le C_2 \|\Delta_g f\|_p,\ \ \forall\ f\in L_0^p(M)\cap W^{2,p}(M),
\end{equation}
\end{corollary}
\begin{proof}
Note that for every $f\in L_0^p(M)\cap W^{2,p}(M)$, we have $f=\Delta_g^{-1}\left(\Delta_g f\right)$.
Applying \eqref{l3-1-2} yields
\begin{align*}
\|\nabla^2 f\|_p \le \|f\|_{2,p}=
\|\Delta_g^{-1}\left(\Delta_g f\right)\|_{2,p}\le c_1\|\Delta_g f\|_{p}.
\end{align*}
\end{proof}

Finally, we state also the following result

\begin{lemma}\label{l3-3}
We can define a projection from $\p: W^{2,p}(\Gamma(TM)) \to L^p_{{\rm div}}(\Gamma(TM))\cap
W^{2,p}(\Gamma(TM))$ by
$\p(v):=v-\nabla \Delta_g^{-1}\div v$. Moreover, there exists a constant $C_3>0$  such that following estimates hold
\begin{equation}\label{l3-3-2}
\|\p(v)\|_{1,p}\le C_3\|v\|_{1,p},\ \ \forall\ v\in W^{2,p}(\Gamma(TM)),
\end{equation}
\begin{equation}\label{l3-3-1}
\|\p(v)\|_{2,p}\le C_3\|v\|_{2,p},\ \ \forall\ v\in W^{2,p}(\Gamma(TM)).
\end{equation}
\end{lemma}
\begin{proof}
\eqref{l3-3-2}, \eqref{l3-3-1} are just direct consequences of \eqref{l3-1-2} and \eqref{l3-1-3}, respectively.
\end{proof}

\subsection{Navier-Stokes equations on $M$}

We consider following  Navier-Stokes equations on $M$ in time interval $[0,T]$,
\begin{equation}\label{e3-2}
\begin{cases}
&\frac{\partial v}{\partial t}(t,x)-\nabla_{v(t)} v(t)(x)=-\nu \Delta v(t,x)-\nabla p(t,x),\ \div v(t,x)=0,
(t,x)\in [0,T]\times M,\\
&v(T,x)=v_0(x),
\end{cases}
\end{equation}
where the couple of unknowns $(v,p)$ are such that the velocity $v(t,\cdot)\in \Gamma(TM)$
and the pressure $p(t,\cdot)$ is a scalar-valued function on $M$.

Let $\ric:TM \times TM \to \R$ denote the Ricci curvature tensor on $M$ and, for every $v\in \Gamma(TM)$, define
$\ric^{\sharp}v\in \Gamma(TM)$, which is the (unique) element in $\Gamma(TM)$ such that
\begin{align*}
\left\langle \ric^{\sharp}v,w \right\rangle(x)=\ric_x\left(v,w\right),\ \forall w\in \Gamma(TM).
\end{align*}
As  explained in \cite[Page 587 (17)]{P}, for every $v\in C^2(\Gamma(TM))$ with $\div v=0$, we have
$\div (\Delta v)=\div(\ric^{\sharp}v)$. Then, taking  the divergence of both sides of \eqref{e3-2}, we deduce that the pressure $p(t,\cdot)$
(at least formally) satisfies the following elliptic equation,
\begin{equation}\label{e3-3}
\Delta_g p(t,x)=\div(\nabla_{v(t)} v(t))(x)-\nu\div\left(\ric^\sharp v\right) (t,x),\ \ \forall\ (t,x)\in [0,T]\times M.
\end{equation}
In order to study the regularity estimates for $p(t,x)$ from equation \eqref{e3-3}, we will introduce the following lemma.

\begin{lemma}\label{l3-2}
There exists a constant $C_4>0$ such that for every $v_1,v_2\in C^2(\Gamma(TM))$
\begin{equation}\label{l3-2-0}
\left\|\div\left(\nabla_{v_1} v_2\right)-v_1\left(\div v_2\right)\right\|_{p}\le
C_4\|v_2\|_{1,\infty}  \|v_1\|_{1,p},
\end{equation}
\begin{equation}\label{l3-2-0a}
\left\|\div\left(\nabla_{v_1} v_2\right)-v_1\left(\div v_2\right)\right\|_{p}\le
C_4\|v_1\|_{1,\infty} \|v_2\|_{1,p},
\end{equation}
\begin{equation}\label{l3-2-1}
\left\|\div\left(\nabla_{v_1} v_2\right)-v_1\left(\div v_2\right)\right\|_{1,p}\le
C_4\left(\sum_{l=1}^2\|v_l\|_{1,\infty}\right)
\left(\sum_{l=1}^2\|v_l\|_{2,p}\right).
\end{equation}
Here $v_1\left(\div v_2\right)$ denotes the vector field $v_1$ acting on the function $\div v_2\in C^1(M)$. In particular,
if $\div v_2=0$, we have
\begin{equation*}
\left\|\div\left(\nabla_{v_1} v_2\right)\right\|_{1,p}\le
C_4\left(\sum_{l=1}^2\|v_l\|_{1,\infty}\right)
\left(\sum_{l=1}^2\|v_l\|_{2,p}\right).
\end{equation*}
\end{lemma}
\begin{proof}
As in the proof of Lemma \ref{l3-1}, there exists a finite open cover $\{V_k\}_{k=1}^m$ of $M$ such that
$V_k\subset \subset U_k$ and $(U_k,\psi_k)$ is a local chart on $M$. Moreover, we can find a moving frame
$\{E_i\}_{i=1}^d$ (we omit the index $k$ in $E_i$ for simplicity) on $V_k$ such that $\{E_i(x)\}_{i=1}^d$ is an orthonormal basis of $T_x M$
for every $x\in V_k$ and $E_i\in C^\infty(\Gamma(TV_k))$, $\forall\ 1\le i \le d$.

Let $w(x):=\div\left(\nabla_{v_1} v_2\right)(x)$. For every $x\in V_k$ we have
\begin{align*}
w(x)&=\sum_{i=1}^d \langle \nabla_{E_i}(\nabla_{v_1}v_2), E_i\rangle(x)=
\sum_{i=1}^d\left(E_i\left(\langle \nabla_{v_1} v_2, E_i\rangle\right)(x)-\langle \nabla_{v_1} v_2, \nabla_{E_i}E_i\rangle(x)\right)\\
&=\sum_{i=1}^d\Big(E_i\big(v_1\langle v_2,E_i\rangle-\langle v_2, \nabla_{v_1} E_i\rangle\big)(x)-\langle \nabla_{v_1} v_2, \nabla_{E_i}E_i\rangle(x)\Big)\\
&=\sum_{i=1}^d\Big(v_1E_i\langle v_2, E_i\rangle(x)+[E_i,v_1]\langle v_2, E_i\rangle(x)-E_i\langle v_2, \nabla_{v_1} E_i\rangle(x)
-\langle \nabla_{v_1} v_2, \nabla_{E_i}E_i\rangle(x)\Big)\\
&=:w_{1,k}(x)+w_{2,k}(x)+w_{3,k}(x)+ w_{4,k}(x).
\end{align*}
Here $[E_i,v_1]=E_iv_1-v_1E_i$ denotes the Lie bracket of two vector fields $E_i$ and $v_1$.

Suppose that $v_l$, $l=1,2$, has an expression in $(U_k,\psi_k)$ as
$$v_l(x)=\sum_{i=1}^d v_{l,i}\left(\psi_k(x)\right)\partial_i
=:
\sum_{i=1}^d v_{l,i}^k(z)\partial_i
,\ x\in U_k,\ z=\psi_k(x).$$
We deduce that
\begin{equation*}\label{l3-2-2}
\begin{split}
w_{1,k}(x)&=v_1\left(\sum_{i=1}^d\langle\nabla_{E_i}v_2, E_i\rangle+\sum_{i=1}^d \langle v_2, \nabla_{E_i}E_i\rangle\right)(x)\\
&=v_1\left(\div v_2+\sum_{i=1}^d \langle v_2, \nabla_{E_i}E_i\rangle\right)(x),
\end{split}
\end{equation*}
which immediately implies that
\begin{equation}\label{l3-2-2}
\begin{split}
w_{1,k}(x)-v_1\left(\div v_2\right)(x)=\sum_{i=1}^d v_1\langle v_2, \nabla_{E_i}E_i\rangle(x).
\end{split}
\end{equation}
Set $w_1^k(z):=w_{1,k}\left(\psi_k^{-1}(z)\right)-v_1\left(\div v_2\right)\left(\psi_k^{-1}(z)\right)$,
and $w_i^k(z):=w_{i,k}\left(\psi_k^{-1}(z)\right)$, $\forall\ z\in \psi_k(V_k)$ for $i=2,3,4$.
According to expression \eqref{l3-2-2} we see that,
for every $z\in \psi_k(U_k)$,
\begin{equation}\label{l3-2-3a}
|w_1^k(z)|\le c_1\left(\sum_{i=1}^d\left|v_{1,i}^k(z)\right|\right)\cdot\left(
\sum_{i,j=1}^d \Big(\left|v^k_{2,i}(z) \right|+\left|\partial_jv_{2,i}^k(z)\right|\Big)\right),
\end{equation}
\begin{align*}
\sum_{i=1}^d \left|\partial_i w_1^k(z)\right|
&\le c_1\left(
\sum_{i,j=1}^d \left(\left|v_{1,i}^k(z)\right|+\left|\partial_jv_{1,i}^k(z)\right|\right)\right)\\
&\quad \cdot\left(\sum_{i,j,m=1}^d
\left(\left|v^k_{2,i}(z) \right|+\left|\partial_j v_{2,i}^k(z)\right|+\left|\partial_j\partial_m v_{2,i}^k(z)\right|\right)\right),
\end{align*}
which implies that
\begin{align*}
 &\quad \quad \int_{V_k}\left(\left|w_{1,k}(x)-v_1(\div v_2)(x) \right|^p+\left|\nabla\left(w_{1,k}-v_1(\div v_2)\right)(x)\right|^p\right)\mu(dx)\\
&\le
c_2\int_{\psi_k(V_k)}
\left(|w_1^k(z)|^p+\sum_{i=1}^d \left|\partial_i w_1^k(z)\right|^p\right)dz\\
&\le c_3\left(\sup_{1\le i,j \le d}\sup_{z\in \psi(U_k)}\left(|v_{1,i}^k(z)|^p+|\partial_j v_{1,i}^k(z)|^p\right)\right)\\
&\quad \cdot \left(\sum_{i,j,m=1}^d\int_{\psi_k(V_k)}
\left(|v_{2,i}^k(z)|^p+\left|\partial_j v_{2,i}^k(z)\right|^p+\left|\partial_j\partial_m v_{2,i}^k(z)\right|^p\right)dz\right)\\
&\le c_4\left(\|v_1\|_\infty^p+\|\nabla v_1\|_\infty^p\right)\cdot \left(\int_{V_k}
\left(\left| v_2(x) \right|^p+\left|\nabla v_2(x)\right|^p+\left|\nabla^2 v_2(x)\right|^p\right)\mu(dx)\right)\\
&\le c_5\left(\|v_1\|_\infty^p+\|\nabla v_1\|_\infty^p\right)\cdot\|v_2\|_{2,p}^p.
\end{align*}
By the definition of $w_{2,k}$, $w_{3,k}$, $w_{4,k}$, we can easily verify that
\begin{equation}\label{l3-2-3}
\begin{split}
&\sum_{l=2}^4 |w_l^k(z)|\le c_6\left(\sum_{i,j=1}^d\left(\left|v_{1,i}^k(z)\right|+\left|\partial_j v_{1,i}^k(z)\right|\right)\right)\cdot\left(
\sum_{i,j=1}^d \left(\left|v_{2,i}^k(z)\right|+\left|\partial_jv_{2,i}^k(z)\right|\right)\right),
\end{split}
\end{equation}
\begin{align*}
&\quad \sum_{i=1}^d \sum_{l=2}^4\left|\partial_i w_l^k(z)\right|\\
&\le c_6
\left(\sum_{i,j,m=1}^d \left(\left|v_{1,i}^k(z)\right|+\left|\partial_jv_{1,i}^k(z)\right|+\left|\partial_j \partial_m v_{1,i}^k(z)\right|\right)\right)
\cdot \left(
\sum_{i,j=1}^d \left(\left|v_{2,i}^k(z)\right|+\left|\partial_jv_{2,i}^k(z)\right|\right)\right)\\
&\quad \quad +c_6\left(
\sum_{i,j=1}^d \left(\left|v_{1,i}^k(z)\right|+\left|\partial_jv_{1,i}^k(z)\right|\right)\right)
\cdot \left(\sum_{i,j,m=1}^d \left(\left|v_{2,i}^k(z)\right|+\left|\partial_jv_{2,i}^k(z)\right|+\left|\partial_j \partial_m v_{2,i}^k(z)\right|\right)\right)
\end{align*}

Based on these estimates and following the
same arguments as above we arrive at
\begin{align*}
\int_{V_k}\left(|w_{i,k}(x)|^p+|\nabla w_{i,k}(x)|^p\right)\mu(dx)
\le c_7\left(\sum_{l=1}^2\left(\|v_l\|_\infty^p+\|\nabla v_l\|_\infty^p\right)\right)\left(\sum_{l=1}^2\|v_l\|_{2,p}^p\right),\ \ i=2,3,4.
\end{align*}
Finally, combining all the above estimates yields
\begin{align*}
&\quad \int_M \left(|w(x)-v_1(\div v_2)(x)|^p+|\nabla w(x)-
\nabla (v_1 (\div v_2))(x)|^p\right)\mu(dx)\\
&\le c_8\sum_{k=1}^{m}\int_{V_k}\left(\left| w_{1,k}(x)-v_1(\div v_2)(x) \right|^p+\left|\nabla w_{1,k}(x)-\nabla(v_1(\div v_2))(x) \right|^p\right)\mu(dx)\\
&+ c_8\sum_{k=1}^m \left(\sum_{i=2}^4
\int_{V_k} \left(|w_{i,k}(x)|^p+|\nabla w_{i,k}(x)|^p\right)\mu(dx)\right)\\
&\le c_{9}\left(\sum_{l=1}^2\left(\|v_l\|_\infty^p+\|\nabla v_l\|_\infty^p\right)\right)\left(\sum_{l=1}^2\|v_l\|_{2,p}^p\right)
\end{align*}
and we have finished the proof for \eqref{l3-2-1}. Analogously and according to estimates \eqref{l3-2-3a} and \eqref{l3-2-3},
we can prove \eqref{l3-2-0} and \eqref{l3-2-0a}.
\end{proof}

For every $v\in W^{2,p}(\Gamma(TM))\cap L^p_{{\rm div}}(\Gamma(TM))$, we define $F_v\in W^{2,p}(\Gamma(TM))$ by
\begin{equation}\label{e3-4}
F_v:=\nabla \Delta_g^{-1}\left(\div(\nabla_v v)-\nu \div \left(\ric^\sharp v\right)\right).
\end{equation}

We have the following

\begin{lemma}\label{l3-4}
For every $p>d$, there exists a positive constant $C_5$ (which is independent of $\nu$) such that the following estimates hold.
\begin{equation}\label{l3-4-1}
\|F_v\|_{2,p}\le C_5(\nu+\|v\|_{2,p})\|v\|_{2,p},\ \ \forall\ v\in W^{2,p}(\Gamma(TM))\cap L^p_{{\rm div}}(\Gamma(TM)).
\end{equation}
\begin{equation}\label{l3-4-2}
\|F_{v_1}-F_{v_2}\|_{1,p}\le C_5
\left(\|v_1\|_{2,p}+\|v_2\|_{2,p}+\nu\right)\|v_1-v_2\|_{1,p},\ \ \ \forall\ v_1,v_2\in W^{2,p}(\Gamma(TM))\cap L^p_{{\rm div}}(\Gamma(TM)).
\end{equation}
\end{lemma}
\begin{proof}
Combining \eqref{l3-1-3} with \eqref{l3-2-1} and the fact $\|v_1\|_{1,\infty}\le c_0\|v_1\|_{2,p}$
(which is due to Sobolev embedding theorem as well as the fact $p>d$ and $M$ is compact) we  immediately obtain \eqref{l3-4-1}.

Note that $\div v_1=\div v_2=0$, we have
\begin{align*}
&\quad \|\nabla\Delta_g^{-1}\left(\div \nabla_{v_1}v_1-\div \nabla_{v_2} v_2\right)\|_{1,p}\\
&\le \left\|\nabla \Delta_{g}^{-1}\left(\div \left(\nabla_{v_1-v_2}v_1\right)\right)\right\|_{1,p}+
 \left\|\nabla \Delta_{g}^{-1}\left(\div\left(\nabla_{v_2}(v_1-v_2)\right)\right)\right\|_{1,p}\\
 &\le c_1\left(\|\div \left(\nabla_{v_1-v_2}v_1\right)\|_p+\|\div\left(\nabla_{v_2}(v_1-v_2)\right)\|_p\right)\\
 &\le c_2\left(\|v_1\|_{2,p}+\|v_2\|_{2,p}\right)\|v_1-v_2\|_{1,p},
\end{align*}
where the second inequality follows from \eqref{l3-1-2} and the last step is due to \eqref{l3-2-0}, \eqref{l3-2-0a} and
Sobolev embedding theorem.
Based on such estimate and the definition of $F_v$ we  prove the desired conclusion \eqref{l3-4-2}.
\end{proof}

\section{Representation of Navier-Stokes equations on $M$ through FBSDEs on tensor fields}
In this section we will use FBSDEs on tensor fields introduced in Section 2 to give a
representation of Navier-Stokes equations on $M$.

In a first step we will choose some suitable smooth (time-independent) vector fields $\{A_i\}_{i=1}^k$ for the forward equation \eqref{e2-5}.
According to Nash's embedding theorem, there exist an isometric embedding
$\Phi:M \to \R^k$ for some $k\in \mathbb{N}_+$.
\emph{Through this section we define $A_i(x):=\nabla\left(\langle  \Phi, e_i\rangle_{\R^k}\right)(x)$, $1\le i \le k$}, where
$\{e_i\}_{i=1}^k$ is a standard orthonormal basis of $\R^k$, $\langle , \rangle_{\R^k}$ is
the Euclidean metric in $\R^k$.  It is not difficult to see that $A_i$ is a smooth vector field on $M$
and $A_i(x)=\Pi(x)e_i$ for every $x\in M$, where $\Pi(x):\R^k\to T_x M$ denotes the orthogonal projection
(induced by $g$) from $\R^k$ to $M$. According to \cite[Chapter 1]{EL} or \cite[Example 3.1]{FL} we know that
the following properties hold for $\{A_i\}_{i=1}^k$,
\begin{equation}\label{e4-1}
\sum_{i=1}^k \left(\langle A_i(x), w\rangle\right)^2=|w|^2, \ \ \forall\ w\in T_x M, x\in M.
\end{equation}
\begin{equation}\label{e4-2}
\sum_{i=1}^k\nabla_{A_i}A_i(x)=0,\ \forall\ x\in M.
\end{equation}
Therefore, by \eqref{e4-1} and \eqref{e4-2} we have the following expression for $\mathscr{L}_t$ defined by
\eqref{e2-8a},
\begin{equation}\label{e4-2a}
\mathscr{L}_t v(x)=\nu\Delta v(x)+\nabla_{A_0(t,\cdot)} v(x),\ \ \forall\ v\in C^2(\Gamma(TM)).
\end{equation}
Given  $v_0\in W^{2,p}\left(\Gamma(TM)\right)\cap L^p_{{\rm div}}\left(\Gamma(TM)\right)$ for some $p>d$, let us consider the following forward-backward stochastic differential system,
\begin{equation}\label{e4-3}
\begin{cases}
& dU_s^{t,u}=\sqrt{2\nu}\sum_{i=1}^k \h_{A_i}(U_s^{t,u})\circ dB_s^i-\h_{v(s)}(U_s^{t,u}) ds,\ \ 0\le t\le s \le T,\\
& d\tilde Y_s^{t,u}=\sum_{i=1}^k \tilde Z_s^{t,u,i}dB_s^i-\tilde F_{v(s)}\left(U_s^{t,u}\right)ds,\\
& U_t^{t,u}=u,\ \tilde Y_T^{t,u}=\S(v_0)(U_T^{t,u}),\\
& v(t,x)=u\tilde Y_t^{t,u},\ \ \forall\ u\in \F(M)\ {\rm with}\ \pi(u)=x.
\end{cases}
\end{equation}
Here 
$F_{v(s)}$ is defined
by \eqref{e3-4} for vector fields $v(s,\cdot)$ and $\tilde F_{v(s)}$ is the scalarization for
$F_{v(s)}$ defined by \eqref{e2-7a} (which is in fact independent of $\tilde Y_s^{t,u}$, $\tilde Z_s^{t,u}$
by \eqref{e3-4}).
We call the quadruple $\Big(U_s^{t,u},\tilde Y_s^{t,u},\tilde Z_s^{t,u},$ $ v\Big)$ the solution of \eqref{e4-3} if
$\left(U_\cdot^{t,u},\tilde Y_\cdot^{t,u},\tilde Z_\cdot^{t,u}\right)\in $ $\mathscr{C}([t,T];\mathscr{F}(M))\times\mathscr{C}([t,T];\R^{d})\times \mathscr{M}([t,T];\R^{dk})$.
In particular, as explained in the proof of Theorem \ref{t2-2},
the term $u\tilde Y_t^{t,u}\in T_{\pi(u)} M$ only depends on the value of $\pi(u)$.
Now we will study the equivalence between \eqref{e4-3} and Navier-Stokes equations on $M$.

\begin{theorem}\label{t4-1}
We assume that $v_0\in C_b^{2+\alpha}(\Gamma(TM))$ for some $\alpha\in (0,1)$ and $\div v_0=0$.
\begin{itemize}
\item [(1)] Suppose  that $\Big(U_s^{t,u},\tilde Y_s^{t,u},\tilde Z_s^{t,u},$ $ v\Big)$ is
a solution of forward-backward stochastic differential system \eqref{e4-3} such that
$\left(U_\cdot^{t,u},\tilde Y_\cdot^{t,u},\tilde Z_\cdot^{t,u}\right)\in $ $\mathscr{C}([t,T];\mathscr{F}(M))\times\mathscr{C}([t,T];\R^{d})$ $\times \mathscr{M}([t,T];\R^{dk})$ and
$v\in C_b^{1,2+\alpha}\left([0,T];\Gamma(TM)\right)$; then
$v$ is a solution of  incompressible Navier-Stokes equation \eqref{e3-2} on $M$.

\item [(2)] On the other hand, suppose that $v\in C_b^{1,2+\alpha}\left([0,T];\Gamma(TM)\right)$ is the (classical) solution for
\eqref{e3-2}. Let $U_s^{t,u}$ be the solution of
first equation in \eqref{e4-3} with such $v$ (which is in fact a $\F(M)$-valued SDE) and set
$\tilde Y_s^{t,u}:=\S\left(v(s)\right)\left(U_s^{t,u}\right)$,
$\tilde Z_s^{t,u}:=
\Big(\S\left(\nabla_{A_1}v(s)\right)\left(U_s^{t,u}\right),\cdots, \S\left(\nabla_{A_k}v(s)\right)\left(U_s^{t,u}\right)\Big)$.
Then $\left(U_\cdot^{t,u},\tilde Y_\cdot^{t,u},\tilde Z_\cdot^{t,u}\right)\in $
$\mathscr{C}([t,T];$ $\mathscr{F}(M))\times\mathscr{C}([t,T];\R^{d})$ $\times \mathscr{M}([t,T];\R^{dk})$ and $\left(U_s^{t,u},\tilde Y_s^{t,u},\tilde Z_s^{t,u},v\right)$ is a solution of
\eqref{e4-3}.
\end{itemize}
\end{theorem}
\begin{proof}
Suppose that $\Big(U_s^{t,u},\tilde Y_s^{t,u},\tilde Z_s^{t,u},$ $ v\Big)$ is a solution to \eqref{e4-3}.
Set $G_v(s,x):=\tilde F_{v(T-s)}(x)$ for every $(s,x)\in [0,T]\times M$.
Since $v\in C_b^{1,2+\alpha}\big([0,T]$ $;\Gamma(TM)\big)$, with the same arguments in Lemma \ref{l3-2} and \ref{l3-4} and using the associated estimates
in H\"older spaces we will obtain that $G_v\in C_b^{1,2+\alpha}([0,T];\Gamma(TM))$.

Note that $\Big(U_s^{t,u},\tilde Y_s^{t,u},\tilde Z_s^{t,u},$ $ v\Big)$ is a solution to \eqref{e4-3}, we can view
$\Big(U_s^{t,u},\tilde Y_s^{t,u},\tilde Z_s^{t,u}\Big)$ as the solution of the FBSDE consisting of the
first three equations in \eqref{e4-3} with $v$ already known. Since $v(t,x):=u\tilde Y_t^{t,u}$ with $\pi(u)=x$,
for such FBSDE we can apply Theorem \ref{t2-2} to obtain that $v\in C_b^{1,2+\alpha}([0,T];\Gamma(TM))$ satisfies the following equation
on $TM$,
\begin{equation}\label{t4-1-1}
\begin{cases}
&\frac{\partial v(t,x)}{\partial t}=-\nu \Delta v(t,x)+\nabla_{v(t)}v(t)(x)
-F_{v(t)}(x),\ t\in [0,T],\\
&v(T,x)=v_0(x).
\end{cases}
\end{equation}
Here we have also applied  property \eqref{e4-2a}. Therefore, according to the expression \eqref{e3-4} for
$F_{v(s)}$, to obtain conclusion of (1) it
only remains to prove that $\div v(t,x)=0$ for all $t\in [0,T]$ (note that here we can take
$p=\Delta_g^{-1}\left({\rm div}(\nabla_v v)-\nu{\rm div}({\rm Ric}^{\sharp}v)\right)$ in equation \eqref{e3-2}).

Note that
\begin{equation}\label{t4-1-2}
\begin{split}
\div F_{v(t)}(x)&=\div\nabla \Delta_g^{-1}\left(\div\left(\nabla_{v(t)} v(t)-\nu\ric^{\sharp}v(t)\right)\right)(x)\\
&=\div\left(\nabla_{v(t)} v(t)-\nu\ric^{\sharp}v(t)\right)(x).
\end{split}
\end{equation}
For $v=(v-\nabla \Delta_g^{-1}\div v)+
\nabla \Delta_g^{-1}\div v:=v_1+\nabla \Delta_g^{-1}\div v$, we have
\begin{align*}
\div \Delta v(t,x)&=\div \Delta v_1(t,x)+\div \Delta \nabla \Delta_g^{-1}\div v(t,x)=:w_1(t,x)+w_2(t,x).
\end{align*}
As explained before, since $\div v_1=0$ and according to \cite[Page 587]{P}, we have $w_1(t,x)=\div \ric^{\sharp}v_1$.
On the other hand we have
\begin{align*}
w_2(t,x)&=\div \Delta \nabla \Delta_g^{-1}\div v(t,x)\\
&=\div \nabla \Delta_g \Delta_g^{-1}\div v(t,x)+\div \ric^{\sharp}\left(\nabla \Delta_g^{-1}\div v\right)(t,x)\\
&=\Delta_g \div v(t,x)+\div \ric^{\sharp}\left(\nabla \Delta_g^{-1}\div v\right)(t,x),
\end{align*}
where we used the property  $\Delta \nabla f=\nabla \Delta_g f+\ric^{\sharp}\nabla f$ for every
$f\in C^3(M)$. Hence the following identity holds
\begin{equation}\label{t4-1-3}
\div \Delta v(t,x)=w_1(t,x)+w_2(t,x)=\Delta_g \div v(t,x)+\div \ric^{\sharp} v(t,x).
\end{equation}
Combining all the above properties  and applying the divergence operator to both sides of \eqref{t4-1-1} we obtain the
following equation for $\div v$,
\begin{equation*}
\begin{cases}
&\frac{\partial \div v(t,x)}{\partial t}=-\nu\Delta_g \div v(t,x),\\
&\div v(T,x)=\div v_0(x)=0.
\end{cases}
\end{equation*}
We remark that, although $\Delta_g \div v(t,x)$ may not be defined in the classical sense
(since we only assume that $v\in C_b^{1,2+\alpha}([0,T]\times\Gamma(TM))$), by standard approximation arguments the
linear equation for $\div v$ can be interpreted in a weak sense. Therefore, according to
the uniqueness of linear equations, we  immediately have $\div v(t,x)=0$ for every
$t\in [0,T]$. By now we have finished the proof of the desired conclusion (1).

On the other hand, if $v\in C_b^{1,2+\alpha}([0,T];\Gamma(TM))$, let $U_s^{t,u}$ be the solution of
first equation in \eqref{e4-3} with such $v$ and define
$\tilde Y_s^{t,u}:=\S\left(v(s)\right)\left(U_s^{t,u}\right)$,
$\tilde Z_s^{t,u}:=\Big(\S\left(\nabla_{A_1}v(s)\right)\left(U_s^{t,u}\right),$ $\cdots, \S\left(\nabla_{A_k}v(s)\right)\left(U_s^{t,u}\right)\Big)$.
Based on the fact $v\in C_b^{1,2+\alpha}([0,T];\Gamma(TM))$, it is not difficult to verify that
$\left(U_\cdot^{t,u},\tilde Y_\cdot^{t,u},\tilde Z_\cdot^{t,u}\right)\in $ $\mathscr{C}([t,T];\mathscr{F}(M))\times\mathscr{C}([t,T];\R^{d})$ $\times \mathscr{M}([t,T];\R^{dk})$.
As in the proof of Theorem \ref{t2-1}, applying It\^o's formula to $\S\left(v(s)\right)\left(U_s^{t,u}\right)$ and using  properties \eqref{t2-1-1},
\eqref{t2-1-2}, \eqref{e3-2}, \eqref{e3-3} and \eqref{e3-4} we can prove the conclusion that
$\left(U_s^{t,u},\tilde Y_s^{t,u},\tilde Z_s^{t,u},v\right)$ is a solution of
\eqref{e4-3} directly.
\end{proof}

\begin{remark}\label{r4-1}
We give a stochastic representation for the incompressible Navier-Stokes equation \eqref{e3-2} on $M$ via the  forward-backward
stochastic differential system \eqref{e4-3} in Theorem \ref{t4-1}. In \eqref{e3-2} we use
the Bochner Laplacian operator $\Delta=\text{Tr}(\nabla^2)$. Using our methods the Bochner Laplacian
operator $\Delta$ can be replaced by the Hodge-de Rham Laplacian operator $-\Box:=\Delta-\text{Ric}^\sharp$. By
the proof of Theorem \ref{t4-1} we will obtain that there is a solution of the incompressible Navier-Stokes equation associated with
$\Box$, which can be represented by the forward-backward stochastic
differential system as follows,
\begin{equation*}
\begin{cases}
& dU_s^{t,u}=\sqrt{2\nu}\sum_{i=1}^k \h_{A_i}(U_s^{t,u})\circ dB_s^i-\h_{v(s)}(U_s^{t,u}) ds,\ \ 0\le t\le s \le T,\\
& d\tilde Y_s^{t,u}=\sum_{i=1}^k \tilde Z_s^{t,u,i}dB_s^i-\tilde F_{v(s)}\left(U_s^{t,u}\right)ds+
\nu\S\left(\text{Ric}^\sharp \left(U_s^{t,u}Y_s^{t,u}\right)\right)\left(U_s^{t,u}\right)ds,\\
& U_t^{t,u}=u,\ \tilde Y_T^{t,u}=\S(v_0)(U_T^{t,u}),\\
& v(t,x)=u\tilde Y_t^{t,u},\ \ \forall\ u\in \F(M)\ {\rm with}\ \pi(u)=x.
\end{cases}
\end{equation*}
\end{remark}

Theorem \ref{t4-1} provides a reasonable way to prove the (local) existence of solutions for incompressible Navier-Stokes equations
on $M$ by establishing the existence of solutions of the associated forward-backward stochastic differential systems.
Now we will apply Theorem \ref{t4-1} to prove the local existence of a solution of \eqref{e3-2} in  a Sobolev space.

For this we need some preliminary results. For every $p>1$, $T>0$, define
\begin{equation*}
\begin{split}
\C_T^\infty:=\Big\{w\in C_b^{1,\infty}([0,T];\Gamma(TM)); \div w(t,x)=0\ \text{for\ every}\ (t,x)\in [0,T]\times M \Big\},
\end{split}
\end{equation*}
and
\begin{align*}
C\left([0,T];W^{2,p}_{{\rm div}}(\Gamma(TM))\right):=&\Big\{w:[0,T]\to W^{2,p}(\Gamma(TM))\ \text{is\ continuous},\\
&\text{and}\ \div w(t,x)=0\ \text{for\ every}\ (t,x)\in [0,T]\times M\Big\}
\end{align*}
with the norm
\begin{equation*}
\|w\|_{2,p,T}:=\sup_{t\in [0,T]}\|w(t)\|_{2,p},\ \ \forall\ w\in C\left([0,T];W_{{\rm div}}^{2,p}(\Gamma(TM))\right).
\end{equation*}

\begin{lemma}\label{l4-1}

Given $T>0$ and $w\in \C_T^\infty$, let $\{U_s^{t,u}\}_{0\le t\le s\le T}$ be the solution of following
$\F(M)$-valued SDE,
\begin{equation}\label{l4-1-1}
\begin{cases}
&dU_s^{t,u}=\sqrt{2\nu}\sum_{i=1}^k\h_{A_i}\left(U_s^{t,u}\right)\circ dB_s^i+\h_{w(s)}\left(U_s^{t,u}\right)ds,\\
&U_t^{t,u}=u,
\end{cases}
\end{equation}
and define $X_s^{t,x}:=\pi(U_s^{t,u})$ for every $u\in \F(M)$ with $\pi(u)=x$. Then there exists a constant
$C_6>0$ which is independent of $w$, $\nu$ and $T$, such that, for every non-negative $f\in C^1(M)$, the following estimate holds,
\begin{equation}\label{l4-1-2}
\E\left[\int_M f\left(X_s^{t,x}\right)\mu(dx)\right]\le e^{C_6(1+\nu)T}\|f\|_1,
\end{equation}
\end{lemma}
\begin{proof}
By \eqref{l4-1-1} we know that $\{X_s^{t,x}\}_{0\le t\le s\le T}$ satisfies the following $M$-valued SDE,
\begin{equation*}
\begin{cases}
& dX_s^{t,x}=\sqrt{2\nu}\sum_{i=1}^k A_i\left(X_s^{t,x}\right)\circ dB_s^i+w\left(s,X_s^{t,x}\right)ds,\\
& X_t^{t,x}=x.
\end{cases}
\end{equation*}
So according to \cite[Proposition 4.4]{Z} we have, for all $0\le t\le s\le T$,
\begin{align*}
&\quad \int_M f(X_s^{t,x})\mu(dx)\\
&=\int_M \exp\left(\sqrt{2\nu}\sum_{i=1}^k \int_t^s \div A_i(X_s^{t,x})\circ dB^i_s+
\int_t^s \div w(s,X_s^{t,x})ds\right)f(x)\mu(dx)\\
&=\int_M \exp\left(\sqrt{2\nu}\sum_{i=1}^k \int_t^s \div A_i(X_s^{t,x})\circ dB^i_s
\right)f(x)\mu(dx)\ \ \text{a.s.},
\end{align*}
where the last step we have used the fact that $\div w(t,x)=0$ for all $t\in [0,T]$.
Hence taking the expectation in both sides of the equality above we derive
\eqref{l4-1-2}.
\end{proof}

\begin{lemma}\label{l4-3}
Given $T>0$ and $w\in \C_T^\infty$, let $\{U_s^{t,u}\}$ and $X_s^{t,x}:=\pi(U_s^{t,u})$ be as in Lemma \ref{l4-1}.
Define
$R_s^{t,x,i}:=DX_s^{t,\cdot}(A_i(x))$, where
$DX_s^{t,\cdot}: TM \to TM$ represents the tangent map for
$X_s^{t,\cdot}:M \to M$. Then, for every $q\ge 2$, there exists a constant $C_7>0$ independent of $\nu$ such that
for every $0\le t \le s\le T$ and $1\le i \le k$,
\begin{equation}\label{l4-3-1}
\E\left[\left|R_s^{t,x,i}\right|^q|\mathscr{G}_T^{t,u}\right]\le \exp\left(C_7\left(1+\nu+\sup_{s\in [0,T]}\|w(s)\|_{1,\infty}\right)T\right),
\end{equation}
where $\mathscr{G}_T^{t,u}:=\sigma\{U_s^{t,u};\ t\le s \le T\}$.
\end{lemma}
\begin{proof}
By \cite[Theorem 3.1]{ELL} and the first equality in the proof of
\cite[Theorem 3.2]{ELL} (for our choice of $\{A_i\}_{i=1}^k$, both of the L-W connection $\breve{\nabla}$
and its dual connection $\widehat{\nabla}$ are exactly  the Levi-Civita connection, see also \cite[Chapter 3]{EL}), there exists
a $\mathscr{G}_T^{t,u}$-measurable metric adapted translation $//_s^{t,x}:\R^k \to \R^k$
and an
$\R^{k}$-valued Brownian motion
$\left\{\beta_s=\left(\beta_1^1,\cdots,\beta_s^k\right); s\in [t,T]\right\}$ which is independent of $\mathscr{G}_T^{t,u}$
such that
\begin{equation}\label{l4-3-1a}
\begin{cases}
&\mathbb{D}_s R_s^{t,x,i}=\sum_{m=1}^k \sqrt{2\nu}\nabla_{R_s^{t,x,i}}A_m\left(X_s^{t,x}\right) d\tilde{\beta}^m_s-
\nu\text{Ric}^\sharp_{X_s^{t,x}}R_s^{t,x,i}ds+\nabla_{R_s^{t,x,i}}w\left(s,X_s^{t,x}\right)ds,\\
& R_t^{t,x,i}=A_i(x),
\end{cases}
\end{equation}
where $\mathbb{D}_s$ denotes the stochastic covariant differential on $TM$ and $d\tilde{\beta}_s:=//_s^{t,x}d\beta_s$.
Since $//_s^{t,x}:\R^k \to \R^k$ is metric adapted, by Levy's characterization $\{\tilde{\beta}_s\}_{s\in[0,T]}$ is still
an $\mathbb{R}^k$-valued Brownian motion.
Hence applying It\^o's formula we obtain,
for every $q\ge 2$,
\begin{equation}\label{l4-3-2}
\begin{split}
d\left|R_s^{t,x,i}\right|^q= q\sum_{m=1}^k M_s^{t,m} \left|R_s^{t,x,i}\right|^q d\tilde{\beta}_s^m+
\left(\frac{q}{2}N_s^t+\frac{q(q-2)}{2}\sum_{m=1}^k |M_s^{t,m}|^2\right) \left|R_s^{t,x,i}\right|^q ds,
\end{split}
\end{equation}
where
\begin{align*}
& M_s^{t,m}:=\sqrt{2\nu}\frac{\left\langle \nabla_{R_s^{t,x,i}}A_m\left(X_s^{t,x}\right),
R_s^{t,x,i}\right\rangle}{|R_s^{t,x,i}|^2},\\
& N_s^t=2\nu\sum_{m=1}^k \frac{\left|\nabla_{R_s^{t,x,i}}A_m\left(X_s^{t,x}\right)\right|^2}{|R_s^{t,x,i}|^2}-
\frac{2\nu\left\langle \text{Ric}^\sharp_{X_s^{t,x}}R_s^{t,x,i}, R_s^{t,x,i} \right\rangle}{|R_s^{t,x,i}|^2}+
\frac{2\left\langle \nabla_{R_s^{t,x,i}}w\left(s,X_s^{t,x}\right),
R_s^{t,x,i}\right\rangle}{|R_s^{t,x,i}|^2}.
\end{align*}
Therefore, according to the  linear SDE \eqref{l4-3-2}, we have
\begin{align*}
&\quad |R_s^{t,x,i}|^q\\
&=\left|A_i(x)\right|^q
\exp\left(\int_t^s \sum_{m=1}^k \left(qM_r^{t,m}d\tilde{\beta}_r^m-\frac{q^2|M_r^{t,m}|^2}{2}dr\right)
+\int_t^s\left(\frac{qN_r^t}{2}+\frac{q(q-2)}{2}\sum_{m=1}^k |M_r^{t,m}|^2\right)dr\right),
\end{align*}
which implies that
\begin{align*}
&\quad \E\left[|R_s^{t,x,i}|^q|\mathscr{G}_T^{t,u}\right]\\
&=\left|A_i(x)\right|^q\E^{\beta}\left[
\exp\left(\int_t^s \sum_{m=1}^k \left(qM_r^{t,m}d\tilde{\beta}_r^m-\frac{q^2|M_r^{t,m}|^2}{2}dr\right)
+\int_t^s\left(\frac{qN_r^t}{2}+\frac{q(q-2)}{2}\sum_{m=1}^k |M_r^{t,m}|^2\right)dr\right)\right]\\
&\le \exp\left(c_1(1+\nu+\sup_{s\in [0,T]}\|w(s)\|_{1,\infty})T\right)\E^{\beta}\left[\exp\left(\int_s^t\sum_{m=1}^k \left(
q M_r^{t,m}d\tilde{\beta}_r^m-
\frac{q^2|M_r^{t,m}|^2}{2}dr\right)\right)\right]\\
&\le  \exp\left(c_2(1+\nu+\sup_{s\in [0,T]}\| w(s)\|_{1,\infty})T\right),
\end{align*}
where $\E^{\beta}$ denotes expectation with respect to the randomness of $\{\beta_s\}_{s\in [0,T]}$, but keeping
the randomness from all $\mathscr{G}_T^{t,u}$-measurable random variables fixed, and the first inequality is due to the following estimate
\begin{align*}
| N_r^t|\le c_1(1+\nu+\sup_{s\in [0,T]}\| w(s)\|_{1,\infty}).
\end{align*}
By now we have finished the proof of \eqref{l4-3-1}.
\end{proof}
\begin{lemma}\label{l4-2}
For $T>0$ and $w\in \C_T^\infty$, let $\{U_s^{t,u}\}_{0\le t\le s\le T}$ be the solution of the
$\F(M)$-valued SDE \eqref{l4-1-1}.
Then for all $p>d$, there exists a positive constant
$C_8$ which is independent of $w$, $\nu$ and $T$, such that, for every $h\in C^2_b\left(\Gamma(TM)\right)$ and
$1\le i,j\le k$, the following estimates hold,
\begin{equation}\label{l4-2-1}
\begin{split}
&\sup_{0\le t \le s \le T}\E\left[\int_M \left|\h_{A_i}\left(\S(h)\left(U_s^{t,\cdot}\right)\right)(u(x))\right|^p\mu(dx)\right]
\le C_8\|h\|^p_{1,p}\Bigl( e^{C_8\mathbf{C}(p,T)} + e^{C_8\widetilde{\mathbf{C}}(T)}\Bigr),
\end{split}
\end{equation}
\begin{equation}\label{l4-2-2}
\begin{split}
&\quad \sup_{0\le t \le s \le T}\E\left[\int_M \left|\h_{A_i}\h_{A_j}\left(\S(h)\left(U_s^{t,\cdot}\right)\right)(u(x))\right|^p\mu(dx)\right]\\
&\le C_{8}\|h\|^p_{2,p}\Big(e^{C_8\widetilde{\mathbf{C}}(T)}+e^{C_8\mathbf{C}(2p,T)}+\mathbf{C}(p,T)e^{C_8\big(\mathbf{C}(p,T)+\mathbf{C}(2p,T)\big)}+K_T(w)^pT^p  e^{C_8\big(\widetilde{\mathbf{C}}(T)+\mathbf{C}(p,T)\big)}\Big),
\end{split}
\end{equation}
where $x\mapsto u(x)$ is a continuous map from $M$ to $\F(M)$ satisfying  $\pi(u(x))=x$,
$\h_{A_i}\left(\S(h)\left(U_s^{t,\cdot}\right)\right)(u(x))$ denotes the vector fields $\h_{A_i}$
(on $\F(M)$) acting on the $T^{1,0}_0\R^d$-valued function $u \mapsto \S(h)\left(U_s^{t,\cdot}\right)$
at point $u(x)\in \mathscr{F}(M)$, and positive constants $\mathbf{C}(p,T)$, $\widetilde{\mathbf{C}}(T)$ are defined by
\begin{equation}\label{l4-2-0}
\begin{aligned}
& \mathbf{C}(p,T)=\mathbf{C}(p,\nu,T,K_T(w))=\Bigl( \nu^{\frac{p}{2}}T^{\frac{p}{2}} +\nu^{p}T^p+K_T(w)^pT^p  \Bigr),\\
& \widetilde{\mathbf{C}}(T)=\widetilde{\mathbf{C}}(\nu,T,K_T(w))=\Bigl( 1+\nu+K_T(w) \Bigr)T
\end{aligned}
\end{equation}
with $K_T(w):=\sup_{s\in [0,T]}\|w(s)\|_{2,p}$.
\end{lemma}
\begin{proof}
For simplicity we will only prove \eqref{l4-2-2}, the proof for \eqref{l4-2-1} is similar and simpler.

Recall that we define $X_s^{t,x}:=\pi\left(U_s^{t,u}\right)$ for every $u\in \F(M)$ with $\pi(u)=x$.
As explained in the proof of Theorem \ref{t2-2}, we can extend the vector field $\h_{A_i}\in T\F(M)$ as well as the function
$\S(h):\F(M)\to T_0^{1,0}\R$ to ambient Euclidean space $\R^L$ (for simplicity of notation we still use $\h_{A_i}$ and $\S(h)$ to
denote associated terms extended to $\R^L$)
and view $U_{s}^{t,u}$ as a solution of an $\R^L$-valued SDE which is still contained in $\F(M)\subset \R^L$.

Let $V_s^{t,u,i}:=DU_s^{t,\cdot}(\h_{A_i}(u))$, $W_s^{t,u,i,j}:=D^2U_s^{t,\cdot}\big(\h_{A_i}(u),\h_{A_j}(u)\big)$,
$R_s^{t,x,i}:=DX_s^{t,\cdot}(A_i(x))$, where
$DU_s^{t,\cdot}: T\R^L\to T\R^L$ and $D^2U_s^{t,\cdot}: T\R^L\times T\R^L\to T\R^L$
denote the first and second order
differential  for
$U_s^{t,\cdot}:\R^L\to \R^L$ respectively
(note that here for simplicity of calculation we use the standard differential in ambient space $\R^L$, not the covariant differential on $\F(M)$).
We have for every $u\in T\F(M)$,
\begin{equation}\label{l4-2-1a}
\h_{A_i}\left(\S(h)\left(U_s^{t,\cdot}\right)\right)(u)=DU_s^{t,\cdot}(\h_{A_i})\left(\S(h)\right)(U_s^{t,u})=D\S(h)(U_s^{t,u})\left(V_s^{t,u,i}\right),
\end{equation}
\begin{equation}\label{l4-2-2a}
\begin{split}
& \quad \left|\h_{A_i}\h_{A_j}\left(\S(h)\left(U_s^{t,\cdot}\right)\right)(u)\right|=\left|
DU_s^{t,\cdot}(\h_{A_i})\Big(DU_s^{t,\cdot}(\h_{A_j})\left(\S(h)\right)\Big)(U_s^{t,u})\right|\\
&=\left|D^2 \S(h)(U_s^{t,u})\left(V_s^{t,u,i}, V_s^{t,u,j}\right)+
D\S(h)(U_s^{t,u})\left(W_s^{t,u,i,j}\right)\right|\\
&\le c_1
\Big(|V_s^{t,u,i}|\cdot|V_s^{t,u,j}|+|R_s^{t,\pi(u),i}|\cdot|V_s^{t,u,j}|+
|R_s^{t,\pi(u),j}|\cdot|V_s^{t,u,i}|+|W_s^{t,u,i,j}|\Big)\\
&\cdot\Big(\left|h\left(X_s^{t,\pi(u)}\right)\right|+\left|\nabla h\left(X_s^{t,\pi(u)}\right)\right|\Big) +c_2|R_s^{t,\pi(u),i}|\cdot |R_s^{t,\pi(u),j}| \cdot \left|\nabla^2 h
\left(X_s^{t,\pi(u)}\right)\right|\\
&\le c_3\Big(|V_s^{t,u,i}|\cdot|V_s^{t,u,j}|+|R_s^{t,\pi(u),i}|\cdot|V_s^{t,u,j}|+
|R_s^{t,\pi(u),j}|\cdot|V_s^{t,u,i}|+|W_s^{t,u,i,j}|\Big)\cdot \|h\|_{2,p}\\
&\quad +c_2|R_s^{t,\pi(u),i}|\cdot |R_s^{t,\pi(u),j}| \cdot \left|\nabla^2 h\left(X_s^{t,\pi(u)}\right)\right|.
\end{split}
\end{equation}
In the first inequality above we have applied
the estimate (which can be verified in each local chart)
\begin{align*}
&\quad \left|D^2 \S(h)\left(V_s^{t,u,i}, V_s^{t,u,j}\right)(U_s^{t,u})\right|\\
&\le c_4|R_s^{t,\pi(u),i}|\cdot |R_s^{t,\pi(u),j}| \cdot \left|\nabla^2 h\left(X_s^{t,\pi(u)}\right)\right|+
c_4|V_s^{t,u,i}|\cdot |V_s^{t,u,j}| \cdot \left|h\left(X_s^{t,\pi(u)}\right)\right|\\
&+c_4
\left(|R_s^{t,\pi(u),i}|\cdot |V_s^{t,u,j}| +|R_s^{t,\pi(u),j}|\cdot |V_s^{t,u,i}|\right)\left|\nabla h\left(X_s^{t,\pi(u)}\right)\right|,
\end{align*}
and the last step is due to Sobolev embedding theorem.

Actually, it is well known that $V_s^{t,u,i}:=DU_s^{t,\cdot}(\h_{A_i}(u))$ and $W_s^{t,u,i,j}:=
D^2U_s^{t,\cdot}\left(\h_{A_i}(u),\h_{A_j}(u)\right)$ satisfy the linear SDEs  (which could be viewed as equations in $\R^L$) as follows
(see e.g. \cite[Theorem E8]{E1} or \cite{L}),
\begin{equation}\label{l4-2-3}
\begin{cases}
&dV_s^{t,u,i}=\sqrt{2\nu}\sum_{m=1}^k D\h_{A_m}\left(U_s^{t,u}\right)V_s^{t,u,i}\circ dB_s^m+D\h_{w(s,\cdot)}
\left(U_s^{t,u}\right)V_s^{t,u,i}ds\\
& V_t^{t,u,i}=\h_{A_i}(u).
\end{cases}
\end{equation}
\begin{equation}\label{l4-2-4}
\begin{cases}
&dW_s^{t,u,i,j}=\sqrt{2\nu}\sum_{m=1}^k \left(D\h_{A_m}\left(U_s^{t,u}\right)W_s^{t,u,i,j}+
D^2 \h_{A_m}\left(U_s^{t,u}\right)\left(V_s^{t,u,i},V_s^{t,u,j}\right)\right)\circ dB_s^m\\
&\quad\quad\quad +\left(D\h_{w(s,\cdot)}\left(U_s^{t,u}\right)W_s^{t,u,i,j}+
D^2 \h_{w(s,\cdot)}\left(U_s^{t,u}\right)\left(V_s^{t,u,i},V_s^{t,u,j}\right)\right)ds\\
& W_t^{t,u,i,j}=0.
\end{cases}
\end{equation}

Note that $M$ is compact (so there exists a finite number of local charts), combining the estimates in all
different local charts it is not difficult to verify that
\begin{align*}
&\sup_{u\in \F(M)}\left(|\mathbf{H}_{A_m}|+|D\h_{A_m}(u)|+|D^2 \h_{A_m}(u)|+|D^3 \h_{A_m}(u)|\right)\le c_1,
\end{align*}
and, for every $u\in \F(M)$,
\begin{align*}
&|D \h_{w(s)}(u)|\le
c_1 \left(|w(s,\pi(u))|+|\nabla w(s,\pi(u))|\right),\\
& |D^2 \h_{w(s)}(u)|\le c_1 \left(|w(s,\pi(u))|+|\nabla w(s,\pi(u))|+|\nabla^2 w(s,\pi(u))|\right).
\end{align*}
Applying these estimates to the linear SDE \eqref{l4-2-3} and using Grownwall's lemma, we obtain that,
for every $q\ge 2$, $0\le t \le s\le T$,
\begin{equation}\label{l4-2-5}
\begin{split}
\sup_{u\in \F(M)}\E\left[\left|V_s^{t,u,i}\right|^q\right]&\le c_2\exp\Big(c_4\big(\nu^{\frac{q}{2}}T^{\frac{q}{2}}+\nu^qT^q +  \sup_{s\in[0,T]}\|D\mathbf{H}_{w(s)} \|^q_\infty T^q\big)\Big)\\
&\le  c_2\exp\Big(c_4\big(\nu^{\frac{q}{2}}T^{\frac{q}{2}}+\nu^qT^q +K_T(w)^qT^q\big)\Big)
=c_2e^{c_4\mathbf{C}(q,T)}
\end{split}
\end{equation}
where in the second step we have applied  the Sobolev embedding theorem since $p>d$.
On the other hand, by equation \eqref{l4-2-4} for any $q\ge 2,\  0\le t\le s\le T$, we have
\begin{align*}
&\quad \E\left[\left|W_s^{t,u,i,j}\right|^q\right]\\
&\le
c_5\Bigg(\nu^{\frac{q}{2}}\E\left[\left|\sum_{m=1}^k\int_t^s D\h_{A_m}\left(U_r^{t,u}\right)W_r^{t,u,i,j}\circ dB^m_r\right|^q\right]
+\E\left[\left|\int_t^s D\h_{w(r)}\left(U_r^{t,u}\right)W_r^{t,u,i,j} dr\right|^q\right]\\
&+\nu^{\frac{q}{2}}\E\left[\left|\int_t^s \sum_{m=1}^k D^2 \h_{A_m}\left(U_r^{t,u}\right)\left(V_r^{t,u,i},V_r^{t,u,j}\right)\circ dB^m_r\right|^q\right]
+\E\left[\left|\int_t^s D^2 \h_{w(r)}\left(U_r^{t,u}\right)\left(V_r^{t,u,i},V_r^{t,u,j}\right)dr\right|^q\right]\Bigg)\\
&\le c_6\left( \nu^{\frac{q}{2}}T^{\frac{q}{2}-1}
+T^{q-1}\sup_{s\in[0,T]}\|D\mathbf{H}_{w(s)}\|^q_\infty \right)
\int_t^s\E\left[|W_r^{t,u,i,j}|^q dr\right]\\
&+c_6\left(\nu^{\frac{q}{2}}T^{\frac{q}{2}-1}
+K_T(w)^q T^{q-1}\right)\int_0^T\E\left[\left|V_r^{t,u,i}\right|^{2q}\right]^{1/2}\E\left[\left|V_r^{t,u,j}\right|^{2q}\right]^{1/2}dr\\
&+c_6T^{q-1}\int^T_0 \mathbb{E}\left[ | \nabla^2 w(r,X^{t,\pi(u)}_r) |^q \mathbb{E}\left[ |R^{t,\pi(u),i}_r|^{2q}|\mathscr{G}^{t,u}_T \right]^\frac{1}{2}  \mathbb{E}\left[ |R^{t,\pi(u),j}_r|^{2q} |\mathscr{G}^{t,u}_T\right]^\frac{1}{2} \right]\,dr\\
&\le\frac{c_6\mathbf{C}(q,T)}{T}\int_t^s\E\left[|W_r^{t,u,i,j}|^q\right]dr+c_7\mathbf{C}(q,T)e^{c_7\mathbf{C}(2q,T)}
+c_8e^{c_7\widetilde{\mathbf{C}}(T)}T^{q-1}\int^T_0 \mathbb{E}\left[ |\nabla^2 w(r,X_r^{t,\pi(u)})|^q\,dr \right]
\end{align*}
Here the second inequality follows from
Burkholder-Davis-Gundy's inequality, H\"older's inequality, as well as the following estimate
\begin{equation}
    \begin{aligned}
         & \mathbb{E}\left[\left|\int^s_t D^2\mathbf{H}_{w(r)}(U^{t,u}_r)(V^{t,u,i}_r,V^{t,u,j}_r)\,dr\right|^q\right]\\
         \le& \ c_9T^{q-1}\int^s_t\mathbb{E}\Big[\Big(|w(r,X_r^{t,\pi(u)})|+|\nabla w(r,X_r^{t,\pi(u)})|\Big)^q|V^{t,u,i}_r|^q|V^{t,u,j}_r|^q\\
         &\quad \quad \quad +\ |\nabla^2 w(r,X_r^{t,\pi(u)})|^q|R^{t,\pi(u),i}_r|^q|R^{t,\pi(u),j}_r|^q \Big]\,dr\\
         \le& c_{10}T^{q-1} K_T(w)^q\int^T_0 \E\left[\left|V_r^{t,u,i}\right|^{2q}\right]^{1/2}\E\left[\left|V_r^{t,u,j}\right|^{2q}\right]^{1/2}\,dr\\
         +& c_{10}T^{q-1}\int^T_0 \mathbb{E}\left[ |\nabla^2 w(r,X_r^{t,\pi(u)})|^q\mathbb{E}\left[ |R^{t,\pi(u),i}_r|^{2q}|\mathscr{G}^{t,u}_T \right]^\frac{1}{2}  \mathbb{E}\left[ |R^{t,\pi(u),j}_r|^{2q} |\mathscr{G}^{t,u}_T\right]^\frac{1}{2} \right]\,dr.\\
    \end{aligned}
\end{equation}
The last step is due to
\eqref{l4-2-5} and Lemma \ref{l4-3}. Hence, applying Gronwall's inequality, we obtain
\begin{align*}
&\quad \sup_{0\le t\le s\le T}
\E\left[\left|W_s^{t,u,i,j}\right|^q\right]\\
&\le c_{11}e^{c_{11}\mathbf{C}(q,T)}\left(e^{c_{11}\widetilde{\mathbf{C}}(T)}T^{q-1}\int^T_0\mathbb{E}\left[ \left| \nabla^2 w(r,X^{t,\pi(u)}_r) \right|^q \right]\,dr+
\mathbf{C}(q,T)e^{c_{11}\mathbf{C}(2q,T)}\right) .
\end{align*}
Therefore by \eqref{l4-1-2} we have
\begin{equation}\label{l4-2-6}
\begin{split}
&\quad \sup_{0\le t\le s\le T}
\int_M\E\left[\left|W_s^{t,u(x),i,j}\right|^p\right]\mu(dx)
\le c_{12}\left(e^{c_{12}\widetilde{\mathbf{C}}(T)}T^{p}
K_T(w)^p+\mathbf{C}(p,T)e^{c_{12}\mathbf{C}(2p,T)}\right) e^{c_{12}\mathbf{C}(p,T)}.
\end{split}
\end{equation}
Based on \eqref{l4-2-2a} we deduce that
\begin{align*}
&\quad \E\left[\int_M\left|\h_{A_i}\h_{A_j}\left(\S(h)\left(U_s^{t,\cdot}\right)\right)(u(x))\right|^p\mu(dx)\right]\\
&\le c_{13}\|h\|_{2,p}^p\cdot\int_M\E\left[\Big||V_s^{t,u(x),i}|\cdot|V_s^{t,u(x),j}|+|R_s^{t,x,i}|\cdot|V_s^{t,u(x),j}|+
|R_s^{t,x,j}|\cdot|V_s^{t,u(x),i}|+|W_s^{t,u(x),i,j}|\Big|^p\right]\mu(dx)\\
&+c_{13}\int_M\E\left[|R_s^{t,x,i}|^p\cdot |R_s^{t,x,j}|^p \cdot \left|\nabla^2 h(X_s^{t,x})\right|^p\right]\mu(dx)\\
&\le c_{14} \|h\|_{2,p}^p\cdot\left(\sup_{1\le i \le k}\sup_{x\in M}\left(\E\left[|V_s^{t,u(x),i}|^{2p}\right]+
\E\left[|R_s^{t,x,i}|^{2p}\right]\right)+\int_M\E\left[\left|W_s^{t,u(x),i,j}\right|^p\right]\mu(dx)\right)\\
&+c_{14}\sup_{1\le i \le k}\sup_{x\in M}\E\left[|R_s^{t,x,i}|^{2p}|\mathscr{G}_T^{t,u(x)}\right]
\cdot\int_M \E\left[\left|\nabla^2 h(X_s^{t,x})\right|^p\right]\mu(dx)\\
&\le c_{15}\|h\|^p_{2,p}\Big( e^{c_{15}\mathbf{C}(2p,T)} +e^{c_{15}\widetilde{\mathbf{C}}(T)}\Big)
+c_{15}\|h\|^p_{2,p}\left(e^{c_{15}\widetilde{\mathbf{C}}(T)}T^{p}K_T(w)^p+\mathbf{C}(p,T)e^{c_{15}\mathbf{C}(2p,T)}\right) e^{c_{15}\mathbf{C}(p,T)}\\
&\le c_{16}\|h\|^p_{2,p}\Big(e^{c_{16}\widetilde{\mathbf{C}}(T)}+e^{c_{16}\mathbf{C}(2p,T)}+\mathbf{C}(p,T)e^{c_{16}\big(\mathbf{C}(p,T)+
\mathbf{C}(2p,T)\big)}+K_T(w)^pT^p  e^{c_{16}\big(\widetilde{\mathbf{C}}(T)+\mathbf{C}(p,T)\big)}\Big),
\end{align*}
where the second inequality follows from H\"older's inequality and the last step is due to
\eqref{l4-3-1}, \eqref{l4-2-5}, \eqref{l4-2-6}. So we have proved \eqref{l4-2-2}.

\end{proof}

\begin{lemma}\label{l4-4}
Let $T>0$ and  $w_1,w_2\in \C_T^\infty$. Let $\{U_{s,1}^{t,u}\}_{0\le t\le s\le T}$,
$\{U_{s,2}^{t,u}\}_{0\le t\le s\le T}$  be solutions of the
$\F(M)$-valued SDE \eqref{l4-1-1} with $w$ replaced by $w_1$, $w_2$ respectively.
Define $X_{s,i}^{t,x}:=\pi(U_{s,i}^{t,u})$ for every $u\in \F(M)$ with $\pi(u)=x$ and $i=1,2$.
Then, for all $p>d$, there exists a positive constant
$C_9$ which is independent of $w_1$, $w_2$, $\nu$ and $T$ and such that, for every $f\in C_b^1(M)$, $h\in C_b^2(\Gamma(TM))$,
\begin{equation}\label{l4-4-1}
\begin{split}
\sup_{0\le t\le s\le T}\sup_{x\in M}\E\left[ \left|f\left(X_{s,1}^{t,x}\right)-
f\left(X_{s,2}^{t,x}\right)\right|^p\right]\le C_9 Te^{C_9\CC_0(T)}\|\nabla f\|_\infty^p \cdot \sup_{s\in [0,T]}\|w_1(s)-w_2(s)\|_{1,p}^p,
\end{split}
\end{equation}
\begin{equation}\label{l4-4-0}
\begin{split}
& \quad \sup_{0\le t\le s\le T} \sup_{u\in \F(M)}\E\left[\left|\S(h)\left(U_{s,1}^{t,u}\right)-
\S(h)\left(U_{s,2}^{t,u}\right)\right|^p\right]\\
&\le C_9\|h\|_{2,p}^pT^pe^{C_9\mathbf{C}_0(p,T)}\cdot \sup_{s\in [0,T]}\|w_1(s)-w_2(s)\|^p_{1,p},
\end{split}
\end{equation}
\begin{equation}\label{l4-4-2}
\begin{split}
&\quad \sup_{0\le t\le s\le T} \E\left[\int_M \left|\h_{A_j}\left(\S(h)\left(U_{s,1}^{t,\cdot}\right)\right)(u(x))-
\h_{A_j}\left(\S(h)\left(U_{s,2}^{t,\cdot}\right)\right)(u(x))\right|^p \mu(dx)\right]\\
&\le C_9\|h\|^p_{2,p}e^{C_{9}\mathbf{C}_0(p,T)}\Bigl( \big(\mathbf{C}_0(p,T)+1\big)T^pe^{C_{9}\mathbf{C}_0(2p,T)}+T^pe^{C_{9}\mathbf{C}_0(p,T)}\\
&\quad+e^{C_{9}\widetilde{\mathbf{C}}_0(T)}
\left(T^p+T^{\frac{1}{2}}+T^{p+\frac{1}{2}}K_T(w_1,w_2)^p\right)\Bigr)\cdot \sup_{s\in [0,T]}\|w_1(s)-w_2(s)\|^p_{1,p},
\end{split}
\end{equation}
where $x\mapsto u(x)$ is a continuous map from $M$ to $\F(M)$ which satisfies $\pi(u(x))=x$ and
$\mathbf{C}_0(p,T)$, $\widetilde{\mathbf{C}}_0(T)$ are defined as in \eqref{l4-2-0} with $K_T(w)$ replaced by
$K_T(w_1,w_2):=\sup_{s\in [0,T]}\left(\|w_1(s)\|_{2,p}+\|w_2(s)\|_{2,p}\right)$.
\end{lemma}
\begin{proof}
For simplicity  we will only prove \eqref{l4-4-2}. Estimates \eqref{l4-4-1} and \eqref{l4-4-0} can be proved with similar (but simpler) arguments.

{\bf Step (i)}
For every $r\in [1,2]$, suppose that $\{U_{s,r}^{t,u}\}_{0\le t\le s\le T}$ is the solution of following
$\F(M)$-valued SDE,
\begin{equation*}
\begin{cases}
& dU_{s,r}^{t,u}=\sqrt{2\nu}\sum_{i=1}^k \h_{A_i}\left(U_{s,r}^{t,u}\right)\circ dB_s^i+
(2-r)\h_{w_1(s)}\left(U_{s,r}^{t,u}\right)ds+
(r-1)\h_{w_2(s)}\left(U_{s,r}^{t,u}\right)ds\\
& U_{t,r}^{t,u}=u.
\end{cases}
\end{equation*}
Let $X_{s,r}^{t,x}:=\pi(U_{s,r}^{t,u})$ for any $u\in \F(M)$ with $\pi(u)=x$. Define also  $V_{s,r}^{t,u}:=\frac{\partial U_{s,r}^{t,u}}{\partial r}$, $W_{s,r}^{t,u,j}:=DU_{s,r}^{t,\cdot}(\h_{A_j})$,
$\Gamma_{s,r}^{t,u,j}:=\frac{\partial W_{s,r}^{t,u,j}}{\partial r}$. Then, repeating  the same procedures as for the proof of \eqref{l4-1-2}
and \eqref{l4-2-5}, we can
deduce that
\begin{equation}\label{l4-4-1a}
\sup_{r\in [1,2]}\sup_{0\le t\le s \le T}\E\left[\int_M |f(X_{s,r}^{t,x})|\mu(dx)\right]\le e^{c_0\CC_0(T)}\|f\|_{1},\ \ \forall\ f\in C^1(M).
\end{equation}
\begin{equation}\label{l4-4-2a}
\sup_{r\in [1,2];u\in \F(M)}\sup_{0\le t\le s \le T}\E\left[\left|W_{s,r}^{t,u,j}\right|^q\right]\le
c_0 e^{c_1\mathbf{C}_0(q,T)}. 
\end{equation}
Since $M$ is compact, we can exchange the order of differential and integral in the equation above to derive that $V_{s,r}^{t,u}$ and $\Gamma_{s,r}^{t,u,j}$ satisfy the following equations (which are equations in ambient space $\R^L$) respectively,
\begin{equation}\label{l4-4-3}
\begin{cases}
&dV_{s,r}^{t,u}=\sqrt{2\nu}\sum_{i=1}^k D\h_{A_i}\left(U_{s,r}^{t,u}\right)V_{s,r}^{t,u}\circ dB_s^i+
\big((2-r)D\h_{w_1(s)}\left(U_{s,r}^{t,u}\right)\\
&+(r-1)D\h_{w_2(s)}\left(U_{s,r}^{t,u}\right)\big)V_{s,r}^{t,u}ds
+\left(\h_{w_2(s)}\left(U_{s,r}^{t,u}\right)-\h_{w_1(s,\cdot)}\left(U_{s,r}^{t,u}\right)\right)ds,\\
&V_{t,r}^{t,u}=0.
\end{cases}
\end{equation}
\begin{equation}\label{l4-4-3a}
\begin{cases}
&d\Gamma_{s,r}^{t,u,j}=\sqrt{2\nu}\sum_{i=1}^k \left(D\h_{A_i}\left(U_{s,r}^{t,u}\right)\Gamma_{s,r}^{t,u,j}+D^2 \h_{A_i}\left(U_{s,r}^{t,u}\right)
\left(V_{s,r}^{t,u}, W_{s,r}^{t,u,j}\right)\right)\circ dB_s^i\\
&+\big((2-r)D\h_{w_1(s)}\left(U_{s,r}^{t,u}\right)
+(r-1)D\h_{w_2(s)}\left(U_{s,r}^{t,u}\right)\big)\Gamma_{s,r}^{t,u,j}ds\\
&+\big((2-r)D^2\h_{w_1(s)}\left(U_{s,r}^{t,u}\right)
+(r-1)D^2\h_{w_2(s)}\left(U_{s,r}^{t,u}\right)\big)\left(V_{s,r}^{t,u}, W_{s,r}^{t,u,j}\right)ds\\
&+\left(D\h_{w_2(s)}\left(U_{s,r}^{t,u}\right)-D\h_{w_1(s)}\left(U_{s,r}^{t,u}\right)\right)W_{s,r}^{t,u,j}ds,\\
&\Gamma_{t,r}^{t,u,j}=0.
\end{cases}
\end{equation}

As explained in the proof of Lemma \ref{l4-2}, using the expression in local charts and applying Sobolev's embedding theorem
(note that $p>d$), we obtain
\begin{align*}
&\sup_{u\in \F(M)}\left(|\h_{A_i}(u)|+|D\h_{A_i}(u)|+|D^2 \h_{A_i}(u)|+|D^3 \h_{A_i}(u)|\right)\le c_1,\ \forall\ 1\le i \le k,
\end{align*}
\begin{align*}
&\quad \sup_{u\in \F(M)}\left|D\h_{w_1(s)}(u)\right|+\left|D\h_{w_2(s)}(u)\right|\\
&\le c_1\sup_{x\in M}\left(|w_1(s,x)|+|w_2(s,x)|+
|\nabla w_1(s,x)|+|\nabla w_2(s,x)|\right)
\le c_2K_T(w_1,w_2),
\end{align*}
\begin{align*}
\quad \left|D^2\h_{w_1(s)}(u)\right|+\left|D^2\h_{w_2(s)}(u)\right|
&\le c_1\left(
\sum_{m=0}^2|\nabla^m w_1(s,\pi(u))|+\sum_{m=0}^2|\nabla^m w_2(s,\pi(u))|\right)\\
& \le c_2K_T(w_1,w_2)+c_1\left(\left|\nabla^2 w_1(s,\pi(u))\right|+\left|\nabla^2 w_2(s,\pi(u))\right|\right)
\end{align*}
\begin{align*}
\quad \sup_{u\in \F(M)}\left|\h_{w_1(s)}(u)-\h_{w_2(s)}(u)\right|
&\le c_1\sup_{s\in [0,T]}\|w_1(s)-w_2(s)\|_\infty\\
&\le c_2\sup_{s\in [0,T]}\|w_1(s)-w_2(s)\|_{1,p},
\end{align*}
\begin{align*}
&\quad \left|D\h_{w_1(s)}(u)-D\h_{w_2(s)}(u)\right|\\
&\le c_1\left(|w_1(s,\pi(u))-w_2(s,\pi(u))|+|\nabla w_1(s,\pi(u))-\nabla w_2(s,\pi(u))|\right)\\
&\le c_2\left(\sup_{s\in [0,T]}\|w_1(s)-w_2(s)\|_{1,p}+|\nabla w_1(s,\pi(u))-\nabla w_2(s,\pi(u))|\right).
\end{align*}
According to the estimates above and applying the same arguments as in the proof of Lemma \ref{l4-2} (in particular those for
estimates \eqref{l4-2-6}) to  linear SDEs \eqref{l4-4-3} and \eqref{l4-4-3a}
we get that, for every $q\ge 2$ and $1\le j \le k$,
\begin{equation}\label{l4-4-4}
\begin{split}
&\sup_{0\le t\le s \le T}\sup_{r\in [1,2];u\in \F(M)}\E\left[\left|V_{s,r}^{t,u}\right|^q\right]
\le c_3 T^qe^{c_4\mathbf{C}_0(q,T)}\sup_{s\in [0,T]}\|w_1(s)-w_2(s)\|_{1,p}^q,
\end{split}
\end{equation}
and for every $p>1, 1\le j\le k$,
\begin{align*}
&\quad\E\left[\left|\Gamma_{s,r}^{t,u(x),j}\right|^p\right]\le \frac{c_3\mathbf{C}_0(p,T)}{T}\int_t^s
\E\left[\left|\Gamma_{s',r}^{t,u(x),j}\right|^p\right]ds'\\
&+c_3\sup_{s\in [0,T]}\|w_1(s)-w_2(s)\|_{1,p}^p\cdot\Bigg(\mathbf{C}_0(p,T)T^pe^{c_4\mathbf{C}_0(2p,T)}+T^pe^{c_4\mathbf{C}_0(p,T)}
\\
&+e^{c_4\widetilde{\mathbf{C}}_0(T)}T^{p-\frac{1}{2}}\int_0^T \sup_{i=1,2}\E\left[\left|\nabla^2 w_i\left(s',X_{s',r}^{t,x}\right)\right|^p
\right]ds'\Bigg)\\
&+c_3e^{c_4\widetilde{\mathbf{C}}_0(T)}
T^{p-1}\int_0^T\E\left[\left|\nabla w_1\left(s',X_{s',r}^{t,x}\right)-\nabla w_2\left(s',X_{s',r}^{t,x}\right)\right|^p\right]ds',
\end{align*}
where we have also used estimates \eqref{l4-4-6a} and \eqref{l4-4-6} established below. Together with \eqref{l4-4-1a}
we deduce
\begin{equation}\label{l4-4-4a}
\begin{split}
&\quad \sup_{0\le t\le s\le T}\sup_{r\in [1,2];u\in \F(M)}
\int_M\E\left[\left|\Gamma_{s,r}^{t,u(x),j}\right|^p\right]\mu(dx)\\
&\le c_5 e^{c_6\mathbf{C}_0(p,T)}\Bigl(\mathbf{C}_0(p,T)T^pe^{c_6\mathbf{C}_0(2p,T)}+T^p e^{c_6\mathbf{C}_0(p,T)}+e^{c_6\widetilde{\mathbf{C}}_0(T)}
\bigl(T^p+T^{p+\frac{1}{2}}K_T(w_1,w_2)^p\bigr)\Bigr)
 \\
&\quad \cdot\sup_{s\in [0,T]}\|w_1(s)-w_2(s)\|_{1,p}^p.
\end{split}
\end{equation}

{\bf Step (ii)} Let $R_{s,r}^{t,x}=\frac{\partial}{\partial r}X_{s,r}^{t,x}$, $Q_{s,r}^{t,x,j}:=DX_{s,r}^{t,\cdot}(A_j)$ and $\mathscr{G}_{T,r}^{t,u}:=\sigma\{U_{s,r}^{t,u};t\le s\le T\}$. Following the same arguments as for \eqref{l4-3-1a}
(in particular the idea to decompose the filtration generated by the noise $\{B_s;s\in [0,T]\}$, see e.g. \cite[Theorem 3.1, 3.2]{ELL} or \cite[Chapter 3]{EL}
for details), we know that, for every $r\in [1,2]$, $R_{s,r}^{t,x}$ and $Q_{s,r}^{t,x,j}$ satisfy the following equations, respectively
\begin{equation*}
\begin{cases}
&\mathbb{D}_s R_{s,r}^{t,x}=\sum_{i=1}^k \sqrt{2\nu}\nabla_{R_s^{t,x}}A_i\left(X_{s,r}^{t,x}\right) d\tilde{\beta}^{i,r}_{s}-
\nu\text{Ric}^\sharp_{X_{s,r}^{t,x}}R_{s,r}^{t,x}ds\\
&+\nabla_{R_{s,r}^{t,x}}\left((2-r)w_1\left(s\right)+(r-1)w_2\left(s\right)\right)(X_{s,r}^{t,x})ds+
\left(w_2(s,X_{s,r}^{t,x})-w_1(s,X_{s,r}^{t,x})\right)ds
,\\
& R_{t,r}^{t,x}=0,
\end{cases}
\end{equation*}
\begin{equation*}
\begin{cases}
&\mathbb{D}_s Q_{s,r}^{t,x,j}=\sum_{i=1}^k \sqrt{2\nu}\nabla_{Q_{s,r}^{t,x,j}}A_i\left(X_{s,r}^{t,x}\right) d\tilde{\beta}^{i,r}_{s}-
\nu\text{Ric}^\sharp_{X_{s,r}^{t,x}}Q_{s,r}^{t,x,j}ds\\
&+\nabla_{Q_{s,r}^{t,x,j}}\left((2-r)w_1\left(s\right)+(r-1)w_2\left(s\right)\right)(X_{s,r}^{t,x})ds,\\
& Q_{t,r}^{t,x,j}=A_j(x),
\end{cases}
\end{equation*}
where $\mathbb{D}$ denotes the stochastic covariant differential along $X_{\cdot,r}^{t,x}$,
$d\tilde{\beta}^{r}_s:=//_{s,r}^{t,x}d\beta^{r}$ such that $//_{s,r}^{t,x}:\R^k \to \R^k$ is a $\mathscr{G}_{T,r}^{t,u(x)}$-measurable
metric adapted translation and
$\{\beta_{s}^r=(\beta_s^{1,r},\cdots, \beta_s^{k,r});s\in [t,T]\}$ is an $\R^k$-valued Brownian motion independent of $\mathscr{G}_{T,r}^{t,u(x)}$.

Based on the second equation above, we can repeat the computations in the proof of Lemma \ref{l4-3} to deduce that,
for every $q\ge 2$,
\begin{equation}\label{l4-4-6a}
\sup_{0\le t\le s \le T}\sup_{r\in [1,2],x\in M}\E\big[\big|Q_{s,r}^{t,x,j}\big|^q\big|\mathscr{G}_{T,r}^{t,u(x)}\big]\le
c_7e^{c_8\widetilde{\mathbf{C}}_0(T)}.
\end{equation}

On the other hand we can apply It\^o's formula to get, for every $q\ge 2$,
\begin{align*}
d|R_{s,r}^{t,x}|^q&=q\sum_{i=1}^k M_{s,r}^{t,i}|R_{s,r}^{t,x}|^qd \tilde{\beta}_s^{i,r}+\left(\frac{q}{2}N_{s,r}^t+\frac{q(q-2)}{2}\sum_{i=1}^k |M_{s,r}^{t,i}|^2+(q-1)\right)
|R_{s,r}^{t,x}|^qds+K_{s,r}^{t,x}ds,
\end{align*}
where
\begin{align*}
& M_{s,r}^{t,i}:=\sqrt{2\nu}\frac{\left\langle \nabla_{R_{s,r}^{t,x}}A_i\left(X_{s,r}^{t,x}\right),
R_{s,r}^{t,x}\right\rangle}{|R_{s,r}^{t,x}|^2},\\
& N_{s,r}^t=2\nu\sum_{i=1}^k \frac{\left|\nabla_{R_{s,r}^{t,x}}A_i\left(X_{s,r}^{t,x}\right)\right|^2}{|R_{s,r}^{t,x}|^2}-
\frac{2\nu\left\langle \text{Ric}^\sharp_{X_{s,r}^{t,x}}R_{s,r}^{t,x}, R_{s,r}^{t,x} \right\rangle}{|R_{s,r}^{t,x}|^2}\\
&+\frac{2\left\langle \nabla_{R_{s,r}^{t,x}}\left((2-r)w_1(s)+(r-1)w_2(s)\right)(X_{s,r}^{t,x}),
R_{s,r}^{t,x}\right\rangle}{|R_{s,r}^{t,x}|^2},\\
&K_{s,r}^{t,x}=q\left\langle R_{s,r}^{t,x}, w_1(s,X_{s,r}^{t,x})-w_2(s,X_{s,r}^{t,x})\right\rangle\cdot |R_{s,r}^{t,x}|^{q-2}-(q-1)|R_{s,r}^{t,x}|^{q}.
\end{align*}

Hence we have
\begin{equation}\label{l4-4-5}
|R_{s,r}^{t,x}|^q=\int_t^s L_{s,r}^{t,x}(L_{v,r}^{t,x})^{-1}K_{v,r}^{t,x}dv,
\end{equation}
where $L_{s,r}^{t,x}$ is the solution to following equation (which is in fact strictly positive),
\begin{equation*}
\begin{cases}
& dL_{s,r}^{t,x}=q\sum_{i=1}^k M_{s,r}^{t,i}L_{s,r}^{t,x}d\tilde{\beta}_s^{i,r}+\left(\frac{q}{2}N_{s,r}^t+\frac{q(q-2)}{2}\sum_{i=1}^k |M_{s,r}^{t,i}|^2+(q-1)\right)
L_{s,r}^{t,x}ds,\\
&L_{t,r}^{t,x}=1.
\end{cases}
\end{equation*}

Applying Young's inequality it is easy to see that
\begin{equation}\label{l4-4-5a}
\sup_{0\le t\le s \le T}\sup_{r\in [1,2],x\in M}K_{s,r}^{t,x}\le c_9\sup_{s\in [0,T]}\|w_1(s)-w_2(s)\|_\infty^q.
\end{equation}

Based on the (linear) equation for $L_{s,r}^{t,x}$ and following the same arguments as in the proof of Lemma \ref{l4-3}, we can verify that
\begin{equation*}
\begin{split}
\sup_{0\le t\le s \le T}\sup_{r\in [1,2],x\in M}\E\big[L_{s,r}^{t,x}(L_{v,r}^{t,x})^{-1}\big|\mathscr{G}_{T,r}^{t,u(x)}\big]\le
e^{c_{10}\widetilde{\mathbf{C}}_0(T)}.
\end{split}
\end{equation*}
Combining this with \eqref{l4-4-5a} into \eqref{l4-4-5} yields
\begin{equation}\label{l4-4-6}
\sup_{0\le t\le s \le T}\sup_{r\in [1,2],x\in M}\E\big[\big|R_{s,r}^{t,x}\big|^q\big|\mathscr{G}_{T,r}^{t,u(x)}\big]\le
c_{11}e^{c_{12}\CC_0(T)}T\cdot\sup_{s\in [0,T]}\|w_1(s)-w_2(s)\|_\infty^q.
\end{equation}

On the other hand, for every $u\in \F(M)$, the following estimates hold
\begin{align*}
&\quad \left|\h_{A_j}\left(\S(h)(U_{s,2}^{t,\cdot})\right)(u)-\h_{A_j}\left(\S(h)(U_{s,1}^{t,\cdot})\right)(u)\right|\\
&=\left|\int_1^2 \frac{\partial}{\partial r}\Big(\h_{A_j}\left(\S(h)(U_{s,r}^{t,\cdot})\right)(u)\Big)dr\right|
=\left|\int_1^2 \frac{\partial}{\partial r}\Big(D\left(\S(h)(U_{s,r}^{t,\cdot})\right)(u)\left(\h_{A_j}\right)\Big)dr\right|\\
&\le c_{13}\int_1^2 \Big(\left|\nabla^2 h(X_{s,r}^{t,\pi(u)})\right|\left|R_{s,r}^{t,\pi(u)}\right|\left|Q_{s,r}^{t,\pi(u),j}\right|\\
&\quad\quad +\left(\left|\nabla h(X_{s,r}^{t,\pi(u)})\right|+\left| h(X_{s,r}^{t,\pi(u)})\right|\right)\cdot \left(\left|\Gamma_{s,r}^{t,u,j}\right|+
\left|W_{s,r}^{t,u,j}\right|\left|V_{s,r}^{t,u}\right|\right)\Big)dr\\
&\le c_{14}\int_1^2 \Big(\left|\nabla^2 h(X_{s,r}^{t,\pi(u)})\right|\left|R_{s,r}^{t,\pi(u)}\right|\left|Q_{s,r}^{t,\pi(u),j}\right|
+\|h\|_{2,p}\cdot \left(\left|\Gamma_{s,r}^{t,u,j}\right|+
\left|W_{s,r}^{t,u,j}\right|\left|V_{s,r}^{t,u}\right|\right)\Big)dr,
\end{align*}
where the first inequality can be deduced by the expression of $\S(h)$ in local charts and the second inequality follows from Sobolev's embedding theorem.

Therefore, based on these estimates, we have
\begin{equation}\label{l4-4-7}
\begin{split}
&\quad \E\left[\int_M \left|\h_{A_j}\left(\S(h)(U_{s,2}^{t,\cdot})\right)(u(x))-\h_{A_j}\left(\S(h)(U_{s,1}^{t,\cdot})\right)(u(x))\right|^p\mu(dx)\right]\\
&\le c_{15}\Bigg(\int_1^2 \int_M \E\left[\left|\nabla^2 h(X_{s,r}^{t,x})\right|^p\cdot
\E\left[\left|R_{s,r}^{t,x}\right|^p\left|Q_{s,r}^{t,x,j}\right|^p\Big|\mathscr{G}_{T,r}^{t,u(x)}\right]\right]\mu(dx)dr\\
&\quad \quad + \|h\|_{2,p}^p \cdot \left(\int_1^2\int_M\E\left[\left|\Gamma_{s,r}^{t,u(x),j}\right|^p+
\left|W_{s,r}^{t,u(x),j}\right|^p\left|V_{s,r}^{t,u(x)}\right|^p\right]\mu(dx)dr\right)\Bigg)\\
&\le c_{16}\Bigg(\int_1^2 \left(\int_M\E\left[\left|\nabla^2 h(X_{s,r}^{t,x})\right|^p\right]\mu(dx)\right)\cdot
\left(\sup_{x\in M}\E\left[\left|R_{s,r}^{t,x}\right|^{2p}\Big|\mathscr{G}_{T,r}^{t,u(x)}\right]
\E\left[\left|Q_{s,r}^{t,x,j}\right|^{2p}\Big|\mathscr{G}_{T,r}^{t,u(x)}\right]\right)^{\frac{1}{2}}dr\\
&+\|h\|^p_{2,p}\cdot\left(\int_1^2 \int_M\E\left[\left|\Gamma_{s,r}^{t,u(x),j}\right|^p\right]\mu(dx)dr+
\int_1^2 \sup_{u\in \F(M)}\left(\E\left[\left|W_{s,r}^{t,u,j}\right|^{2p}\right]^{\frac{1}{2}}
\E\left[\left|V_{s,r}^{t,u}\right|^{2p}\right]^{\frac{1}{2}}\right)dr\right)\Bigg)\\
&\le c_{17}\|h\|^p_{2,p}e^{c_{17}\mathbf{C}_0(p,T)}\Bigl(\big(\mathbf{C}_0(p,T)+1\big)T^pe^{c_{17}\mathbf{C}_0(2p,T)}+T^pe^{c_{17}\mathbf{C}_0(p,T)} \\
&\quad+e^{c_{17}\widetilde{\mathbf{C}}_0(T)}
\left(T^p+T^{\frac{1}{2}}+T^{p+\frac{1}{2}}K_T(w_1,w_2)^p\right)\Bigr) \cdot \sup_{s\in [0,T]}\|w_1(s)-w_2(s)\|^p_{1,p}.
\end{split}
\end{equation}
Here the first inequality is due to H\"older's inequality and the fact that $X_{s,r}^{t,x}$ is $\mathscr{G}_{T,r}^{t,u(x)}$-measurable. For the second inequality we
 applied H\"older's inequality again and the last step follows from  estimates \eqref{l4-4-1a}, \eqref{l4-4-2a}, \eqref{l4-4-4}, \eqref{l4-4-4a},
\eqref{l4-4-6a}, \eqref{l4-4-6}. By now we have finished the proof of \eqref{l4-4-2}.
\end{proof}
\vskip 5mm
Given $\nu>0$, $T>0$, $w\in \C_T^\infty$ and $v_0\in C^\infty_b(\Gamma(TM))$  with $\div v_0=0$, we consider the following FBSDE,
\begin{equation}\label{e4-4}
\begin{cases}
&dU_s^{t,u}=\sqrt{2\nu}\sum_{i=1}^k \h_{A_i}(U_s^{t,u})\circ dB_s^i-\h_{w(s)}(U_s^{t,u})ds,\ 0\le t\le s \le T,\\
&d\tilde Y_s^{t,u}=\sum_{i=1}^k \tilde Z_s^{t,u,i}dB_s^i-\tilde F_{w(s)}(U_s^{t,u})ds,\\
& U_t^{t,u}=u,\ \tilde Y_T^{t,u}=\S(v_0)(U_T^{t,u}),
\end{cases}
\end{equation}
where $F_{w(s)}$ is defined by \eqref{e3-4} associated with the vector field $w(s,\cdot)$ and $\tilde F_{w(s)}\in T^{1,0}_0\R^d$ is the scalarization of
$F_{w(s)}$ defined by \eqref{e2-7a}.  Then we define $\I_{\nu,T,v_0}: \C_T^\infty \to \C_T^\infty$ by
\begin{equation}\label{e4-5}
\I_{\nu,T,v_0}(w)(t,x):=\p\left(\theta(t)\right)(x),\ \ \forall\ \ (t,x)\in [0,T]\times M,\ w\in \C_T^\infty,
\end{equation}
where $\p$ is the projection map defined in Lemma \ref{l3-3},
$\theta(t,x):=u\tilde Y_t^{t,u}\in T_x M$ for every $u\in \F(M)$ with $\pi(u)=x$
as explained in the proof of Theorem \ref{t2-2}. The value of
$\theta(t,x)$ is independent of the choice of $u$, so $\theta(t,x)$ is well defined (since $w\in \C_T^\infty$ is regular enough, by
the arguments in \cite{PP2} it is not difficult to verify that $\theta\in C^{1,\infty}_b\left([0,T];\Gamma(TM)\right)$).

We have the following

\begin{lemma}\label{l4-5}
Suppose that $\nu>0$ and $v_0\in C^\infty_b(\Gamma(TM))\cap L^2_{{\rm div}}(\Gamma(TM))$. For every $p>d$, there exist
positive constants $K_0(\|v_0\|_{2,p},\nu)$ (which depend on $\|v_0\|_{2,p}$ and $\nu$) and $T_0(\|v_0\|_{2,p},\nu)$ such that, for every $T\in (0,T_0)$ and
$w_1,w_2\in \C_T^\infty$ satisfying $\sup_{t\in [0,T]}\max\big(\|w_1(t)\|_{2,p},$ $\|w_2(t)\|_{2,p}\big)\le K_0$, the
following estimates hold.
\begin{equation}\label{l4-5-1}
\begin{split}
\sup_{t\in [0,T]}\|\I_{\nu,T,v_0}(w_1)(t,\cdot)\|_{2,p}\le K_0,
\end{split}
\end{equation}
\begin{equation}\label{l4-5-2}
\begin{split}
\sup_{t\in [0,T]}\|\I_{\nu,T,v_0}(w_1)(t,\cdot)-\I_{\nu,T,v_0}(w_2)(t,\cdot)\|_{1,p}\le \frac{1}{2}\sup_{t\in [0,T]}\|w_1(t)-w_2(t)\|_{1,p}.
\end{split}
\end{equation}
Moreover, if we assume that $\nu\in (0,1)$, then the constants $K_0$ and $T_0$ chosen above only depend on $\|v_0\|_{2,p}$ (and are independent of $\nu$).
\end{lemma}
\begin{proof}
{\bf Step (i)} For $l=1,2$, let $\theta_l(t,x):=u\tilde Y_{t,l}^{t,u}\in T_x M$ for every $u\in \F(M)$ with $\pi(u)=x$, where
$\left(U_{s,l}^{t,u}, \tilde Y_{s,l}^{t,u},\tilde Z_{s,l}^{t,u}\right)$ denotes the solution of \eqref{e4-4} with
associated $w\in \C_T^\infty$ replaced by $w_l(s)\in \C_T^\infty$. By definition of $\I_{\nu,T,v_0}$ we
 know that $\I_{\nu,T,v_0}(w_l)(t,\cdot)=\p\left(\theta_l(t)\right)(\cdot)$.

By definition of $\theta_l$ we have
$\tilde Y_{t,l}^{t,u}=\S(\theta_l)(t,u)$. According to properties \eqref{t2-1-1}, \eqref{t2-1-2}, \eqref{e4-1}, \eqref{e4-2}, we can deduce that
\begin{equation}\label{l4-5-3}
\|\theta_l(t)\|_{2,p}^p\le c_0\sum_{i,j=1}^k
\int_M\left(\left|\tilde Y_{t,l}^{t,u(x)}\right|^p+
\left|\h_{A_i}\left(\tilde Y_{t,l}^{t,\cdot}\right)(u(x))\right|^p+
\left|\h_{A_j}\h_{A_i}\left(\tilde Y_{t,l}^{t,\cdot}\right)(u(x))\right|^p\right)\mu(dx).
\end{equation}
\begin{equation}\label{l4-5-2a}
\begin{split}
&\quad \|\theta_1(t)-\theta_2(t)\|_{1,p}^p\\
&\le
c_1\sum_{i=1}^k\int_M\left(\left|\tilde Y_{t,1}^{t,u(x)}-\tilde Y_{t,2}^{t,u(x)}\right|^p+
\left|\h_{A_i}\left(\tilde Y_{t,1}^{t,\cdot}\right)(u(x))-\h_{A_i}\left(\tilde Y_{t,2}^{t,\cdot}\right)(u(x))\right|^p\right)\mu(dx)
\end{split}
\end{equation}
Here, as before, $x\mapsto u(x)$ is a continuous map from $M$ to $\F(M)$ satisfying  $\pi(u(x))=x$.

We can exchange the order of differentiation and integration in  \eqref{e4-4}
(since $M$ is compact and $F_{w(s,)}\in C_b^\infty(\Gamma(TM))$) to obtain

\begin{equation}\label{l4-5-3a}
\begin{split}
\h_{A_j}\h_{A_i}\left(\tilde Y_{t,l}^{t,\cdot}\right)(u)&=\h_{A_j}\h_{A_i}\left(\S(v_0)\left(U_{T,l}^{t,\cdot}\right)\right)(u)-
\sum_{m=1}^k \int_t^T \h_{A_j}\h_{A_i}\left(\tilde Z_{s,l}^{t,\cdot,m}\right)(u)dB_s^m\\
&\quad +\int_t^T \h_{A_j}\h_{A_i}\left(\tilde F_{w_l(s)}\left(U_{s,l}^{t,\cdot}\right)\right)(u)ds.
\end{split}
\end{equation}
Noting that $Y_{t,l}^{t,\cdot}$ is non-random and according to the equation above, we have
\begin{equation*}
\begin{split}
&\quad \left|\h_{A_j}\h_{A_i}\left(\tilde Y_{t,l}^{t,\cdot}\right)(u)\right|
=\left|\E\left[\h_{A_j}\h_{A_i}\left(\tilde Y_{t,l}^{t,\cdot}\right)(u)\right]\right|\\
&\le \E\left[\left|\h_{A_j}\h_{A_i}\left(\S(v_0)\left(U_{T,l}^{t,\cdot}\right)\right)(u)\right|\right]
+\int_t^T \E\left[\left|\h_{A_j}\h_{A_i}\left(\tilde F_{w_l(s)}\left(s, U_{s,l}^{t,\cdot}\right)\right)(u)\right|\right] ds.
\end{split}
\end{equation*}
Hence we can deduce that
\begin{equation}\label{l4-5-4}
\begin{split}
&\quad \quad \int_M \left|\h_{A_j}\h_{A_i}\left(\tilde Y_{t,l}^{t,\cdot}\right)(u(x))\right|^p \mu(dx)\\
&\le c_3\Bigg(\int_M \E\left[\left|\h_{A_j}\h_{A_i}\left(\S(v_0)\left(U_{T,l}^{t,\cdot}\right)\right)(u(x))\right|^p\right]\mu(dx)\\
&\quad \quad +T^{p-1}\cdot \int_t^T \left(\int_M\E\left[\left|\h_{A_j}\h_{A_i}
\left(\tilde F_{w_l(s)}\left(U_{s,l}^{t,\cdot}\right)\right)(u(x))\right|^p \right]\mu(dx)\right)ds\Bigg)\\
&\le c_4\Big(e^{c_4\widetilde{\mathbf{C}}_0(T)}+e^{c_4\mathbf{C}_0(2p,T)}+\mathbf{C}_0(p,T)e^{c_4\big(\mathbf{C}_0(p,T)+\mathbf{C}_0(2p,T)\big)}+K_T(w_l)^pT^p  e^{c_4\big(\widetilde{\mathbf{C}}_0(T)+\mathbf{C}_0(p,T)\big)}\Big)\\
&\quad \cdot \left(\|v_0\|^p_{2,p} +T^p\sup_{s\in[0,T]} \|F_{w_l(s)}\|^p_{2,p}\right)\\
&\le c_5\Big(e^{c_4\widetilde{\mathbf{C}}_0(T)}+e^{c_4\mathbf{C}_0(2p,T)}+\mathbf{C}_0(p,T)e^{c_4\big(\mathbf{C}_0(p,T)+\mathbf{C}_0(2p,T)\big)}+K_T(w_l)^pT^p  e^{c_4\big(\widetilde{\mathbf{C}}_0(T)+\mathbf{C}_0(p,T)\big)}\Big)\\
&\quad \cdot \left(\|v_0\|^p_{2,p} +T^p\bigl( \nu^pK_T(w_l)^p+K_T(w_l)^{2p} \bigr)\right),
\end{split}
\end{equation}
Here $\mathbf{C}_0(p,T)$, $\CC_0(T)$ are the same constants as those in Lemma \ref{l4-4}, $K_T(w_l):=\sup_{s\in [0,T]}\|w_l(s)\|_{2,p}$, the first inequality above is due to
H\"older's inequality, in the second inequality we have applied \eqref{l4-2-2} and the last inequality follows from
\eqref{l3-4-1}.

By \eqref{l4-5-4} we can find positive constants $T_0$ and $K_0>\|v_0\|_{2,p}$ (note that $\lim_{T\downarrow 0}
\big(\mathbf{C}_0(p,T)$ $+\CC_0(T)\big)=0$)
such that, for every $T\in (0,T_0)$, $l=1,2$ and $w_l\in \C_T^\infty$ satisfying $\sup_{s\in [0,T]}\|w_l(s)\|_{2,p}\le K_0$, it holds that 
\begin{equation*}
\sum_{i,j=1}^k \sup_{t\in [0,T]}\int_M \left|\h_{A_j}\h_{A_i}\left(\tilde Y_{t,l}^{t,\cdot}\right)(u(x))\right|^p \mu(dx)\le \frac{1}{3c_0}\left(\frac{K_0}{C_3}\right)^p,
\end{equation*}
where $c_0$ and $C_3$ are the positive constants in \eqref{l4-5-3} and \eqref{l3-3-1} respectively.
Furthermore, if we assume $\nu\in (0,1)$, then according to \eqref{l4-5-4} the constants $K_0$ and $T_0$ chosen
above only depend on $\|v_0\|_{2,p}$ and are independent of $\nu$.

Similarly we can also prove that, for every $l=1,2$, $T\in (0,T_0)$, $t\in [0,T]$ and
$w_l\in \C_T^\infty$ satisfying $\sup_{s\in [0,T]}\|w_l(s)\|_{2,p}\le K_0$,
\begin{equation*}
\int_M\left|\tilde Y_{t,l}^{t,u(x)}\right|^p\mu(dx)+\sum_{i=1}^k\int_M \left|\h_{A_i}\left(\tilde Y_{t,l}^{t,\cdot}\right)(u(x))\right|^p \mu(dx)\le \frac{2}{3c_0}\left(\frac{K_0}{C_3}\right)^p
\end{equation*}
Putting all the above estimates into \eqref{l4-5-3} yields that
for every  $T\in (0,T_0)$ and $w_l\in \C_T^\infty$ satisfying $\sup_{s\in [0,T]}\|w_l(s)\|_{2,p}\le K_0$
\begin{equation}\label{l4-5-4a}
\sup_{t\in [0,T]}\|\theta_l(t)\|_{2,p}\le \frac{K_0}{C_3},\ \forall\ w_l\in \C_T^\infty.
\end{equation}
Hence, according to \eqref{l3-3-1}, we have, for every $T\in (0,T_0)$ and $w_1\in \C_T^\infty$ satisfying $\sup_{s\in [0,T]}\|w_1(s)\|_{2,p}\le K_0$
\begin{equation*}
    \begin{aligned}
        &\sup_{t\in [0,T]}\|\I_{\nu,T,v_0}(w_1)(t,\cdot)\|_{2,p}=\sup_{t\in [0,T]}\|\p\left(\theta_1(t)\right)(\cdot)\|_{2,p}\le K_0,
        \end{aligned}
\end{equation*}
This implies \eqref{l4-5-1}.

{\bf Step (ii)} According to the same arguments used for \eqref{l4-5-3a} we obtain, for every $l=1,2$,
\begin{equation*}
\begin{split}
\h_{A_i}\left(\tilde Y_{t,l}^{t,\cdot}\right)(u)&=\h_{A_i}\left(\S(v_0)\left(U_{T,l}^{t,\cdot}\right)\right)(u)-
\sum_{m=1}^k \int_t^T \h_{A_i}\left(\tilde Z_{s,l}^{t,\cdot,m}\right)(u)dB_s^m\\
&\quad +\int_t^T \h_{A_i}\left(\tilde F_{w_l(s)}\left(U_{s,l}^{t,\cdot}\right)\right)(u)ds.
\end{split}
\end{equation*}
Based on such expression we have
\begin{equation}\label{l4-5-5}
\begin{split}
&\quad\int_M \left|\h_{A_i}\left(\tilde Y_{t,1}^{t,\cdot}\right)(u(x))-\h_{A_i}\left(\tilde Y_{t,2}^{t,\cdot}\right)(u(x))\right|^p\mu(dx)\\
&\le c_8\Bigg(\int_M \E\left[\left|\h_{A_i}\left(\S(v_0)\left(U_{T,1}^{t,\cdot}\right)\right)(u(x))-
\h_{A_i}\left(\S(v_0)\left(U_{T,2}^{t,\cdot}\right)\right)(u(x))\right|^p\right]\mu(dx)\\
&+T^{p-1}\cdot \int_t^T \left(\int_M\E\left[\left|\h_{A_i}
\left(\tilde F_{w_1(s)}\left(U_{s,1}^{t,\cdot}\right)\right)(u(x))-
\h_{A_i}\left(\tilde F_{w_2(s)}\left(U_{s,2}^{t,\cdot}\right)\right)(u(x))\right|^p \right]\mu(dx)\right)ds\Bigg)\\
&\le c_9e^{c_{9}\mathbf{C}_0(p,T)}\Bigl(\big(\mathbf{C}_0(p,T)+1\big)T^pe^{c_{9}\mathbf{C}_0(2p,T)}+T^pe^{c_{9}\mathbf{C}_0(p,T)} +e^{c_{9}\widetilde{\mathbf{C}}_0(T)}
\left(T^p+T^{\frac{1}{2}}+T^{p+\frac{1}{2}}K_T(w_1,w_2)\right)\Bigr)\\
&\quad \cdot \left(\|v_0\|_{2,p}^p+T^p\sup_{s\in [0,T]}\|F_{w_1(s,\cdot)}\|^p_{2,p}\right)\cdot \sup_{s\in [0,T]}\|w_1(s)-w_2(s)\|^p_{1,p}\\
&\quad + c_9\bigl( e^{c_9\mathbf{C}(p,T)} + e^{c_9\widetilde{\mathbf{C}}(T)}\bigr)T^p\cdot \sup_{s\in [0,T]}\|F_{w_1(s,\cdot)}-F_{w_2(s,\cdot)}\|_{1,p}^p\\
&\le c_{10}\Bigg(e^{c_{9}\mathbf{C}_0(p,T)}\left(\big(\mathbf{C}_0(p,T)+1\big)T^pe^{c_{9}\mathbf{C}_0(2p,T)}+T^pe^{c_{9}\mathbf{C}_0(p,T)} +e^{c_{9}\widetilde{\mathbf{C}}_0(T)}
\left(T^p+T^{\frac{1}{2}}+T^{p+\frac{1}{2}}K_T(w_1,w_2)^p\right)\right)\\
&\quad \quad \quad \cdot \left(\|v_0\|_{2,p}^p+T^p\bigl( \nu^pK_T(w_1,w_2)^p+K_T(w_1,w_2)^{2p} \bigr)\right)\\
&\quad \quad \quad+\bigl(e^{c_9\mathbf{C}(p,T)} + e^{c_9\widetilde{\mathbf{C}}(T)}\bigr)T^p
\cdot\bigl(K_T(w_1,w_2)^p+\nu^p\bigr)\Bigg)\cdot \sup_{s\in [0,T]}\|w_1(s)-w_2(s)\|^p_{1,p}
\end{split}
\end{equation}
Here $K_T(w_1, w_2):=\sup_{s\in [0,T]}\max\left(\|w_1(s)\|_{2,p},\|w_2(s)\|_{2,p}\right)$, the first inequality above is due to
H\"older's inequality, the second  follows from \eqref{l4-2-1} and \eqref{l4-4-2} and
in the last one we have applied  \eqref{l3-4-1}, \eqref{l3-4-2}.

By \eqref{l4-5-5} there exist positive constants $T_0$ and $K_0>\|v_0\|_{2,p}$
such that, for every $T\in (0,T_0)$ and $w_1,w_2\in \C_T^\infty$ satisfying $K_T(w_1,w_2)\le K_0$, it holds that 
\begin{align*}
&\quad \sum_{i=1}^k\sup_{t\in [0,T]}\int_M \left|\h_{A_i}\left(\tilde Y_{t,1}^{t,\cdot}\right)(u(x))-\h_{A_i}\left(\tilde Y_{t,2}^{t,\cdot}\right)(u(x))\right|^p \mu(dx)\\
&\le \frac{1}{2c_1}\left(\frac{\sup_{t\in [0,T]}\|w_1(t)-w_2(t)\|_{1,p}}{2C_3}\right)^p,
\end{align*}
where $c_1$ and $C_3$ are the positive constants in \eqref{l4-5-2a} and \eqref{l3-3-2} respectively.
Moreover  the constants $K_0$ and $T_0$ above
are independent of $\nu$ if $\nu\in (0,1)$.

Following the same procedures as above, we can deduce that, for every $T\in (0,T_0)$ and $w_1,w_2\in \C_T^\infty$ satisfying $K_T(w_1,w_2)\le K_0$,
\begin{equation*}
\sup_{t\in [0,T]}\int_M \left|\tilde Y_{t,1}^{t,u(x)}-\tilde Y_{t,2}^{t,u(x)}\right|^p \mu(dx)\le \frac{1}{2c_1}\left(\frac{\sup_{s\in [0,T]}\|w_1(t)-w_2(t)\|_{1,p}}{2C_3}\right)^p.
\end{equation*}

Therefore, combining all the above estimates into \eqref{l4-5-2a} we deduce that, for every $T\in (0,T_0)$ and $w_1,w_2\in \C_T^\infty$ satisfying $K_T(w_1,w_2)\le K_0$,
\begin{equation}\label{l4-5-6}
\sup_{t\in [0,T]}\|\theta_1(t)-\theta_2(t)\|_{1,p}\le \frac{1}{2C_3}\sup_{s\in [0,T]}\|w_1(t)-w_2(t)\|_{1,p}.
\end{equation}
Then, according to \eqref{l3-3-2}, we obtain, for every $T\in (0,T_0)$ and $w_1,w_2\in \C_T^\infty$ satisfying $K_T(w_1,w_2)\le K_0$,
\begin{equation*}
\begin{split}
&\quad \sup_{t\in [0,T]}\|\I_{\nu,T,v_0}(w_1)(t,\cdot)-\I_{\nu,T,v_0}(w_2)(t,\cdot)\|_{1,p}
=\sup_{t\in [0,T]}\|\p\left(\theta_1(t)\right)(\cdot)-\p\left(\theta_2(t)\right)(\cdot)\|_{1,p}\\
&\le \frac{1}{2}\sup_{t\in [0,T]}\|w_1(t)-w_2(t)\|_{1,p}.
\end{split}
\end{equation*}
We have finished the proof of \eqref{l4-5-2}.
\end{proof}

We are now in position to prove  existence and uniqueness of a local solution in Sobolev space for the Navier-Stokes equation \eqref{e3-2} on $M$ via
the forward-backward stochastic differential system \eqref{e4-3}.

\begin{theorem}\label{t4-2}
Suppose that $v_0\in W^{2,p}(\Gamma(TM))\cap L_{{\rm div}}^p(\Gamma(TM))$ for some $p>d$. Then
we can find a positive constant $T_0\left(\|v_0\|_{2,p},\nu\right)$ such that
there exists a unique solution
$\Big(U_s^{t,u},\tilde Y_s^{t,u},\tilde Z_s^{t,u},$ $ v\Big)$ in the time interval $0\le t\le s \le T_0$
of the forward-backward stochastic differential system \eqref{e4-3}
satisfying that
$\left(U_\cdot^{t,u},\tilde Y_\cdot^{t,u},\tilde Z_\cdot^{t,u}\right)\in $
$\mathscr{C}([t,T_0];\mathscr{F}(M))\times\mathscr{C}([t,T_0];\R^{d})\times \mathscr{M}([t,T_0];\R^{dk})$
(for all $0\le t\le T_0$) and
$v\in C\left([0,T_0];W_{{\rm div}}^{2,p}\left(\Gamma(TM)\right)\right)$. Moreover $v\in C\left([0,T_0];W_{{\rm div}}^{2,p}\left(\Gamma(TM)\right)\right)$ is
a solution of the incompressible Navier-Stokes equation \eqref{e3-2} on $M$ (in time interval $[0,T_0]$).
\end{theorem}
\begin{proof}
{\bf Step (i)} We first assume that $v_0\in C_b^\infty\left(\Gamma(TM)\right)$. Let $\I_{\nu,T,v_0}:\C_T^\infty\to \C_T^\infty$
be the map studied in Lemma \ref{l4-5}. According to Lemma \ref{l4-5} we can find positive constants
$K_0\left(\|v_0\|_{2,p},\nu\right)$ and $T_0\left(\|v_0\|_{2,p},\nu\right)$
such that, for every $w_1,w_2\in \C_{T_0}^\infty$ with $\sup_{t\in [0,T_0]}\max\left(\|w_1(t)\|_{2,p}, \|w_2(t)\|_{2,p}\right)\le K_0$,
\begin{equation}\label{t4-2-1}
\sup_{t\in [0,T_0]}\|\I_{\nu,T_0,v_0}(w_1)(t,\cdot)\|_{2,p}\le K_0,
\end{equation}
\begin{equation}\label{t4-2-2}
\sup_{t\in [0,T_0]}\|\I_{\nu,T_0,v_0}(w_1)(t,\cdot)-\I_{\nu,T_0,v_0}(w_1)(t,\cdot)\|_{1,p}\le \frac{1}{2}
\sup_{t\in [0,T_0]}\|w_1(t)-w_2(t)\|_{1,p}.
\end{equation}

Let $w_0(t,x)=v_0(x)$ for all $t\in [0,T_0]$ and define $w_n\in \C_{T_0}^\infty$ inductively as
$w_{n+1}:=\I_{\nu,T_0,v_0}(w_n)$ for every $n\ge 0$. Tracking the proof of Lemma \ref{l4-5} we can choose
the constant $K_0$ such that $K_0>\|v_0\|_{2,p}$, which implies that
$\sup_{t\in [0,T_0]}\|w_0(t)\|_{2,p}\le K_0$. Note that, by definition, $w_n=\I_{\nu,T_0,v_0}(w_{n-1})$, so according to \eqref{t4-2-1}
and induction procedures
we obtain
\begin{equation}\label{t4-2-3}
\sup_{t\in [0,T_0]}\|w_n(t)\|_{2,p}\le K_0,\ \forall\ n\ge 0.
\end{equation}
Hence, applying \eqref{t4-2-2}, we deduce
\begin{equation}\label{t4-2-3a}
\begin{split}
&\quad \sup_{t\in [0,T_0]}\|w_n(t)-w_{n+1}(t)\|_{1,p}=
\sup_{t\in [0,T_0]}\|\I_{\nu,T_0,v_0}(w_{n-1})(t,\cdot)-\I_{\nu,T_0,v_0}(w_{n})(t,\cdot)\|_{1,p}\\
&\le \frac{1}{2}
\sup_{t\in [0,T_0]}\|w_{n-1}(t)-w_n(t)\|_{1,p},\ \forall\ n\ge 1.
\end{split}
\end{equation}
By \eqref{t4-2-3a} and the fixed point theorem, there exists a (unique) $v\in C\left([0,T_0];W_{{\rm div}}^{1,p}(\Gamma(TM))\right)$ such that
\begin{equation}\label{t4-2-4}
\lim_{n \to \infty}\sup_{t\in [0,T_0]}\|w_n(t)-v(t)\|_{1,p}=0
\end{equation}
On the other hand, from \eqref{t4-2-3} we know that $\{w_n\}_{n\ge 1}$ is weakly compact in
$C\left([0,T_0]; W^{2,p}_{{\rm div}}(\Gamma(TM))\right)$, so $v$ must be the weak limit of $w_n$ and
we have $v\in C\left([0,T_0];W^{2,p}_{{\rm div}}(\Gamma(TM))\right)$.

{\bf Step (ii)}
For every $n\ge 0$, we consider the following FBSDE,
\begin{equation}\label{t4-2-5}
\begin{cases}
&dU_{s,n}^{t,u}=\sqrt{2\nu}\sum_{i=1}^k \h_{A_i}(U_{s,n}^{t,u})\circ dB_s^i-\h_{w_n(s)}(U_{s,n}^{t,u})ds,\ 0\le t\le s \le T_0,\\
&d\tilde Y_{s,n}^{t,u}=\sum_{i=1}^k \tilde Z_{s,n}^{t,u,i}dB_s^i-\tilde F_{w_n(s)}(U_{s,n}^{t,u})ds,\\
& U_{t,n}^{t,u}=u,\ \tilde Y_{T_0,n}^{t,u}=\S(v_0)(U_{T_0,n}^{t,u}).
\end{cases}
\end{equation}
Define $\theta_{n+1}(t,x):=u\tilde Y_{t,n}^{t,u}$ for every $u\in \F(M)$ with $\pi(u)=x$. Note that
the coefficients $v_0$ and $w_n$ in equation \eqref{t4-2-5} are regular enough so that we can apply Theorem \ref{t2-2} to deduce
that $\theta_{n+1}$ satisfies the following equation (in the classical sense) on $TM$
\begin{equation}\label{t4-2-6}
\begin{split}
\theta_{n+1}(t,x)=&v_0(x)+\int_t^{T_0} \Delta \theta_{n+1}(s,x)ds-\int_t^{T_0} \nabla_{w_n(s)}\theta_{n+1}(s)(x)ds
+\int_t^{T_0} F_{w_n(s)}(x)ds.
\end{split}
\end{equation}
By the proof of \eqref{l4-5-4a} and \eqref{l4-5-6} we have
\begin{equation*}
\begin{split}
&\sup_{n\ge 1}\sup_{t\in [0,T_0]}\|\theta_n(t)\|_{2,p}\le c_1,\\
&\sup_{t\in [0,T_0]}\|\theta_{n}(t)-\theta_{n+1}(t)\|_{1,p}
\le c_1\sup_{t\in [0,T_0]}\|w_n(t)-w_{n-1}(t)\|_{1,p},\ \forall\ n\ge 1.
\end{split}
\end{equation*}
It follows from these estimates  and \eqref{t4-2-3a} that $\{\theta_n\}_{n\ge 1}$ is weakly compact in
$C\left([0,T_0]; W^{2,p}(\Gamma(TM))\right)$, as well as a Cauchy-sequence
in $C\left([0,T_0]; W^{1,p}(\Gamma(TM))\right)$. As explained above, we can find  $\theta\in C\left([0,T_0]; W^{2,p}(\Gamma(TM))\right)$
such that $\theta_n$ converges weakly to $\theta$ in $C\left([0,T_0]; W^{2,p}(\Gamma(TM))\right)$ and
\begin{equation*}
\lim_{n \to \infty}\sup_{T\in [0,T_0]}\|\theta_n(t)-\theta(t)\|_{1,p}=0.
\end{equation*}
Combining this with \eqref{t4-2-4} and letting $n \to \infty$ in \eqref{t4-2-6}
the following equation holds (in distributional sense),
\begin{equation}\label{t4-2-7}
\begin{split}
\theta(t,x)=&v_0(x)+\int_t^{T_0} \nu\Delta \theta(s,x)ds-\int_t^{T_0} \nabla_{v(s)}\theta(s)(x)ds
+\int_t^{T_0} F_{v(s)}(x)ds.
\end{split}
\end{equation}
Since $w_{n+1}(t)=\p\left(\theta_{n+1}(t)\right)$, taking $n\to \infty$ we have
$v(t)=\p\left(\theta(t)\right)$. Hence, by definition of the projection operator
$\p$, we can prove that
\begin{equation}\label{t4-2-8}
\begin{split}
\|\theta(t)-v(t)\|_{1,p}&=\|\theta(t)-\p\left(\theta(t)\right)\|_{1,p}=
\|\nabla \Delta_g^{-1}\div \theta(t)\|_{1,p}\\
&\le \|\Delta_g^{-1}\div \theta(t)\|_{2,p}\le c_2\|\div \theta(t)\|_p,\\
\end{split}
\end{equation}
where the last step is due to \eqref{l3-1-2}.

On the other hand, taking the divergence at both sides of \eqref{t4-2-7} and applying \eqref{t4-1-2}, \eqref{t4-1-3} we
obtain
\begin{equation*}
\begin{split}
\div \theta(t,x)&=\nu\int_t^{T_0} \Delta_g \div \theta(s,x)ds-\int_t^{T_0}\div\left(\nabla_{v(s)}
\left(\theta(s)-v(s)\right)\right)(x)ds\\
&+\nu\int_t^{T_0}\div \left(\text{Ric}^\sharp\left(\theta(s)-v(s)\right)\right)(x)ds\\
&=\nu\int_t^{T_0} \Delta_g \div \theta(s,x)-
\langle v(s,x), \nabla \div \theta(s,x)\rangle ds+\int_t^{T_0}\left(\Lambda_1(s,x)+\Lambda_2(s,x)\right)ds,
\end{split}
\end{equation*}

\begin{equation*}
\begin{split}
&\Lambda_1(s,x):=-\div\left(\nabla_{v(s)}
\left(\theta(s)-v(s)\right)\right)(x)+v(s)\left(\div \theta(s,\cdot)\right)(x)\\
&\quad \quad \quad \quad =-\div\left(\nabla_{v(s)}
\left(\theta(s)-v(s)\right)\right)(x)+v(s)\left(\div\left(\theta(s,\cdot)-v(s,\cdot)\right)\right)(x)\\
&\Lambda_2(s,x):=\nu\div \left(\text{Ric}^\sharp\left(\theta(s)-v(s)\right)\right)(x).
\end{split}
\end{equation*}
We have  also used the fact that $\div v(s)=0$.

This implies that
\begin{equation*}
\begin{split}
\div \theta (t,x)=\int_t^{T_0}P_{t,s}^{\nu,v}\left(\Lambda_1(s,\cdot)+\Lambda_2(s,\cdot)\right)(x)ds,\ t\in [0,T_0],
\end{split}
\end{equation*}
where $\{P_{t,s}^{\nu,v}\}_{0\le t\le s}$ denotes the (time inhomogeneous) Markovian semi-group on $M$ whose infinitesimal generator is $\nu\Delta_g-v(s,\cdot)\cdot \nabla$.
Hence we have for any $t\in [0,T_0]$,
\begin{equation*}
\begin{split}
\|\div \theta (t)\|_p& \le \int_t^{T_0}
\left(\|P_{t,s}^{\nu,v}\left(\Lambda_1(s)\right)\|_p + \|P_{t,s}^{\nu,v}\left(\Lambda_2(s)\right)\|_p\right)ds\\
&\le \int_t^{T_0}
\left(\|P_{t,s}^{\nu,v}\left(|\Lambda_1(s)|^p\right)\|_1^{1/p} + \|P_{t,s}^{\nu,v}\left(|\Lambda_2(s)|^p\right)\|_1^{1/p}\right)ds\\
&\le c_3\int_t^{T_0}
\left(\|\Lambda_1(s)\|_p+\|\Lambda_2(s)\|_p\right)ds\\
&\le c_4(1+\sup_{s\in [0,T_0]}\|v(s)\|_{2,p})\cdot \int_t^{T_0}\|\theta(s)-v(s)\|_{1,p}ds,
\end{split}
\end{equation*}
where in the third inequality we have applied \eqref{l4-1-2} and the last inequality follows from \eqref{l3-2-0a}.
Then, combining this
with \eqref{t4-2-8} yields
\begin{equation*}
\begin{split}
\|\theta(t)-v(t)\|_{1,p}\le c_5\int_t^{T_0}\|\theta(s)-v(s)\|_{1,p}ds,\ \forall\ t\in [0,T_0].
\end{split}
\end{equation*}
From Gronwall's lemma we obtain
\begin{equation*}
\|\theta(t)-v(t)\|_{1,p}=0,\ \ \forall\ t\in [0,T_0],
\end{equation*}
which implies that $\theta(t)=v(t)$ and $\div \theta(t)=0$ for all $t\in [0,T_0]$. Therefore, by \eqref{t4-2-7},
we have
\begin{equation*}
\begin{split}
v(t,x)=&v_0(x)+\int_t^{T_0} \nu\Delta v(s,x)ds-\int_t^{T_0} \nabla_{v(s)}v(s)(x)ds
+\int_t^{T_0} F_{v(s)}(x)ds,\ t\in [0,T_0],
\end{split}
\end{equation*}
which means that $v\in C\left([0,T_0];W^{2,p}_{{\rm div}}(\Gamma(TM))\right)$ is a solution of \eqref{e3-2}.

{\bf Step (iii)} According to \eqref{t4-2-3}, \eqref{t4-2-4}, \eqref{l3-4-2} and
Sobolev's embedding theorem we have
\begin{equation*}
\begin{split}
\lim_{n\to \infty}\sup_{s\in [0,T_0]}
\left(\|w_n(s)-v(s)\|_\infty+\|F_{w_n(s)}-F_{v(s)}\|_\infty\right)=0.
\end{split}
\end{equation*}
Based on these estimates 
and analogously to the proof of \eqref{l4-5-5} (as well as the standard approximation arguments from
a sequence of $\C_{T_0}^\infty$ to $v\in C\left([0,T_0];W^{2,p}_{{\rm div}}(\Gamma(TM))\right)$), we can show that,
for every $0\le t\le s\le T_0$ and $u\in \F(M)$,
\begin{equation}\label{add-ref}
\begin{split}
\lim_{n \to \infty}\E\left[\left|\tilde Y_{s,n}^{t,u}-\tilde Y_s^{t,u}\right|^2+
\sum_{i=1}^k \int_t^{T_0}\left|\tilde Z_{s,n}^{t,u,i}-\tilde Z_s^{t,u,i}\right|^2 ds\right]=0,
\end{split}
\end{equation}
where
$\Big(U_\cdot^{t,u},\tilde Y_\cdot^{t,u},\tilde Z_\cdot^{t,u}\Big)$
$\in \mathscr{C}([t,T_0];\mathscr{F}(M))\times\mathscr{C}([t,T_0];\R^{d})\times \mathscr{M}([t,T_0];\R^{dk})$
is the solution of following FBSDE,
\begin{equation*}
\begin{cases}
&dU_{s}^{t,u}=\sqrt{2\nu}\sum_{i=1}^k \h_{A_i}(U_{s}^{t,u})\circ dB_s^i-\h_{v(s)}(U_{s}^{t,u})ds,\ 0\le t\le s \le T_0,\\
&d\tilde Y_{s}^{t,u}=\sum_{i=1}^k \tilde Z_{s}^{t,u,i}dB_s^i-\tilde F_{v(s)}(U_{s}^{t,u})ds,\\
& U_{t,n}^{t,u}=u,\ \tilde Y_{T_0}^{t,u}=\S(v_0)(U_{T_0}^{t,u}).
\end{cases}
\end{equation*}
Also note that $\theta_{n+1}(x)=u\tilde Y_{t,n}^{t,u}$ for every $u\in \F(M)$ with $\pi(u)=x$.
According to the properties
\begin{equation*}
\begin{split}
& \lim_{n\to \infty}\|\theta_n(t)-v(t)\|_\infty=
\lim_{n\to \infty}\|\theta_n(t)-\theta(t)\|_\infty=0,\\
& \lim_{n \to \infty}\left|\tilde Y_{t,n}^{t,u}-\tilde Y_t^{t,u}\right|=
\lim_{n \to \infty}\E\left[\left|\tilde Y_{t,n}^{t,u}-\tilde Y_t^{t,u}\right|\right]=0
\end{split}
\end{equation*}
we deduce  that $v(t,x)=u \tilde Y_t^{t,u}$ for every $u\in \F(M)$ with $\pi(u)=x$.

Combining all  above properties we conclude that
$\Big(U_s^{t,u},\tilde Y_s^{t,u},\tilde Z_s^{t,u},v\Big)$ is a solution for \eqref{e4-3}.

{\bf Step (iv)} Now we assume that $v_0\in W^{2,p}(\Gamma(TM))\cap L_{{\rm div}}^p(\Gamma(TM))$. Then there exists a sequence
$\{v_{0,m}\}_{m=1}^\infty\subset C_b^\infty(\Gamma(TM))$ such that
\begin{equation*}
\lim_{m \to \infty}\|v_{0,m}-v_0\|_{2,p}=0.
\end{equation*}
By the analysis before we know that the constants $T_0$ and $K_0$ chosen above only depend on $\|v_0\|_{2,p}$. Hence we can find $T_0$,
$K_0>\sup_{m\ge 1}\|v_{0,m}\|_{2,p}$, which are uniform in
$m\ge 1$, such that, for every $m\ge 1$, there exists a solution $\Big(U_s^{t,u,m},\tilde Y_s^{t,u,m},\tilde Z_s^{t,u,m},v_m\Big)$ of \eqref{e4-3}
in the time interval $0\le t\le s\le T_0$ with terminal value $v_{0,m}$ which satisfies
$\Big(U_\cdot^{t,u,m},\tilde Y_\cdot^{t,u,m},\tilde Z_\cdot^{t,u,m}\Big)$
$\in \mathscr{C}([t,T_0];\mathscr{F}(M))\times\mathscr{C}([t,T_0];\R^{d})\times \mathscr{M}([t,T_0];\R^{dk})$
and $v_m\in C\big([0,T_0];$ $W_{{\rm div}}^{2,p}(\Gamma(TM))\big)$. Moreover, the following estimate holds
\begin{equation}\label{t4-2-9}
\sup_{m\ge 1}\sup_{t\in [0,T_0]}\|v_m\|_{2,p}\le K_0.
\end{equation}

By equation \eqref{e4-3} we have, for
$n,m\ge 1$ and $t\in [0,T_0]$,
\begin{align*}
&\int_M\left|\h_{A_j}\left(\tilde Y_t^{t,\cdot,n}-\tilde  Y_t^{t,\cdot,m}\right)(u(x))  \right|^p\,\mu(dx)\\
\le& c_6\Big(\int_M \E\left[\left|\h_{A_j}\big(\S\left(v_{0,n}\right)(U_{T_0}^{t,\cdot,n})\big)(u(x))-
\h_{A_j}\big(\S\left(v_{0,n}\right)(U_{T_0}^{t,\cdot,m})\big)(u(x))\right|^p\right]\mu(dx)\\
+& \int_M \E\left[\left|\h_{A_j}\big(\S\left(v_{0,n}\right)(U_{T_0}^{t,\cdot,m})\big)(u(x))-
\h_{A_j}\big(\S\left(v_{0,m}\right)(U_{T_0}^{t,\cdot,m})\big)(u(x))\right|^p\right]\mu(dx)\\
+&\quad T_0^{p-1}\int_{t}^{T_0}\int_M \E\left[\left|\h_{A_j}\big(\tilde F_{v_n(s)}(U_{s}^{t,\cdot,n})\big)(u(x))-
\h_{A_j}\big(\tilde F_{v_m(s)}(U_{s}^{t,\cdot,m})\big)(u(x))\right|^p\right]\mu(dx)ds\Big).
\end{align*}
Hence, applying \eqref{l4-2-1}, \eqref{l4-4-2}, \eqref{t4-2-9}, as well as the standard approximation
arguments, we have
\begin{equation}\label{t4-2-9a}
\begin{split}
&\int_M\left|\h_{A_j}\left(\tilde Y_t^{t,\cdot,n}-\tilde  Y_t^{t,\cdot,m}\right)(u(x))  \right|^p\,\mu(dx)\\
&\le c_{7}\Bigg(e^{c_{7}\mathbf{C}_1(p,T_0)}\left( \big(\mathbf{C}_1(p,T_0)+1\big)T^pe^{c_{7}\mathbf{C}_1(2p,T_0)}+T^pe^{c_{7}\mathbf{C}_1(p,T_0)} +e^{c_{7}\widetilde{\mathbf{C}}_1(T_0)}
\left(T_0^p+T_0^{\frac{1}{2}}+T_0^{p+\frac{1}{2}}K_0^p\right)\right)\\
&\quad \cdot \left(K_0^p
+T_0^p\bigl( \nu^pK_0^p+K_0^{2p} \bigr)\right)
+\bigl(e^{c_7\mathbf{C}_1(p,T_0)} + e^{c_7\widetilde{\mathbf{C}}_1(T_0)}\bigr)T_0^p
\cdot\bigl(K_0^p+\nu^p\bigr)\Bigg)\cdot \sup_{s\in [0,T_0]}\|w_1(s)-w_2(s)\|^p_{1,p}\\
&+c_7\bigl( e^{c_7\mathbf{C}_1(p,T_0)} + e^{c_7\widetilde{\mathbf{C}}_1(T_0)}\bigr)\|v_{0,n}-v_{0,m}\|^p_{1,p},
\end{split}
\end{equation}
where $\CC_1(T)$, $\mathbf{C}_1(p,T)$ are defined as the same way of \eqref{l4-2-0} with $K_T(w)$ replaced by $K_0$.

Also note that $\S(v_n)(u)=\tilde Y_t^{t,u,n}$. Applying  similar arguments to those above for
\[
    \int_M \left|\left(\tilde Y_t^{t,\cdot,n}-\tilde  Y_t^{t,\cdot,m}\right)(u(x)) \right|^p\mu(dx)
\]
 and choosing $T_0$ small enough if necessary (by the expression \eqref{t4-2-9a} and the fact that $\lim_{T \to 0}(\CC_1(p,T)+\mathbf{C}_1(T))=0$) we will obtain
\begin{equation*}
    \sup_{t\in [0,T_0]}\|v_n(t)-v_m\|^p_{1,p} \le \frac{1}{2}\sup_{t\in [0,T_0]}\|v_n(t)-v_m(t)\|^p_{1,p}
    +c_8\|v_{0,n}-v_{0,m}\|_{1,p}^p.
    \end{equation*}
Therefore we have
\begin{align*}
\sup_{t\in [0,T_0]}\|v_n(t)-v_m(t)\|_{1,p} \le c_{9}\|v_{0,n}-v_{0,m}\|_{1,p}.
\end{align*}
So $\{v_m\}_{m\ge 1}$ is a Cauchy sequence in $C\left([0,T_0]; W^{1,p}(\Gamma(TM))\right)$ and
there exists  $v\in C\left([0,T_0]; W^{1,p}(\Gamma(TM))\right)$ such that
\begin{equation}\label{t4-2-10}
\lim_{m \to \infty}\sup_{t\in [0,T_0]}\|v_m(t)-v(t)\|_{1,p}=0.
\end{equation}
On the other hand, by \eqref{t4-2-9} we know that $\{v_m\}_{m\ge 1}$ is weakly compact in $C\left([0,T_0];W^{2,p}(\Gamma(TM))\right)$ and
$v$ is the weak limit of $\{v_m\}_{m\ge 1}$, so we have $v\in C\left([0,T_0];W^{2,p}_{{\rm div}}(\Gamma(TM))\right)$.
Moreover, by the conclusion of {\bf Step (iii)},  $v_m$ is the solution of \eqref{e3-2} with terminal value
$v_m(T,x)=v_{0,m}(x)$. Letting $m\to \infty$ in \eqref{e3-2} and applying \eqref{t4-2-10} we deduce that
$v\in C\left([0,T_0];W^{2,p}_{{\rm div}}(\Gamma(TM))\right)$ is a solution of \eqref{e3-2}.  Still according to \eqref{t4-2-10}, taking
$m\to \infty$ in \eqref{e4-3} with terminal value $v_{0,m}$ and using
the same arguments as for the proof of \eqref{t4-2-9a},  we conclude that there exists $\left(U_\cdot^{t,u},\tilde Y_\cdot^{t,u},\tilde Z_\cdot^{t,u}\right)\in $  $\mathscr{C}([t,T_0];\mathscr{F}(M))\times\mathscr{C}([t,T_0];\R^{d})\times \mathscr{M}([t,T_0];\R^{dk})$
(for all $0\le t\le T_0$) such that $\left(U_s^{t,u},\tilde Y_s^{t,u},\tilde Z_s^{t,u},v\right)$ is a solution to \eqref{e4-3}.

On the other hand, suppose that $\Big(\hat U_s^{t,u},\hat Y_s^{t,u},\hat Z_s^{t,u},\hat v\Big)$  is also a solution of
\eqref{e4-3} satisfying
$\left(\hat U_\cdot^{t,u},\hat Y_\cdot^{t,u},\hat Z_\cdot^{t,u}\right)\in $ $\mathscr{C}([t,T_0];\mathscr{F}(M))\times\mathscr{C}([t,T_0];\R^{d})\times \mathscr{M}([t,T_0];\R^{dk})$
(for all $0\le t\le T_0$) and
$\hat v\in C\left([0,T_0];W_{{\rm div}}^{2,p}\left(\Gamma(TM)\right)\right)$. Since
$\S(\hat v)(u)=\hat Y_t^{t,u}$ and $\S(v)(u)=\tilde Y_t^{t,u}$,
by \eqref{e4-3}, using the same arguments as in the proof of \eqref{t4-2-9a} (and choosing $T_0$ smaller if necessary) we
can show that
\begin{align*}
\sup_{t\in [0,T_0]}\|\hat v(t)-v(t)\|_{1,p}\le \frac{1}{2}\sup_{t\in [0,T_0]}\|\hat v(t)-v(t)\|_{1,p}.
\end{align*}
This implies  that $\sup_{t\in [0,T_0]}\|\hat v(t)-v(t)\|_{1,p}=0$, so we have
$v(t)=\hat v(t)$ for all $t\in [0,T_0]$.
In particular we remark that, although
the estimate \eqref{t4-2-9a} is proved under the condition  $v_n,v_m\in \C_{T_0}^\infty$, by the standard approximation
procedure it still holds under the weaker condition $v_n,v_m\in C\left([0,T_0];W_{{\rm div}}^{2,p}(\Gamma(TM))\right)$
(note that \eqref{l4-4-7} only depends on $K_T(w_1,w_2)$ and $\sup_{s\in [0,T]}\|w_1(s)-w_2(s)\|_{1,p}$).

Observe that, with $w\in C\left([0,T_0];W_{{\rm div}}^{2,p}(\Gamma(TM))\right)$ and $v_0\in W^{2,p}(\Gamma(TM))\cap L^p_{{\rm div}}(\Gamma(TM))$
fixed, the solution for the (de-coupled) FBSDE \eqref{e4-4} is unique. Since $v(t)=\hat v(t)$ and from
the expression \eqref{e4-3} we deduce that $\left(U_s^{t,u},\tilde  Y_s^{t,u},\tilde  Z_s^{t,u}\right)$=$\left(\hat U_s^{t,u},\hat Y_s^{t,u},\hat Z_s^{t,u}\right)$.
So we have verified that  the solution $\left(U_s^{t,u},\tilde  Y_s^{t,u},\tilde  Z_s^{t,u},v \right)$ of \eqref{e4-3} is unique.

\end{proof}

\noindent {\bf Acknowledgements.}\,\,  The research of Xin Chen is supported by National Natural Science Foundation of China
(No. 12122111), Ana Bela Cruzeiro acknowledges Portuguese FCT project UIDB/00208/2020, Qi Zhang is supported by National Key R\&D Program of China (No. 2022YFA1006101), National Natural Science Foundation of China (No. 11871163) and the Science and Technology Commission of Shanghai Municipality (No. 22ZR1407600), Wenjie Ye is supported by National Key R\&D Program of China (No. 2020YFA0712700), the NSFC (No.
11931004, 12090014, 12288201) and key Lab of Random Complex Structures and Data Science, Youth Innovation Promotion Association (2020003), Chinese Academy of Science.



\begin{thebibliography}{99}

\leftskip=-2mm
\parskip=-1mm




\bibitem{AC} M. Arnaudon, Marc, A.B. Cruzeiro,
Lagrangian Navier-Stokes diffusions on manifolds: variational principle and stability,
\emph{Bull. Sci. Math.} \textbf{136} (2012), 857--881.


\bibitem{ACF} M. Arnaudon, A.B. Cruzeiro and S.Z. Fang,
Generalized stochastic Lagrangian paths for the Navier-Stokes equation,
\emph{Ann. Sc. Norm. Super. Pisa Cl. Sci.} \textbf{18} (2018), 1033--1060.


\bibitem{Bi} J.-M. Bismut, Linear quadratic optimal stochastic control with random coeffcients, \emph{SIAM J. Control Optim.}
\textbf{14} 419--444, (1976).

\bibitem{B1} F. Blache, Backward stochastic differential equations on manifolds, \emph{Probab. Theory
Related Fields}, \textbf{132} (2005), 391--437.

\bibitem{B2} F. Blache, Backward stochastic differential equations on manifolds II, \emph{Probab. Theory
Related Fields}, \textbf{136} (2006), 234--262.





\bibitem{CC} X. Chen and A.B. Cruzeiro, Stochastic geodesics and forward-backward stochastic differential equations on Lie groups,
\emph{Discrete Contin. Dyn. Syst.} 2013, Dynamical systems, differential equations and applications. 9th AIMS Conference. Suppl, 115--121.

\bibitem{CCQ} X. Chen, A.B. Cruzeiro and Z.M. Qian, Navier-Stokes equation and forward-backward stochastic differential system in the Besov spaces, arXiv:1305.0647.

\bibitem{CY} X. Chen and W. Ye, A probabilistic representation for heat flow of harmonic map on manifolds with time-dependent Riemannian metric,
\emph{Statist. Probab. Lett.} \textbf{177} No. 109165, 10 pp, (2021).

\bibitem{CY1} X. Chen and W. Ye, A study of backward stochastic differential equation on a Riemannian
manifold, \emph{Electron. J. Probab.} \textbf{26} (2021), 1--31,




\bibitem{CI}  P. Constantin and G. Iyer,
A stochastic Lagrangian representation of the three-dimensional incompressible Navier-Stokes equations,
\emph{Comm. Pure Appl. Math.} \textbf{61} (2008), 330--345.

\bibitem{CQ} A.B. Cruzeiro and Z.M. Qian, Backward stochastic differential equations associated with the vorticity equations,
 \emph{J. Funct. Anal.} \textbf{267} (2014), 660--677.

\bibitem{CS}A.B. Cruzeiro and E. Shamarova, Navier-Stokes equations and forward-backward SDEs on the group of diffeomorphisms of the torus,
\emph{Stoch. Proc. and their Applic.}  \textbf{119} (2009),  4034-4060








\bibitem{Da1} R. W. R. Darling, Constructing gamma-martingales with prescribed limit, using backwards SDE.
\emph{Ann. Probab.} \textbf{23} 1234--1261, (1995).



\bibitem{DMT} N. Down, S.P. Meyn, R.L. Tweedie, Exponential and uniform ergodicity of Markov processes, \emph{Ann. Probab.} \textbf{23}
(1995), 1671--1691.




\bibitem{E1} K.D. Elworthy, \emph{Stochastic differential equations on manifolds},
London Mathematical Society Lecture Note Series, \textbf{70}, Cambridge University Press, Cambridge--New York, (1982).


\bibitem{ELL} K.D. Elworthy, Y. Le Jan and X.M. Li,
Concerning the geometry of stochastic differential equations and stochastic flows,
\emph{New trends in stochastic analysis} (Charingworth, 1994), 107--131, World Sci. Publ., River Edge, NJ, (1997).

\bibitem{EL}  K.D. Elworthy, Y. LeJan and X.M. Li, \emph{On the geometry of
    diffusion operators and stochastic flows}, \textbf{1720} of Lecture Notes in
    Mathematics, Springer-Verlag, Berlin, (1999).

\bibitem{E} M. Emery, \emph{Stochastic calculus in manifolds}, Universitext Springer-Verlag, New
York/Berlin, (1989).

\bibitem{EP} A. Estrade and M. Pontier, Backward stochastic differential equations in a Lie group,
\emph{S\'eminaire de Probabilit\'es}, XXXV, 241--259, Lecture Notes in Math. 1755, Springer, Berlin, 2001.

\bibitem{EV} L.C. Evans, \emph{Partial differential equations},
Second edition. Graduate Studies in Mathematics, \textbf{19} American Mathematical Society, Providence, RI, (2010).

\bibitem{F} S.Z. Fang, Nash Embedding, Shape Operator and Navier-Stokes Equation on a Riemannian Manifold,
\emph{Acta Math. Appl. Sin. Engl. Ser.} \textbf{36} (2020), 237--252.

\bibitem{FL} S.Z. Fang and D.J. Luo, Constantin and Iyer's representation formula for the Navier-Stokes equations on manifolds,
\emph{Potential Anal.} \textbf{48} (2018), 181--206.

\bibitem{FWZ} C.R. Feng, X.C. Wang and H.Z. Zhao,
Quasi-linear PDEs and forward--backward stochastic differential equations: weak solutions,
\emph{J. Differential Equations} \textbf{264} (2018), 959--1018.

\bibitem{FQT} D. Freddy, J.N. Qiu and S.J. Tang,
Forward-backward stochastic differential systems associated to Navier-Stokes equations in the whole space,
\emph{Stochastic Process. Appl.} \textbf{125} (2015), 2516--2561.

\bibitem{GT}  D. Gilbarg and N.S. Trudinger, \emph{Elliptic partial differential equations of second order}, Classics in Mathematics, Springer-Verlag, Berlin, (2001).




\bibitem{H} E.P. Hsu, \emph{Stochastic analysis on manifolds}, vol 38, Graduate Studies in Mathematics.
American Mathematical Society, (2002).







\bibitem{Le}  D.S. Ledesma,
A local solution to the Navier-Stokes equations on manifolds via stochastic representation,
\emph{Nonlinear Anal.} \textbf{198} (2020), 111927, 9 pp.

\bibitem{L} X.-M. Li, Stochastic differential equations on noncompact manifolds: moment stability and its topological consequences,
\emph{Probab. Theory Related Fields} \textbf{100} (1994), 417--428.








\bibitem{PP1} E. Pardoux and S. Peng, Adapted solution of a backward stochastic differential equation,
\emph{Syst. Control Lett.} \textbf{14} (1990), 55--61.

\bibitem{PP2} E. Pardoux and S. Peng, \emph{Backward stochastic differential equations and
quasilinear parabolic partial differential equations,}
Stochastic Partial Differential Equations and Their Applications,
Lecture Notes in Control and Information Sciences  176, 200--217,
Springer, New York, (1992).

\bibitem{PT}  E. Pardoux and S.J. Tang,
Forward-backward stochastic differential equations and quasilinear parabolic PDEs,
\emph{Probab. Theory Related Fields} \textbf{114} (1999), 123--150.

\bibitem{Pe} S. Peng, Probabilistic interpretation for systems of quasilinear parabolic partial differential equations,
\emph{Stochastics Stochastics Rep.} \textbf{37} (1991), 61--74.

\bibitem{PW}  S. Peng and Z. Wu,
Fully coupled forward-backward stochastic differential equations and applications to optimal control,
\emph{SIAM J. Control Optim.} \textbf{37} (1999), 825--843.

\bibitem{P} V. Pierfelice,
The incompressible Navier-Stokes equations on non-compact manifolds, \emph{J. Geom. Anal.} \textbf{27} (2017), 577--617.




\bibitem{RY} D. Revuz and M. Yor, \emph{Continuous Martingales and Brownian Motion}, 3rd edition,
Springer-Verlag, (1999).









\bibitem{V} C. Villani, \emph{Optimal transport. Old and new.}
Grundlehren der Mathematischen Wissenschaften, \textbf{338} Springer-Verlag, Berlin, (2009).

\bibitem{W}
F.Y. Wang, \emph{Functional Inequalities, Markov Processes and
Spectral Theory}, Science Press, Beijing (2005).

\bibitem{XZ}  H. Xing and G. Zitkovi\'c,
A class of globally solvable Markovian quadratic BSDE systems and applications, \emph{Ann. Probab.} \textbf{46} 491--550, (2018).


\bibitem{WY} Z. Wu and Z.Y. Yu,
Probabilistic interpretation for a system of quasilinear parabolic partial differential equation combined with algebra equations,
\emph{Stochastic Process. Appl.} \textbf{124} (2014), 3921--3947.


\bibitem{ZZ} Q. Zhang and H.Z. Zhao,
Probabilistic representation of weak solutions of partial differential equations with polynomial growth coefficients,
\emph{J. Theoret. Probab.} \textbf{25} (2012), 396--423.

\bibitem{Z1} X.C. Zhang, A stochastic representation for backward incompressible Navier-Stokes equations,
\emph{Probab. Theory Related Fields} \textbf{148} (2010), 305--332.


\bibitem{Z} X.C. Zhang, Quasi-invariant stochastic flows of SDEs with non-smooth drifts on compact manifolds,
\emph{Stochastic Process. Appl.} \textbf{121} (2011), 1373--1388.

\end{thebibliography}
\end{document}